\documentclass{amsart}
\usepackage[utf8]{inputenc}

\usepackage[foot]{amsaddr}

\usepackage{biblatex}
\usepackage{mathtools}
\usepackage{amsmath}
\usepackage{amssymb}
\usepackage{amsthm}
\usepackage{stmaryrd}
\usepackage{hyperref}
\usepackage{cleveref}
\usepackage{listings}
\usepackage{enumitem}

\usepackage[hyphenbreaks]{breakurl}

\usepackage{xcolor}
\usepackage{dutchcal}
\usepackage{tikz-cd}
\usepackage{tikz}
\usetikzlibrary{calc,decorations.markings,decorations.pathmorphing}
\usepackage{ifthen}
\usepackage{bm}

\pgfdeclarelayer{background}
\pgfsetlayers{background,main}

\theoremstyle{definition}
\newtheorem{Def}[]{Definition}[section]
\newtheorem{Pro}[Def]{Proposition}
\newtheorem{Rem}[Def]{Remark}
\newtheorem{Lem}[Def]{Lemma}

\newtheorem{Theo}[Def]{Theorem}
\newtheorem{Cor}[Def]{Corollary}

\newtheorem{Con}[Def]{Construction}
\newtheorem{Conv}[Def]{Convention}
\newtheorem*{Ques*}{Question}
\newtheorem*{Theo*}{Theorem}

\DeclareMathOperator{\Ext}{Ext}
\DeclareMathOperator{\Hom}{Hom}
\DeclareMathOperator{\RHom}{\mathbb{R}\strut\kern-.2em\operatorname{Hom}}

\DeclareMathOperator{\SL}{SL}
\DeclareMathOperator{\GL}{GL}
\DeclareMathOperator{\PGL}{PGL}

\DeclareMathOperator{\supp}{supp}

\DeclareMathOperator{\Db}{\mathrm{D}^{\mathrm{b}}}

\DeclareMathOperator{\Irr}{Irr}

\DeclareMathOperator{\PSL}{PSL}

\DeclareMathOperator{\Sym}{Sym}

\DeclareMathOperator{\diag}{diag}

\newcommand{\bbone}{\text{\usefont{U}{bbold}{m}{n}1}}
\MakeRobust{\bbone}

\renewcommand{\mod}{\operatorname{mod}}
\renewcommand{\Bbbk}{\mathbb{C}}
\renewcommand{\Im}{\operatorname{Im}}

\title[$2$-RI algebras from non-abelian subgroups of $\SL_3$. Part II]{$2$-representation infinite algebras from non-abelian subgroups of $\SL_3$ \\[0.8em]\smaller{}\smaller{} Part II: Central extensions and exceptionals}

\author[Darius Dramburg]{Darius Dramburg}
\address{Darius Dramburg, Department of Mathematics, Uppsala University, Box 480, 75106 Uppsala, Sweden}

\email{darius.dramburg@math.uu.se}

\date{\today}

\begin{document}

\begin{abstract}
    Let $G \leq \SL_3(\mathbb{C})$ be a non-trivial finite group, acting on $R = \mathbb{C}[x_1, x_2, x_3]$. We continue our investigation from \cite{DramburgGasanova3} into when the resulting skew-group algebra $R \ast G$ is a $3$-preprojective algebra of a $2$-representation infinite algebra, defined by a so-called cut. We consider the subgroups arising from $\GL_2(\Bbbk) \hookrightarrow \SL_3(\Bbbk)$, called type (B), as well as the exceptional subgroups, called types (E) -- (L). For groups of type (B), we show that a $3$-preprojective cut exists on $R \ast G$ if and only if $G$ is not isomorphic to a subgroup of $\SL_2(\Bbbk)$ or $\PSL_2(\Bbbk)$. For groups $G$ of the remaining types (E) -- (L), every $R \ast G$ admits a $3$-preprojective cut, except for type (H) and (I). 
    To prove our results for type (B), we explore how the notion of isoclinism interacts with the shape of McKay quivers. We compute the McKay quivers in detail, using a knitting-style heuristic. 
    For the exceptional subgroups, we compute the McKay quivers directly, as well as cuts, and we discuss how this task can be done algorithmically.
    This provides many new examples of $2$-representation infinite algebras, and together with \cite{DramburgGasanova, DramburgGasanova3} completes the classification of finite subgroups of $\SL_3(\Bbbk)$ for which $R \ast G$ is a $3$-preprojective algebra.
\end{abstract}

\maketitle

\section{Introduction}
This article is a continuation of the investigation of $3$-preprojective structures on skewed polynomial rings $\mathbb{C}[x_1, x_2, x_3] \ast G$, where $G \leq \SL_3(\mathbb{C})$ is a (non-trivial) finite subgroup. This was initiated in \cite{DramburgGasanova} for abelian subgroups, and continued in \cite{DramburgGasanova3} for imprimitive subgroups. In this article, we consider the remaining finite subgroups of $\SL_3(\Bbbk)$ and determine whether $\mathbb{C}[x_1, x_2, x_3] \ast G$ admits a $3$-preprojective structure. These structures arise from grading the McKay quiver of $G$ and are called \emph{cuts}. We are interested in this setup due to its significance to higher Auslander-Reiten theory and the McKay correspondence. 

The finite subgroups of $\SL_3(\Bbbk)$ are classified according to types (A) to (L) \cite{YauYu}. In \cite{DramburgGasanova}, we classified the abelian subgroups $G \leq \SL_3(\Bbbk)$ for which $R \ast G$ admits a $3$-preprojective cut. In terms of the classification of subgroups of $\SL_3(\Bbbk)$, the abelian groups form the type (A). The types (C) and (D) were dealt with in the companion to this article \cite{DramburgGasanova3}. Thus, the remaining types are (B) and (E) to (L). They are the subject of this article. 

To understand why we grouped the types (B) and (E) to (L) together, we first briefly recall the classification of finite subgroups of $\SL_3(\Bbbk)$. More details are provided in \Cref{SSec: Finite sub of SL3}. Let $G \leq \SL_3(\Bbbk)$ be finite, and consider its natural representation on $V = \Bbbk^3$. If $V$ decomposes into $3$ irreducible representations, then $G$ is abelian, which gives rise to type (A). If $V$ decomposes, but one of the irreducible summands has dimension $2$, then this summand is itself a faithful representation of $G$, so we obtain $G$ as a finite subgroup of $\GL_2(\Bbbk)$, embedded into $\SL_3(\Bbbk)$ in the obvious way. This is called type (B). Next, we consider the case when $V$ is irreducible. To further distinguish cases, we ask whether $V$ is primitive or imprimitive. In the imprimitive case, the stabilizer of a block in $V$ can be trivial or $C_2$, giving rise to types (C) and (D) respectively. Finally, if $V$ is primitive, there are finitely many possibilities for $G$, giving rise to the exceptional subgroups of $\SL_3(\Bbbk)$. These are given the types (E) to (L). 

One can see, as for example in \cite{DramburgGasanova3}, that $G$ of type (C) or (D) is a semi-direct product $G \simeq N \rtimes H$ where $N$ is of type (A). As such, they are amenable to being studied via $N$, to which methods from \cite{DramburgGasanova} apply. 

A group $G$ of type (B), in contrast, usually does not admit a helpful semi-direct decomposition. Instead, it is a central extension of its image $H \simeq G/Z(G)$ in $\PGL_2(\Bbbk)$, so different techniques are required to compute the McKay quivers and to produce $3$-preprojective cuts. 

It turns out that a helpful notion for type (B) is that of \emph{isoclinism}. Denote by $BH \leq \SL_2(\Bbbk)$ the binary version of the image $H$ of $G$ in $\PGL_2(\Bbbk)$. We show that $G$ and $BH$ are isoclinic, and use this to show that the McKay quiver for $G$ as a subgroup of $\GL_2(\Bbbk)$ looks like a ``repeated version'' of the McKay quiver for $BH \leq \SL_2(\Bbbk)$. Once the McKay quiver $Q_G$ is determined, we use folding arguments to produce a quotient of $Q_G$ that is the McKay quiver of a group of type (A), and use our methods from \cite{DramburgGasanova} to produce a $3$-preprojective cut for the quotient. We then lift this back to $Q_G$. This procedure is possible as long as the defining representation of $G $ is not self-dual, and $G$ is not a binary group. Thus, our first main theorem can be phrased as follows. 

\begin{Theo*}[\Cref{Theo: Type (B) classification}]
   Let $G \leq \SL_3(\Bbbk)$ be a group of type (B). Then $R \ast G \simeq_M \Bbbk Q_G/I$ admits a $3$-preprojective cut if and only if the defining representation $\rho \colon G \hookrightarrow \GL_2(\Bbbk)$ is not self-dual and $G$ is not isomorphic to a subgroup of $\SL_2(\Bbbk)$.
\end{Theo*}

We point out here that in the above theorem, we really mean abstract isomorphism, rather than that the embedding $G \to \SL_3(\Bbbk)$ does not factor through $\SL_2(\Bbbk)$ or $\PSL_2(\Bbbk)$. This is due to the binary tetrahedral group having a faithful representation of dimension $2$ whose image is not in $\SL_2(\Bbbk)$, so not factoring through $\SL_2(\Bbbk)$ is not enough. 

Our computations of the McKay quiver follow a heuristic that strongly resembles the knitting-algorithm in Auslander-Reiten theory. For small subgroups of $\GL_2(\Bbbk)$, the computations were already done by Auslander and Reiten in \cite{AuslanderReitenMcKay}, see also \cite{IyamaWemyssSpecial}, as well as \cite{NolladeCelisDihedral} for the case of small dihedral groups and \cite{NdCS} for finite subgroups of $\operatorname{SO}_3(\Bbbk)$. However, for our purposes we also need to consider the non-small subgroups of $\GL_2(\Bbbk)$, since they become small upon embedding them into $\SL_3(\Bbbk)$. It would be interesting to find a category whose Auslander-Reiten quiver is the McKay quiver of the non-small group $G$ with respect to its embedding into $\GL_2(\Bbbk)$. This could give an interpretation of our main heuristic in terms of the Serre functor. 

For the remaining types (E) to (L), we go through each case by hand. We compute the adjacency matrix of the McKay quiver $Q_G$, draw the quiver, and provide a $3$-preprojective cut on $R \ast G \simeq \Bbbk Q_G/I$ when it exists. We find the following. 

\begin{Theo*}[\Cref{Theo: exceptional class}]
    Let $G \leq \SL_3(\Bbbk)$ be an exceptional subgroup. Then $R \ast G \simeq_M \Bbbk Q_G/I$ admits a $3$-preprojective cut if an only if $G$ is not of type (H) or type (I).  
\end{Theo*}

Since these results require some computations, we also discuss how the McKay quiver can be computed with standard methods in GAP. We explain that the problem of finding a $3$-preprojective cut on $Q_G$ is an instance of an exact cover problem. This is a classical NP-complete problem, but our instances are small enough to be solved algorithmically. However, we provide explicit quivers and cuts in all cases which can and have been verified by hand. For the readers convenience, we indicate both how to check the cuts by hand, and how one would solve the problem on the computer.     

We conclude with combining the relevant results from \cite{DramburgGasanova, DramburgGasanova3} and the current paper. This provides the classification of finite subgroups $G \leq \SL_3(\Bbbk)$ for which $R \ast G$ admits a $3$-preprojective cut. The classification can be stated as follows, using the types as established in \cite{YauYu}. 

\begin{Theo*}[\Cref{Theo: All classification}]
    Let $G \leq \SL_3(\Bbbk)$ be a finite group. Then $R \ast G \simeq_M \Bbbk Q_G/I$ admits a $3$-preprojective cut if and only if the group $G$ satisfies one of the following conditions.  
    \begin{itemize}
        \item $G$ is of type (A), $G \not \simeq C_2 \times C_2$, and its embedding does not factor through $\SL_2(\Bbbk)$. 
        \item $G$ is of type (B), its defining representation is not self-dual, and it is not isomorphic to a finite subgroup of $\SL_2(\Bbbk)$.
        \item $G$ is of type (C) or (D) and has order divisible by $9$.
        \item $G$ is an exceptional subgroup and not of type (H) or type (I). 
    \end{itemize}
\end{Theo*}

We mention as well that in almost all cases when $R \ast G \simeq_M \Bbbk Q_G/I$ does not admit a $3$-preprojective cut, this is witnessed by a loop in the McKay quiver. However, there is \emph{one} case where the quiver has no loops, yet $R \ast G \simeq_M \Bbbk Q_G/I$ fails to admit a $3$-preprojective cut. A similar phrasing about lengths of shortest cycles in $Q_G$ is not possible either as witnessed by abelian groups for which $R \ast G \simeq_M \Bbbk Q_G/I$ admits a $3$-preprojective cut but the quiver $Q_G$ contains $2$-cycles. It therefore seems to be an artifact of dimension $3$ that the existence of $3$-preprojective cuts often can be seen from the length of cycles in $Q_G$. 

\subsection{Outline}
We first collect the necessary background material in \Cref{Sec: Preliminaries}. In particular, we give a more detailed summary of the classification of finite subgroups of $\SL_3(\Bbbk)$. 

We summarise our results on abelian groups from \cite{DramburgGasanova} in \Cref{Sec: type (A)}, since they will be needed for type (B).

In \Cref{Sec: Type (B)}, we consider the groups $G$ of type (B). These arise from embedding $\GL_2(\mathbb{C}) \to \SL_3(\mathbb{C})$, so we need to take into account all finite subgroups of $\GL_2(\mathbb{C})$, not just the small ones. We compute the McKay quivers of these groups, and explain their shape by using the isoclinism of $G$ with a finite subgroup of $\SL_2(\mathbb{C})$. In particular, we obtain that $R \ast G$ admits a $3$-preprojective cut if and only if $G$ is not isomorphic to a subgroup of $\SL_2(\mathbb{C})$ or $\PSL_2(\Bbbk)$.

The remaining types, labeled (E) to (L), are finitely many ``exceptional subgroups''. We deal with these in \Cref{Sec: Except}, where we provide for each group a structure description, the adjacency matrix of the McKay quiver, a drawing of the quiver and a $3$-preprojective cut on the quiver if it exists. We also discuss the algorithmic aspects of computing McKay quivers and higher preprojective cuts. 

We conclude this article with \Cref{Sec: Classification}, where we summarise the classification of $3$-preprojective skew-group algebras $R \ast G$ that is completed with this article. 

\section{Preliminaries}\label{Sec: Preliminaries}
We work over the field $\mathbb{C}$ throughout, but the same results hold over any algebraically closed field of characteristic $0$. The reader may find it convenient to consult \cite{DramburgGasanova3} together with this article. Much of the preliminaries is the same, since \cite{DramburgGasanova3} and this paper are two parts of the overarching project of classifying $3$-preproejctive skew-group algebras. We use the following notation and conventions. 

\begin{enumerate}
    \item All undecorated tensors $\otimes$ are taken over $\Bbbk$.
    \item We denote isomorphism by $\simeq$ and Morita equivalence by $\simeq_M$.
    \item A quiver $Q$ consists of vertices $Q_0$, arrows $Q_1$ and source and target maps $s, t \colon Q_1 \to Q_0$. 
    \item We write the difference of sets $A$ and $B$ as $A - B = \{ a \in A \mid a \not \in B\}$. For a quiver $Q$ and a set $C \subseteq Q_1$, we write $Q -C = (Q_0, Q_1 - C)$. 
    \item For a finite group $G$, we denote by $Z(G)$ its center and by $G' = [G,G]$ its derived subgroup.   
    \item We denote by $\epsilon_n = \exp{(\frac{2 \pi i}{n})}$ a primitive $n$-th root of unity.  
\end{enumerate}

\subsection{Higher representation infinite algebras}
We begin by recalling the parts of higher Auslander-Reiten theory (henceforth higher AR-theory) as they pertain to this article. This is the same summary as can be found in \cite{DramburgGasanova3}. Let $\Lambda$ be a finite dimensional $\Bbbk$-algebra of finite global dimension $d \geq 1$, and denote by $D = \operatorname{Hom}_\Bbbk(-,\Bbbk)$ the standard vector space duality. The following functors play a central role in higher AR-theory:
\[ \nu = D\RHom_\Lambda(\--,\Lambda) \colon \Db( \mod \Lambda) \to \Db( \mod \Lambda)  \] 
is the \emph{derived Nakayama functor} on the bounded derived category of finitely generated $\Lambda$-modules, and its quasi-inverse is
\[\nu^{-1} = \RHom_{\Lambda^{\mathrm{op}}}(D(\--), \Lambda) \colon \Db( \mod \Lambda) \to \Db( \mod \Lambda).\]
The \emph{derived higher Auslander-Reiten translation} is the autoequivalence 
\[\nu_d := \nu \circ [-d] \colon \Db(\mod \Lambda) \to \Db( \mod \Lambda).\]

\begin{Def}\cite[Definition 2.7]{HIO} 
The algebra $\Lambda$ is called \emph{$d$-representation infinite} if for any projective module $P$ in $ \mod \Lambda$ and integer $i \geq 0$ we have 
\[ \nu_d^{-i} P \in   \mod \Lambda. \]
\end{Def}

As alluded to in the introduction, we are interested in these algebras because their \emph{higher preprojective algebras} are bimodule Calabi-Yau.

\begin{Def}\cite[Definition 2.11]{IyamaOppermannStable}
Let $\Lambda$ be of global dimension at most $d$. The $(d+1)$-preprojective algebra of $\Lambda$ is 
\[ \Pi_{d+1}(\Lambda) = \operatorname{T}_\Lambda \Ext^{d}_\Lambda(D(\Lambda), \Lambda). \]
\end{Def}

Crucially, $\Pi_{d+1}(\Lambda)$ graded by tensor degrees, and $\Lambda$ is the degree $0$ part of this grading. Therefore, when we write the term ``higher preprojective alegbra'', we always mean a graded algebra. The properties of this grading, together with being Calabi-Yau, essentially determine the preprojective algebras of higher representation infinite algebras. 

\begin{Def}\cite[Definition 3.1]{AIR} \label{Def: n-CY GP a}
Let $\Gamma=\bigoplus_{i \geq 0}\Gamma_i$ be a positively graded $\Bbbk$-algebra. We call $\Gamma$ a \emph{bimodule $(d+1)$-Calabi-Yau algebra of Gorenstein parameter $a$} if there exists a bounded graded projective $\Gamma$-bimodule resolution
$P_\bullet$ of $\Gamma$ and an isomorphism of complexes of graded $\Gamma$-bimodules
\[ P_\bullet \simeq \Hom_{\Gamma^{\mathrm{e}}}(P_\bullet, \Gamma^{\mathrm{e}})[d+1](-a).  \]
\end{Def}

The same definition for ungraded algebras and bimodules produces that of a bimodule $(d+1)$-Calabi-Yau algebra. One can introduce different gradings on $\Gamma$ so that the projective bimodule resolution becomes graded. We are interested in those gradings where the Goenstein parameter becomes $a=1$.

\begin{Theo}\cite[Theorem 4.35]{HIO}\label{Theo: HPG is f.d. GP1}
There is a bijection between $d$-representation infinite algebras $\Lambda$ and graded bimodule $(d+1)$-Calabi-Yau algebras $\Gamma$ of Gorenstein parameter 1 with $\dim_\Bbbk \Gamma_i < \infty $ for all $i \in \mathbb{N}$, both sides taken up to isomorphism. The bijection is given by 
\[ \Lambda \mapsto \Pi_{d+1}(\Lambda) \quad \mathrm{ and } \quad \Gamma \mapsto \Gamma_0.  \]
\end{Theo}

Hence we take the perspective that admitting a higher preprojective grading is a property of a Calabi-Yau algebra. The question whether a given Calabi-Yau algebra has this property was raised in \cite{Thibault}. As we shall see, a prototypical example of bimodule Calabi-Yau algebras are skew-group algebras, and our goal is to answer this question completely for skew-group algebras $\Bbbk[x_1, x_2, x_3] \ast G$ where $G \leq\SL_3(\Bbbk)$ is finite. We are not aware of general criterion answering the question for subgroups $G \leq \SL_{d+1}(\Bbbk)$. However, large classes of examples are known, such as those coming from certain cyclic or metacyclic groups \cite{AIR, Giovannini}, and classes of counterexamples coming from direct product decompositions \cite{Thibault}. We also mention the classification of the abelian groups $G \leq \SL_{d+1}(\Bbbk)$ for which $R \ast G$ admits such a grading, which can be found in \cite{DramburgGasanova2}.

\subsection{Skew-group algebras and cuts}
Let $G$ be a finite group acting (from the left) on the $\Bbbk$-algebra $R$ via algebra-automorphisms. Then the \emph{skew-group algebra} of $R$ by $G$ is the vector space 
\[ R \ast G = R \otimes \Bbbk G, \]
with multiplication induced from 
\[ (r \otimes g) (s \otimes h) = r g(s) \otimes gh. \]

We are interested in the case where $G \leq \SL_{d+1}(\Bbbk)$ is finite. Then $G$ acts naturally on $\Bbbk^{d+1}$, and hence on the polynomial ring $R = \Bbbk [x_1, \ldots, x_{d+1}]$. We fix this notation throughout. 

\begin{Pro}\cite[Theorem 3.2]{BSW}
    Let $G \leq \SL_{d+1}(\Bbbk)$ be finite, acting on the polynomial ring $R = \Bbbk [x_1, \ldots, x_{d+1}]$. Then $R \ast G$ is $(d+1)$-Calabi-Yau.  
\end{Pro}

The following description of the quiver and superpotential for $R \ast G$ is based on the fact that $R$, and hence $R \ast G$, is Koszul with respect to the standard polynomial grading. For the precise description of the superpotential $\omega$, we refer the reader to \cite{BSW}. Right now, it suffices to recall the McKay quiver and some properties of the relations for $R \ast G$. 

\begin{Def}
    Let $G \leq \SL_{d+1}(\Bbbk)$ be finite. Denote the defining $(d+1)$-dimensional representation of $G$ by $\rho$. The McKay quiver $Q = Q_G$ has as vertices the irreducible representations of $G$, i.e.\ $Q_0 = \Irr(G)$, and for two irreducible representations $\chi_i, \chi_j \in Q_0$, the arrows from $\chi_i$ to $\chi_j$ are a basis of $\Hom_{\Bbbk G}(\chi_i , \chi_j \otimes \rho)$. 
\end{Def}

\begin{Theo}\cite[Lemma 3.1, Theorem 3.2]{BSW} \label{Theo: BSW}
    Let $G \leq \SL_{d+1}(\Bbbk)$ be finite, acting on the polynomial ring $R = \Bbbk [x_1, \ldots, x_{d+1}]$ in $d+1$ variables. Then $R \ast G$ is Morita equivalent to a basic algebra $\Bbbk Q_G/I$. The quiver $Q_G$ is the McKay quiver of $G$, and the ideal of relations $I \subseteq (\Bbbk(Q_G)_1)^2$ is induced from the commutativity relations in $R$. More precisely, the ideal $I$ is generated by $(d-1)$-th derivatives of a superpotential $\omega$ on $\Bbbk Q_G$, which is a linear combination of cycles of length $d+1$. 
\end{Theo}

\begin{Rem}
    We will want to refer to ``the quiver'' of $R \ast G$ in this article. However, this terminology is not well-defined. Instead we use this terminology for a locally finite, nonnegatively graded algebra $\Gamma = \bigoplus_{i \geq 0} \Gamma_i$, generated in degrees $0$ and $1$. Recall that locally finite means that $\dim(\Gamma_i)$ is finite for each $i$. We then mean by ``the quiver'' of $\Gamma$ the quiver of the finite-dimensional algebra $\Gamma/\Gamma_{\geq 2}$. We refer the reader to \cite[Section 2.1]{DramburgGasanova3} for a discussion of this convention. In our case, we consider $\Gamma = R \ast G$ with its Koszul grading, and hence ``the quiver'' of $R \ast G$ is the McKay quiver $Q_G$.
\end{Rem}

Finding a correct grading on $R \ast G$ to make it higher preprojective can be simplified by using the notion of a cut. 

\begin{Def}
    Let $G \leq \SL_{d+1}(\Bbbk)$ be finite, and consider $R \ast G \simeq_M \Bbbk Q_G/I$. A higher preprojective grading on $\Bbbk Q_G/I$ is called a \emph{cut} if every arrow in $Q_G$ is homogeneous of degree $0$ or $1$, and every vertex is of degree $0$.
\end{Def}

We discuss the notion of a cut and refer the reader to \cite{DramburgGasanova3} for details. 

\begin{Rem}
    If $\Bbbk Q_G/I$ admits a $(d+1)$-preprojective grading, then so does $R \ast G$. If we assume that the Gabriel quiver of the $d$-representation infinite algebra $\Lambda$ with $\Pi_{d+1}(\Lambda) \simeq \Bbbk Q_G/I \simeq_M R \ast G$ is acyclic, then it follows from \cite[Theorem 5.14]{DramburgSandoy} that an isomorphic subalgebra $\Lambda' \simeq \Lambda$ can be obtained from $\Bbbk Q_G/I$ via a cut. 

    Recently, Tomonaga \cite{TomonagaSilting} discovered an example of a non-acyclic $2$-representation infinite algebra. However, the example can not be obtained from a skew-group algebra, and we believe that all $n$-representation infinite algebras whose higher preprojective algebra is Morita equivalent to $R \ast G$ has acyclic Gabriel quiver. 
\end{Rem}

From now on, we restrict our attention to finding cuts, but warn the reader that it is in general an open problem to find an example of a higher preprojective grading on $R \ast G$ that is not (conjugate to) a cut. Finding cuts is simplified by using the fact that $\Bbbk Q_G/I$ is the derivation quotient of a superpotential as explained in \Cref{Theo: BSW}. In particular, the set $(d+1)$-cycles in the following proposition contains the support of the superpotential. We will later see a slightly strengthened variant of this. 

\begin{Pro}\cite[Corollary 4.7]{Giovannini}\label{Pro: All cycles cut and locally finite suffices}
    Let $R \ast G \simeq_M \Bbbk Q_G/I$ be graded such that each arrow of $Q_G$ is homogeneous of degree $0$ or $1$, and the vertices are of degree $0$. Then this also defines a grading on $\Bbbk Q_G$. If $\Bbbk Q_G/I$ is locally finite dimensional as a graded algebra and all $(d+1)$-cycles in $\Bbbk Q_G$ are homogeneous of degree $1$, then the grading is a cut. 
\end{Pro}

We will also need some way to exclude the existence of cuts. We provide two criteria for this, both are based on loops. The first one follows directly from the superpotential description. 

\begin{Pro}\cite[Proposition 2.15]{DramburgGasanova3}\label{Pro: loop in potential}
    Let $G \leq \SL_{d+1}(\Bbbk)$ be finite and write $R \ast G \simeq_M \Bbbk Q_G/I$ for the McKay quiver $Q_G$ and $I$ given by derivatives of a superpotential $\omega$. If there exists a loop $l$ in $Q_G$ such that $l^i$ for some $i \geq 2$ appears as a summand in $\omega$, then $\Bbbk Q_G/I$ does not admit a higher preprojective cut. 
\end{Pro}

The second criterion, proved in \Cref{Pro: Determinant loops prevent cuts}, requires some arguments involving the center of $\Bbbk Q_G/I$, so we recall some of its properties here. 

\begin{Rem}\label{Rem: Center of R*G}
The center of $R \ast G$ is $Z(R \ast G) = R^G$ the invariant ring. In particular, $Z(R \ast G)$ is a domain. By Morita equivalence, we have that $Z( \Bbbk Q_G/I) \simeq R^G$ as well. 
\end{Rem}

For the readers convenience, we also recall a general structural observation about centers of quotients of path algebras. 

\begin{Rem}\label{Rem: Center of kQ/I}
    Let $Q$ be a quiver and $I \leq \Bbbk Q$ an ideal contained in the square of the arrow ideal. Furthermore, let $z \in Z(\Bbbk Q/I)$ be a central element. Since $z$ commutes with the idempotents $e_i$ coming from vertices in $Q$, we can express $z$ as a sum $z = \sum_{i \in Q_0} e_i z e_i = \sum_{i \in Q_0} e_i z$, and we denote the summands as $z_i = e_i z$. Each summand $z_i$ consists of cycles starting (and hence ending) at $i$. Furthermore, if $a \colon i \to j$ is an arrow in $Q$, we have $z_i a = a z_j$. It follows that if $z_i a \neq 0$, we have $a z_j \neq 0$ and therefore $z_j \neq 0$. In this way, we can think of the summand $z_i$ propagating along arrows in $Q$. 
\end{Rem}

\subsection{Finite subgroups of \texorpdfstring{$\SL_3$}{SL3}} \label{SSec: Finite sub of SL3}
The classification of finite subgroups of $\SL_3(\mathbb{C})$ has a long history going back at least to works of Blichfeldt \cite{BlichfeldtDisc, BlichfeldtFinColl}. A summary can be found in the book \cite{BDMBook} by Miller, Blichfeldt and Dickson. However, the classification was only completed by Yau and Yu \cite{YauYu} when they found two more groups that were overlooked in previous works. We summarize the classification here to give an overview of why the cases will be treated separately. We mainly follow Yau and Yu, but we also refer the reader to the thesis by Carrasco Serrano \cite{Serrano}. 

The classification leads to several natural classes of groups, arising from the following scheme: Let $V = \mathbb{C}^3$, and $G \leq \SL_3(\mathbb{C}) = \SL(V)$. We regard $V$ as a (faithful) $G$-module. Then we can first distinguish whether $V$ is decomposable or not. 
\begin{itemize}[align=left]
    \item[\underline{Decomposable}] If $V$ is decomposable, there are two subcases to consider. Since $V$ has dimension $3$, it can decompose into at most three indecomposable summands. 
    \begin{itemize}[align = left]
        \item[\underline{Type (A)}] If $V$ decomposes into three summands, all of them are $1$-dimensional and hence the group $G$ is abelian. This is called type (A). 
        \item[\underline{Type (B)}] If $V$ decomposes into two summands $V = V_1 \oplus V_2$, where $\dim(V_i) = i$, it follows from the fact that $G \leq \SL(V)$ that $G$ already embeds into $\operatorname{GL}(V_2)$, and we can view $V_1 = \det(V_2)^{-1}$ as the inverse determinant representation of $V_2$. Hence, $G$ is a finite subgroup of $\operatorname{GL}_2(\mathbb{C})$. This is called type (B). Note that sometimes type (A) and type (B) are not taken to be mutually exclusive, so we emphasize that in this article, type (B) refers only to the \emph{nonabelian} groups in $\operatorname{GL}_2(\mathbb{C})$.  
    \end{itemize}
    \item[\underline{Indecomposable}] If $V$ is indecomposable, there are two subcases to consider based on whether $V$ is a primitive or an imprimitive representation. 
    \begin{itemize}[align=left]
        \item[\underline{Types (C) and (D)}] If $V$ is imprimitive, then depending on the stabilizer of a block of imprimitivity, the group $G$ can be realised as an extension of a group of type (A) by a permutation matrix, leading to type (C), or a permutation matrix and a monomial transposition, leading to type (D). 
        \item[\underline{Types (E) - (L)}] If $V$ is primitive, there remain finitely many possible groups. Each group is given its own letter, giving rise to the remaining types. For the scope of this article, we will refer to these groups as \emph{exceptional subgroups}, while of course being aware that these are not the usual \emph{exceptional algebraic groups}.  
    \end{itemize}
\end{itemize}

Let us make some remarks on the above classification as it pertains to this article: The classification gives rise to four infinite families (A) - (D), and finitely many exceptional cases. For the purpose of constructing cuts on the skew-group algebra $R \ast G$, we can (and do) deal with the exceptional cases separately on the computer. We compute the McKay quiver using standard methods in GAP. In many cases, a cut was first found algorithmically, but we provide cuts that we hope the reader finds more easily verifiable by hand. 

In the remaining cases, we have type (B) standing out as the one relating to $\GL_2(\mathbb{C})$. Giving a concise description of the groups and McKay quivers requires detailed knowledge of \emph{all} finite subgroups of $\GL_2(\mathbb{C})$, not just the small ones. This means that we need to describe McKay quivers for arbitrary finite central extensions of the finite subgroups of $\operatorname{PGL}_2(\Bbbk)$. 

The types (C) and (D), in contrast, can be seen as iterated extensions of a group of type (A). These are dealt with in \cite{DramburgGasanova3}. In particular, we note that all such groups are solvable, and this fact is used in the description of the McKay quivers by computing the McKay quiver of some normal subgroup $N \leq G$ and then skewing by the appropriate group $G/N$

\section{Type (A)}\label{Sec: type (A)}
Type (A) are the abelian subgroups of $\SL_3(\Bbbk)$, and have been studied thoroughly both in geometry and representation theory. A construction and classification of cuts for this case has been performed in \cite{DramburgGasanova}, and generalized to the abelian subgroups of $\SL_d(\Bbbk)$ in \cite{DramburgGasanova2}. We will relate type (B) to type (A), so we summarize the necessary parts of these articles here. The same summary can be found in \cite{DramburgGasanova3}. Let $G \leq \SL_3(\Bbbk)$ be finite abelian and denote its order by $n = |G|$. It is easy to see that $G$ then has at most $2$ invariant factors, see \cite{DramburgGasanova}. It follows that $G$ is the quotient of a free abelian group of rank $2$. To encode not only the isomorphism type of $G$ but the given embedding $ \rho \colon G \to \SL_3(\Bbbk)$, we decompose $\rho = \rho_1 \oplus \rho_2 \oplus \rho_3$ as representations. To encode the specific $\rho$, consider the dual group $\hat{G} = \Hom( G, \mathbb{C}^\ast) \simeq G$, and the lattice $\mathbb{Z}^2$ with the basis $e_1, e_2$, and fix a third vector $e_3 = -(e_1 + e_2)$. We fix the map
\[ \mathbb{Z}^2 \to \hat{G}, e_i \mapsto \rho_i, \]
and denote its Kernel by $L \leq \mathbb{Z}^2$. 
We choose a matrix $B$ for the embedding $L \to \mathbb{Z}^2$, where we fix the basis $(e_1, e_2)$ for $\mathbb{Z}^2$. This matrix is only unique up to $\GL_2(\mathbb{Z})$-right multiplication, corresponding to a choice of basis for $L$. However, we may choose to write $B$ in Hermite normal form so that $B$ is upper triangular and $\det(B) = n$. To incorporate the last summand $\rho_3$, we write 
\[  B' = \left(\begin{array}{ c | c }
    B & \begin{matrix}
           1 \\
           1 \\
         \end{matrix} \\
    \hline
    \begin{matrix}
        0 & 0
    \end{matrix}   & 1
  \end{array}\right).
\]

We then have the following characterization for the existence of a $3$-preprojective cut on $R \ast G$ in terms of $B'$. 

\begin{Theo}\cite[Theorem 7.7]{DramburgGasanova} \label{Theo: SL3 type (A) classification}
    Let $G \leq \SL_3(\Bbbk)$ and $B$ be as above. Then there exists a higher preprojective cut $R \ast G$ if and only if there exists a vector $\gamma = (\gamma_1, \gamma_2 , \gamma_3) \in \mathbb{Z}^{1 \times 3}$ with $\gamma_i > 0$ and $\gamma_1 + \gamma_2 + \gamma_3 = n$, such that 
    \[ \gamma \cdot  B'  \in n \mathbb{Z}^{1 \times 3}. \]
\end{Theo}

We also have the following corollary, which the reader may choose to use in place of \Cref{Theo: SL3 type (A) classification} when applying results in type (B). 

\begin{Cor}\cite[Theorem 6.1]{DramburgGasanova}
    Let $G \leq \SL_3(\Bbbk)$ be of type (A). Then $R \ast G$ admits a $3$-preprojective cut if and only if $G \not \simeq C_2 \times C_2$ and $G \hookrightarrow \SL_3(\Bbbk)$ does not factor through $\SL_2(\Bbbk)$.
\end{Cor}

We need the construction of the quiver $Q_G$ as well, which is built from the same data $L \leq \mathbb{Z}^2$. For this, it suffices to note that $Q_G$ is the Cayley-Graph of $\hat{G}$ with respect to the generating set $\{ \rho_1, \rho_2,\rho_3\}$. Define the infinite quiver $\hat{Q}$ via 
\begin{align*}
    \hat{Q}_0 = \mathbb{Z}^2, \quad \hat{Q}_1 = \{ (x \to x + e_i) \mid x \in \mathbb{Z}^2, 1 \leq i \leq 3 \}.
\end{align*}
Both $\mathbb{Z}^2$ and $L$ act on $\hat{Q}$ in an obvious way, and it is easy to see that we can identify
\[ Q_G = \hat{Q}/L.  \]

\section{Type (B)}\label{Sec: Type (B)}
In this section, we investigate the finite subgroups $G \leq \SL_3(\Bbbk)$ whose embedding factors through $\GL_2(\Bbbk)$. We note that it follows from \cite[Theorem 5.15]{Thibault} that if $G \leq \SL_2(\Bbbk) \leq \SL_3(\Bbbk) $, then $Q_G$ does not admit a cut. Furthermore, if $G$ is abelian, then $G$ already occurred in type (A), so from now on we focus on the case where $G \leq \GL_2(\Bbbk) \leq \SL_3(\Bbbk)$ is a finite, nonabelian group, and keep in mind the special case when the embedding factors through $\SL_2(\Bbbk)$. It is important here to point out that in the study of quotient singularities, one only considers the \emph{small} finite subgroups of $\GL_2(\Bbbk)$. However, we want to consider \emph{all} finite subgroups, since even the non-small ones become small after embedding them into $\SL_3(\Bbbk)$. 

\subsection{The group structure} \label{SSec: (B) structure}
Traditionally, the finite subgroups of $\GL_2(\Bbbk)$ are classified in terms of their images in $\operatorname{PGL}_2(\Bbbk)$. We denote by $\pi \colon \GL_2(\Bbbk) \to \operatorname{PGL}_2(\Bbbk) = \operatorname{PSL}_2(\Bbbk)$ the quotient morphism factoring out the scalar multiples of the identity. The finite subgroups of $\PGL_2(\Bbbk)$ are well known, and they come in the typical ADE classification, where type A corresponds to abelian groups, type D corresponds to dihedral groups, and type E gives the tetrahedral, octahedral and icosahedral group. It is easy to see that a finite subgroup $G \leq \GL_2(\Bbbk)$ is therefore a central extension of its image $ H = \pi(G) \leq \PGL_2(\Bbbk)$, but the precise structure is perhaps not obvious. The well-known central extensions are the so-called binary groups $BH = \pi^{-1}(H) \cap \SL_2(\Bbbk)$.  

We now follow \cite{AlgebraicGL2} to provide explicit matrix generators for all finite subgroups of $\GL_2(\Bbbk)$. However, not all of the structure we provide is strictly necessary to deduce the shape of the McKay quiver, but is rather given for completeness. In order to classify finite subgroups of $\GL_2(\Bbbk)$, the authors of \cite{AlgebraicGL2} proceed as follows. Choose a finite subgroup $H \leq \operatorname{PSL}_2(\Bbbk)$. Then the preimage $\pi^{-1}(H) \leq \GL_2(\Bbbk)$ is a subgroup of $\GL_2(\Bbbk)$, but not finite since it contains all scalar multiples of the identity. However, by notherianity, there exists a \emph{minimal} subgroup $H_{\min} \leq \pi^{-1}(H)$ such that $\pi(H_{\min}) = H$, and no smaller subgroup $H' < H_{\min}$ satisfies $\pi(H') = H$. The authors classify these minimal groups, and every other finite subgroup $G_H \leq \GL_2(\Bbbk)$ with $\pi(G_H) = H$ then arises as $G_H =  \langle \mu_l,  H_{\min} \rangle $, where $\mu_l = \epsilon_l \cdot I_2$ for some primitive $l$-th root of unity $\epsilon_l$. We now list the minimal groups. 

\begin{Theo}\cite[Theorem 4]{AlgebraicGL2} \label{Theo: Minimal dihedral groups}
    Let $H \leq \operatorname{PGL}_2(\Bbbk)$ be a dihedral group, of order $|H| = 2n$ and degree $n$. 
    \begin{enumerate}
        \item If $n$ is odd, then there exists, for each $k \geq 1$, a minimal group 
        \[ H_{2n, k} = \left\langle \left( \begin{smallmatrix}
            \epsilon_n & 0 \\ 0 & \epsilon_n^{-1}
        \end{smallmatrix} \right) , \left( \begin{smallmatrix}
            0 & \epsilon_{2^k} \\ \epsilon_{2^k} & 0
        \end{smallmatrix} \right) \right\rangle \]
        such that $\pi( H_{2n ,k}) = H$. The minimal group $H_{2n, 2} \leq \SL_2(\Bbbk)$ is the binary dihedral group. These are all the minimal groups up to conjugation.
        \item If $n\geq 4$ is even, then there exist, for each $k \geq 1$, two minimal groups
       \begin{align*}
           H_{2n, k, 1} &= \left\langle \epsilon_{2^{k+1}} \cdot \left( \begin{smallmatrix}
            \epsilon_{2n} & 0 \\ 0 & \epsilon_{2n}^{-1}
        \end{smallmatrix} \right) , \left( \begin{smallmatrix}
            0 & \epsilon_4 \\ \epsilon_4 & 0
        \end{smallmatrix} \right) \right\rangle, \\
        H_{2n, k, 2} &= \left\langle \left( \begin{smallmatrix}
            \epsilon_{2n} & 0 \\ 0 & \epsilon_{2n}^{-1}
        \end{smallmatrix} \right) , \epsilon_{2^{k+1}} \cdot \left( \begin{smallmatrix}
            0 & \epsilon_4 \\ \epsilon_4 & 0
        \end{smallmatrix} \right) \right\rangle
       \end{align*} 
        such that $\pi( H_{2n ,k, 1}) = \pi( H_{2n ,k, 2}) =  H$. There is one remaining group for each $n$, which is the binary dihedral group $BD_{2n} \leq \SL_2(\Bbbk)$, which can be obtained as $H_{2n, 0, 2}$. These are all the minimal groups up to conjugation.
        \item If $n=2$, then there exists for each $k \geq 0$ one minimal group 
        \[ H_{4, k} = \left\langle \left( \begin{smallmatrix}
            \epsilon_{2n} & 0 \\ 0 & \epsilon_{2n}^{-1}
        \end{smallmatrix} \right) , \epsilon_{2^{k+1}} \cdot \left( \begin{smallmatrix}
            0 & \epsilon_4 \\ \epsilon_4 & 0
        \end{smallmatrix} \right) \right\rangle \]
        such that $\pi( H_{4 ,k}) =  H$. These are all the minimal groups up to conjugation.
    \end{enumerate}
    
\end{Theo}

In particular, we note the following about the groups lying over $D_{2n}$. The observation about the centers will become an important case distinction later on. 

\begin{Rem}
    Let $G \leq \GL_2(\Bbbk)$ be finite with image $\pi(G) \simeq D_{2n}$ a dihedral group. Then the subgroup of diagonal matrices in $G$ has index $2$, so all irreducible representations of $G$ have dimension $1$ or $2$. Furthermore, every non-abelian group $G$ with a cyclic normal subgroup of index $2$ appears this way. 
    The subgroup of scalar matrices is non-trivial when $n$ is even, since in both minimal groups a power of a generator is $-I_2$. When $n$ is odd, the minimal groups contain $-I_2$ if $k \geq 2$. The center of $H_{2n,k}$ is trivial only if $n$ is odd and $k=1$. 
\end{Rem}

Next, we consider the tetrahedral, octahedral and icosahedral case. 

\begin{Theo}\cite[Theorem 4]{AlgebraicGL2}
    Let $A_4 \simeq  H  \leq \operatorname{PGL}_2(\Bbbk)$ be the group of tetrahedral symmetries. For each $k \geq 0$ there exists a minimal group $H_k = \mu_{3^k}  \{\lambda_k(a) a \mid a \in BH   \}$, where $\lambda_k$ is the composition
    \[ BH \to C_3 \to \{1, \epsilon_{3^{k+1}}, \epsilon_{3^{k+1}}^2 \}.  \]
     Here, $BH \to C_3$ is the quotient homomorphism by the derived subgroup $[BH, BH]$ and $C_3 \to \{1, \epsilon_{3^{k+1}}, \epsilon_{3^{k+1}}^2 \} $ is an arbitrary bijection that takes $1$ to $1$. 
     We have $H_0 \simeq BH$, but $H_0 \not \subseteq \SL_2(\Bbbk)$. There is one more minimal group given by $BH \leq \SL_2(\Bbbk) \leq \GL_2(\Bbbk)$. 
\end{Theo}

\begin{Theo}\cite[Theorem 4]{AlgebraicGL2}
    Let $ S_4 \simeq H \leq \PGL_2(\Bbbk)$ be the group of octahedral symmetries. For each $k \geq 0$, there exists a minimal group $H_k = \mu_{2^k}  \{\lambda_k(a) a \mid a \in BH   \}$, where $\lambda_k$ is the composition 
    \[ BH \to C_2 \to \{1, \epsilon_{2^{k+1}} \}.  \]
     Here, $BH \to C_2$ is the quotient homomorphism by the derived subgroup $[BH, BH]$ and $C_2 \to \{1, \epsilon_{2^{k+1}}\} $ is the bijection that takes $1$ to $1$. 
\end{Theo}

\begin{Theo}\cite[Theorem 4]{AlgebraicGL2}
    Let $ A_5  \simeq H \leq \PGL_2(\Bbbk)$ be the group of icosahedral symmetries. Then there is a unique minimal group, given by $BH \leq \SL_2(\Bbbk)$. 
\end{Theo}

We remind the reader that we obtain \emph{every} finite subgroup of $\GL_2(\Bbbk)$ as some $\mu_l \cdot H_{\min}$ for a minimal group $H_{\min}$ as above. To recover or exclude the small groups, one needs to enforce additional conditions on the parameters in $H_{\min}$ and $l$. We refer the reader to \cite{WemyssGL2} for a list of small subgroups of $\GL_2(\Bbbk)$, where the additional conditions can be found.

\subsection{On isoclinism and McKay quivers}
Amongst the central extensions $G \leq \GL_2(\Bbbk)$ of $H = \pi(G) \leq \PGL_2(\Bbbk)$, a special role is played by the binary version $BH \leq \SL_2(\Bbbk)$ of $H$. These are the main objects of the classical $\SL_2(\Bbbk)$-McKay correspondence and have been studied extensively. We will see below that the McKay quivers we obtain for an arbitrary finite $G \leq \GL_2(\Bbbk)$ with image $H \leq \PGL_2(\Bbbk)$ looks like a ``repeated version'' of the quiver for $BH$. This may be expected by noticing that all groups $G \leq \GL_2(\Bbbk)$ with the same image $H \leq \PGL_2(\Bbbk)$ are isoclinic. For the following equivalent definitions of isoclinism, we refer the reader to \cite{ProjRepOfFinGroups, Atlas}, see also the original paper \cite{Hallpgroups} by Hall where the notion was introduced. All other notions from representation- or character theory of finite groups that we use are standard and can be found for example in \cite{isaacsCT}. 

\begin{Def}
    Let $G$ and $H$ be finite groups, and $G'$ and $H'$ their derived subgroups. Then $G$ and $H$ are called \emph{isoclinic} if any of the following equivalent conditions hold. 
    \begin{enumerate}
        \item There exists an isomorphism $\varphi \colon G/Z(G) \to H/Z(H)$ and an isomorphism $\psi \colon G' \to H'$ which commutes with the commutator maps on $G/Z(G)$. More precisely, the following square commutes:
        \[
        \begin{tikzcd}
    G/Z(G) \times G/Z(G) \arrow[rr, "\varphi \times \varphi"] \arrow[dd, "{ [ - ,-]}"] &  & H/Z(H) \times H/Z(H) \arrow[dd, "{[-,-]}"] \\
                                                                                   &  &                                            \\
    G' \arrow[rr, "\psi"]                                                              &  & H'                                        
    \end{tikzcd}
        \]
        \item There exists a group $K$ such that $G, H \leq K$ and 
        \[K = \langle G, Z(K) \rangle = \langle H, Z(K) \rangle.\] 
    \end{enumerate}
\end{Def}

We now fix isoclinic groups $G$ and $H$, as well as a minimal joined supergroup $K$ as above. While the construction of a $K$ from the pair $(\varphi, \psi)$ is not trivial, in all our applications the groups $G$ and $H$ are given as explicit subgroups of $\GL_n(\Bbbk)$, so we can simply choose $K = \langle G, H \rangle \leq \GL_n(\Bbbk)$. The following observations are classical, and we refer the reader in particular to \cite[Section 7]{Atlas}.

\begin{Pro}\label{Pro: irrep extends to K}
    Let $G \leq K$ be finite groups such that $K = \langle G, Z(K) \rangle$, and let $\rho$ be an irreducible representation of $G$. Then $\rho$ extends to $K$, meaning that there exists an irreducible representation $\hat{\rho}$ of $K$ which restricts to $\hat{\rho}_{|G} = \rho$. 
\end{Pro}

\begin{proof}
    Since $K = \langle G, Z(K) \rangle$ and $G$ commutes with $Z(K)$, it follows that $K \simeq (G \times Z(K))/N$ for some normal subgroup $N \leq G \times Z(K)$. More precisely, the group $N$ is given by $N = \{(t, t^{-1}) \mid t \in G \cap Z(K) \} $. Let $\rho \colon G \to \GL_n(\Bbbk)$ be an irreducible representation. Since $\rho$ is irreducible, its restriction to $Z(G)$ is $\rho_{|Z(G)} = \lambda^{\oplus n}$ for a $1$-dimensional representation $\lambda$ of $Z(G)$ by Schur's lemma. Since $Z(K)$ is abelian, the representation $\lambda_{| G \cap Z(K)}$ extends to a representation $\hat{\lambda}$ of $Z(K)$. Then the representation 
    \[ \hat{\rho} = \rho \boxtimes \hat{\lambda}^{-1} \colon G \times Z(K) \to \GL_n(\Bbbk), (g, z) \mapsto \rho(g) \cdot \hat{\lambda}(z)^{-1}  \]
    is an irreducible representation of $G \times Z(K)$. The kernel of $\hat{\rho}$ contains $N$, so we found an irreducible representation of $K$ which satisfies that $\hat{\rho}_{|G} = \rho$. 
\end{proof} 

Note that $\hat{\rho}$ depends on a choice of extension $\hat{\lambda}$, so there are in general several possible extensions $\hat{\rho}$. 
Next, we observe that in our setup, the restriction functor from $K$ to $G$ preserves irreducibility. 

\begin{Pro}\label{Pro: restr of irrep from K to G is irrep}
    Let $G \leq K$ be finite groups such that $K = \langle G, Z(K) \rangle$, and let $\rho$ be an irreducible representation of $K$. Then the restriction $\rho_{|G}$ is an irreducible representation of $G$. 
\end{Pro}

\begin{proof}
    Since $\rho$ is irreducible, it follows that $\rho(Z(K))$ consists of scalar matrices. The representation $\rho_{|G}$ takes $G$ to $\rho(G) \leq \rho(K)$. If $\rho_{|G}$ was decomposable, we could conjugate $\rho(G)$ to simultaneously block-diagonalise the matrices in $\rho(G)$. But that would mean we could simultaneously block-diagonalise all matrices in $\rho(G) \cup \rho(Z(K))$, and this generates $\rho(K)$, hence $\rho$ would be decomposable. 
\end{proof}

This allows us to transfer irreducible representations between isoclinic groups. 

\begin{Con}\label{Con: Transfering irreps between isoclinics}
    Let $G, H \leq K$ be isoclinic finite groups. Let $\rho \colon G \to \GL_n(\Bbbk)$ be an irreducible representation. Then we obtain an irreducible representation $\hat{\rho}$ of $K$ by \Cref{Pro: irrep extends to K}. By \Cref{Pro: restr of irrep from K to G is irrep}, its restriction $\hat{\rho}_{|H}$ to $H$ then is an irreducible representation of $H$ of the same dimension. The matrices in $\Im(\hat{\rho}_{|H})$ differ from those in $\Im(\rho)$ by scalars. This means in particular that we can choose generators $\langle g_1, \ldots, g_l \rangle = G$ such that $\Im(\hat{\rho}_{|H}) = \langle c_1 \rho(g_1), \ldots, c_l \rho(g_l) \rangle$ for non-zero scalars $c_i \in \Bbbk$.
\end{Con}

A well-known statement for isoclinic groups is the following.

\begin{Lem}
    Let $G$ and $H$ be isoclinic groups. Then the dimensions of irreducible representations of $G$ and $H$ are the same. Furthermore, if we denote by $m_1 \neq 0$ the number of irreducible representations of $G$ of dimension $d$, and by $m_2$ the number of irreducible representations of $H$ of dimension $d$, we have that $\frac{m_1}{m_2} = \frac{|G|}{|H|}. $
\end{Lem}

We will use some terminology that is, to our knowledge, not standard. 

\begin{Def}
    Let $G, H \leq K$ be isoclinic groups, where $K$ is chosen minimal, and let $\rho$ be an irreducible representation of $G$ and let $\rho'$ be an irreducible representation of $H$. Then we call the representations $\rho$ and $\rho'$ \emph{isoclinic representations} with respect to $K$ if there exists a representation $\hat{\rho}$ of $K$ such that $\hat{\rho}_{|H} = \rho'$ and $\hat{\rho}_{|G} = \rho$, and we write 
    \[ \rho \sim_K \rho'. \]
\end{Def}

Isoclinism, in particular \Cref{Con: Transfering irreps between isoclinics} and \Cref{Pro: restr of irrep from K to G is irrep}, has the following consequence, which indicates why the quivers look similar. 

\begin{Cor}\label{Lem: Tensors of isoclinics decompose in the same dimensions}
    Let $G, H \leq K$ be isoclinic, and let $\rho_1$ and $\rho_2$ be irreducible representations of $G$. Let $\rho_1'$ and $\rho_2'$ be irreducible representations of $H$ such that $\rho_i'$ is isoclinic to $\rho_i$. Then the following hold. 
    \begin{enumerate}
        \item There are decompositions $\rho_1 \otimes \rho_2 = \varphi_1 \oplus \ldots \oplus \varphi_n$ and $\rho_1' \otimes \rho_2' = \varphi_1' \oplus \ldots \oplus \varphi_n'$ into irreducible summands such that $\varphi_i \sim_K \varphi_i'$ for $1 \leq i \leq k$. 
        \item The irreducible summands of $\rho_1 \otimes \rho_2$ have the same multiset of dimensions as the irreducible summands of $\rho_1' \otimes \rho_2'$.
    \end{enumerate}
\end{Cor}

\begin{proof}
    It suffices to note that an extension of $\rho_1 \otimes \rho_2$ to $K$ is given by $\hat{\rho}_1 \otimes \hat{\rho}_2$ for extensions $\hat{\rho}_i$ to $K$ as in \Cref{Pro: irrep extends to K}. Since restriction commutes with tensor products and preserves indecomposability in our setup, the claims follow.
\end{proof}

We conclude the discussion of isoclinism by pointing out that this is not enough to transfer the knowledge of a higher preprojective cut. 

\begin{Rem}
    The reader familiar with \cite{McKayDecomp, browne2021connectivity} may expect that for two isoclinic subgroups $G_1, G_2 \leq \SL_{n+1}(\Bbbk)$, the McKay quiver of $G_1$ admits a higher preprojective cut if and only if the McKay quiver of $G_2$ does. This is not the case, as one can see by comparing the essentially unique embedding of $C_2 \times C_2$ into $\SL_3(\Bbbk)$ with the embedding of $C_4$ into $\SL_3(\Bbbk)$ taking a generator to the matrix $\frac{1}{4}(1,1,2)$. The quiver for $C_4$ admits a cut, while the one for $C_2 \times C_2$ does not. The groups are abelian, hence isoclinic, and even have the same order. However, the defining representations are decomposable, and not isoclinic in a meaningful way. It would be interesting to find necessary and sufficient criteria to decide when a cuts can be transferred between the McKay quivers of two isoclinic groups with fixed representations.
\end{Rem}

\subsection{The McKay quiver and cuts}
Now we construct the McKay quivers and cuts for groups of type (B). As before, fix a finite subgroup $G \leq \GL_2(\Bbbk)$. Later, we will embed $G$ into $\SL_3(\Bbbk)$ in the obvious way: 
\[ g \mapsto \left(\begin{smallmatrix}
    g & 0 \\ 0 & \det(g)^{-1}
\end{smallmatrix}\right) \]
On the level of representations of $G$, this means that we have the $3$-dimensional representation $\mathbb{C}^3 = V = V' \oplus \det(V')^{-1}$. In particular, the McKay quiver can be computed by first computing the McKay quiver only for $V'$, and then adding the additional arrows coming from $\det(V')^{-1}$. These will be called the \emph{determinant arrows}. We also point out that the reader may recognize many of the computed quivers from \cite{AuslanderReitenMcKay} and \cite{IyamaWemyssSpecial}, as well as \cite{NolladeCelisDihedral} and \cite{NdCS}. However, both sources deal only with small subgroups of $\GL_2(\Bbbk)$, and we do not yet understand the precise relationship between our methods and \cite{AuslanderReitenMcKay}, despite the strong resemblance. 
The following convention will allow us to write the coming computations more succinctly, since we can obtain the quiver $Q_G$ from the irreducible characters of $G$.  

\begin{Conv}
    Let $\rho \colon G \to \GL_n(\Bbbk)$ be a representation. Then we refer to its character by $\rho$ as well. For two representations $\rho_1$ and $\rho_2$, we denote by $\rho_1 \otimes \rho_2$ the tensor product of representations, and by $\rho_1 \cdot \rho_2$ the product of the characters, which we often abbreviate to $\rho_1 \rho_2$. Similarly $\rho_1 \oplus \rho_2$ is the direct sum of representations and $\rho_1 + \rho_2$ the sum of characters. Moreover, $\rho^\ast$ is the dual representation while $\overline{\rho}$ is the dual character. 
    We denote the inner product of class functions $\varphi$ and $\psi$ of $G$ by 
    \[ (\varphi, \psi)_G,    \]
    and we drop the subscript $G$ when no confusion is possible. 
\end{Conv}

\begin{Rem}
    We remind the reader that the McKay quiver of $\rho \colon G \hookrightarrow \GL_n(\Bbbk)$ has as its vertices the irreducible representations of $G$, and for two irreducible representations $\chi_1$ and $\chi_2$ the number of arrows from $\chi_1$ to $\chi_2$ equals 
    \[ \dim \Hom_{kG} ( \chi_1, \chi_2 \otimes \rho) = (\chi_1, \chi_2 \cdot \rho)_G, \]
\end{Rem}

Throughout, we will need the following easy observation which we prove here for the readers convenience. 

\begin{Lem}\label{Lem: Dimension 2 adjoint trick}
    Let $\rho \colon G \rightarrow \GL_2(\Bbbk)$ be a two dimensional representation of an arbitrary finite group. Then
    \[ \rho \simeq \rho^\ast \otimes \det(\rho). \]
\end{Lem}

\begin{proof}
    For $g \in G$, consider the matrix $\rho(g)$. This is diagonalisable, so denote by $a$ and $b$ the eigenvalues. Then the character of $\rho$ at $g$ takes the value $a+b$, and $\det(\rho)$ takes the value $ab$. The dual representation $\rho^\ast$ takes $g$ to the conjugate transpose of the matrix $\rho(g)$, hence the character of $\rho^\ast$ at $g$ takes the value $\overline{a} + \overline{b}$, and since $a$ and $b$ are roots of unity it follows that 
    \[ (\overline{a} + \overline{b}) \cdot (ab) =  \left( \frac{1}{a} + \frac{1}{b} \right) (ab) = b+a, \]
    so the characters of $\rho$ and $\rho^\ast \otimes \det(\rho)$ agree. 
\end{proof}

We note the following observation, which is relevant because McKay quivers for self-dual representations are symmetric. 

\begin{Cor}\label{Cor: self-duality}
    Let $\rho \colon G \rightarrow \GL_2(\Bbbk)$ be a two dimensional representation of an arbitrary finite group. Then
    \[ \rho \simeq \rho \otimes \det(\rho),  \]
    if and only if $\rho \simeq \rho^\ast$ is self-dual.
\end{Cor}

The fact that we embed $G$ into $\SL_3(\Bbbk)$ also allows for the following variant of \Cref{Pro: All cycles cut and locally finite suffices}, that follows from \cite[Corollary 4.7 \& Proposition 4.8]{Giovannini}

\begin{Cor}\label{Cor: deter cycles cut and acyclic suffices}
    Let $G \leq \GL_n(\Bbbk) \leq \SL_{n+1}(\Bbbk)$ be a finite group, and let $R \ast G \simeq_M \Bbbk Q_G/I$ be graded such that each arrow of $Q_G$ is homogeneous of degree $0$ or $1$, and the vertices are of degree $0$. Then this also defines a grading on $\Bbbk Q_G$. If $\Bbbk Q_G/I$ is locally finite dimensional as a graded algebra and all $(n+1)$-cycles in $\Bbbk Q_G$ which contain exactly one determinant arrow are homogeneous of degree $1$, then the grading is a cut. 
\end{Cor}

\begin{proof}
    The proof is verbatim that of \Cref{Pro: All cycles cut and locally finite suffices}, with the only change that by \cite[Proposition 4.8]{Giovannini}, the support of $\omega$ consists of cycles containing exactly one determinant arrow. 
\end{proof}

The distinct set of determinant arrows at our disposal also leads to the following criterion to disprove the existence of cuts when a determinant arrow is a loop, which we prove here in any dimension. 

\begin{Pro}\label{Pro: Determinant loops prevent cuts}
    Let $G \leq \GL_n(\Bbbk) \leq \SL_{n+1}(\Bbbk)$ be a finite group, and suppose that a determinant arrow $l \colon i \to i$ in $Q_G$ is a loop. Then $\Bbbk Q_G/I$ does not admit an $(n+1)$-preprojective cut. 
\end{Pro}

\begin{proof}
    Suppose for a contradiction that there exists a cut, which in particular means that the degree $0$ part of this grading is finite-dimensional. Note first that the loop $l$ can not be in degree $0$ for any preprojective cut by the No-Loop theorem \cite{NoLoops}. Since we consider a cut, it follows that $l$ has to have degree $1$. Next, we consider the center $Z(\Bbbk Q_G/I) \simeq R^G$ as discussed in \Cref{Rem: Center of kQ/I}. Since the group $G$ acts by $\delta^{-1}$ on the last variable $x_{n+1}$ of the polynomial ring $R$, it follows that $z = x_{n+1}^m \in R^G$ for $m = |\delta|$, hence $z$ is central. Next, we interpret $z$ as an element in $\Bbbk Q_G/I$. We can view the determinant arrows as $e_j x_{n+1} e_l$ for vertices $j, l \in (Q_G)_0$, and hence when we identify the central element $z = x_{n+1}^m$ with a sum of $m$-cycles in $\Bbbk Q_G/I$, we see that $l^m$ is a summand in $z$, and that when we write $ z = \sum_{j \in {Q_G}_0} e_j z$ each component $z_j = e_j z$ consists only of determinant arrows. More precisely, each summand is just a scalar multiple of one $m$-cycle consisting of determinant arrows, since each vertex only has one outgoing determinant arrow. In particular, the summand $z_i$ is a scalar multiple of $l^m$, and hence homogeneous of degree $m$. Because each summand $z_j$ is a scalar multiple of a cycle, we can conclude that each summand $z_j$ is homogeneous of some degree. Next, we argue that indeed each summand $z_j$ is homogeneous of degree $m$. Suppose that this is not the case. Then we can decompose $z$ into homogeneous components, each of which is of the form $z_J = \sum_{j \in J} z_j$ for some $J \subset (Q_G)_0$. Next, recall that the center of any graded algebra becomes a graded subalgebra. It therefore follows that because each $z_J$ is homogeneous, each $z_J$ is central and hence an invariant polynomial in $\Bbbk[x_{n+1}]$ with respect the the determinant action of $G$. But since the components $z_J$ are supported at disjoint sets of vertices, we have $z_J z_{J'} = 0$ for different sets $J \neq J'$, which contradicts the fact that the invariant ring $\Bbbk[x_{n+1}]^G \subseteq \Bbbk[x_{n+1}]$ is a domain. 
    Thus, we conclude that $z$ is homogeneous of degree $m$, meaning that each determinant arrow has degree $1$, meaning that $x_{n+1}$ is homogeneous of degree $1$. 
    Finally, since the support of the potential contains only $(n+1)$-cycles with a determinant arrow and each arrow in $Q_G$ appears in the potential, it follows that the cut consists only of determinant arrows. But this means that the degree $0$ part of $\Bbbk Q_G/I$ contains the infinite-dimensional subalgebra $\Bbbk Q_G'/I' \simeq_M \Bbbk[x_1, \ldots, x_n] \ast G$ for $G \leq \GL_n(\Bbbk)$, contradicting the assumption that we started with a cut.  
\end{proof}

Finally, since we use inductive reasoning in this section, we remind the reader of the fact that the quivers we consider are connected. Recall that being strongly connected means that for every pair of vertices $(v_1,v_2)$, there is a directed path from $v_1$ to $v_2$. 

\begin{Pro}\cite[Proposition 3.3 \& Proposition 3.10]{browne2021connectivity} \label{Pro: Quiver is connected}
    Let $\rho \colon G \rightarrow \GL_n(\Bbbk)$ be a representation. Then $Q_G$ is strongly connected if and only if $\rho$ is faithful. 
\end{Pro}

\subsubsection{The dihedral case}
Let now $G \leq \GL_2(\Bbbk)$ be a central extension of a dihedral group $H = D_{2n}$ of order $|H| = 2n$ and degree $n \geq 2$. We denote by $\mathbf{1}$ the trivial representation, by $\rho \colon G \to \GL_2(\Bbbk)$ the defining representation, by $\delta = \det(\rho)$ the determinant representation, and by $\gamma$ the central character of $\rho$. Since $G$ has a faithful irreducible representation, it follows that the center is cyclic of order $|\gamma|$. We subdivide cases based on whether $Z(G)$ has order divisible by $2$. This is akin to the usual distinction of dihedral groups into odd and even degrees. Note that the cases can be seen as asking whether $G$ contains $Z(\SL_2(\mathbb{C}))$ or not, so we refer to this as the central respectively non-central cases. The non-central case requires some more discussion of the group structure, while the central case requires a longer computation of McKay quivers.

\medskip

\paragraph{\textit{The non-central case}}
Let $G$ be a central extension of $D_{2n}$ such that $-I_2 \not \in G$. Then we have from \Cref{Theo: Minimal dihedral groups} that $G$ contains a minimal group $H_{2n,k}$ for odd $n$, and furthermore $k=1$. Note that $|H_{2n,1}| = 2n$. We write $m = \frac{|G|}{2n}$, and can conclude that $G \simeq H_{2n,1} \times C_m$, where $C_m$ is embedded into $\GL_2(\Bbbk)$ as diagonal matrices. In particular this means that $m$ is odd, since otherwise $C_m$ would give rise to an element $-I_2 \in G$. 

\begin{Rem}
    If $m=1$, these groups were considered in \cite[Section 3.3]{NdCS}, because $H_{2n, 1} \simeq D_{2n}$. In particular, the quiver for $G \leq \GL_2(\Bbbk)$ contains loops. 
\end{Rem}

Now, we embed $G \leq \GL_2(\Bbbk)$ into $\SL_3(\Bbbk)$. This means that $C_m \leq \SL_3(\Bbbk)$ is generated by a diagonal matrix $\diag(\epsilon_m, \epsilon_m, \epsilon_m^{-2}) $. If $m \geq 3$, this is enough to produce a cut on $R \ast C_m$, which we then can extend to one on $R \ast G \simeq (R \ast C_m) \ast D_{2n} $. 

\begin{Pro}
    Let $G \leq \SL_3(\Bbbk)$ be a group of type (B) that is a central extension of a dihedral group. If $-I_2 \not \in G $, then $R \ast G \simeq \Bbbk Q_G/I$ admits a $3$-preprojective cut if and only if $G \not \simeq D_{2n}$. 
\end{Pro}

\begin{proof}
    If $G \simeq D_{2n}$, then it follows from \cite{NdCS} that $Q_G$ has a loop. It is easy to check that the loops are given by determinant arrows, so by \Cref{Pro: Determinant loops prevent cuts}, $R \ast G$ can not be $3$-preprojective. Otherwise, we have that $G \simeq C_m \times D_{2n}$ for $m \geq 3$. We first note that $R \ast C_m$ has a cut by \Cref{Theo: SL3 type (A) classification}, since a type vector for $C_m$ is given by $(1,1,m-2)$. Then, by \cite[Corollary 5.3]{DramburgGasanova}, this gives rise to a cut on $R \ast G \simeq (R \ast C_m) \ast D_{2n}$. 
\end{proof}

We finish by noting that the value $m$ can also be seen as the order of $\det(\rho)$, since $m$ is odd. 

\medskip

\paragraph{\textit{The central case}}
Now, we consider the case where $-I_2 \in G$. As alluded to, we need to compute the McKay quiver explicitly. We denote by $\mathbf{1}$ the trivial representation, by $\rho \colon G \to \GL_2(\Bbbk)$ the defining representation, by $\delta = \det(\rho)$ the determinant representation, and by $\gamma$ the central character of $\rho$. In particular, we now know that $\gamma$ has order divisible by $2$. 

We first work out the part of the quiver surrounding the representation $\rho$. To find the arrows \emph{into} the vertex $\rho$, we need to decompose $\rho \otimes \rho$. 

\begin{Rem}
    Let $\rho$ be the defining representation. Then $\rho \otimes \rho$ has dimension $4$, and decomposes into the alternating and symmetric square 
    \[ \rho \otimes \rho \simeq \wedge^2 (\rho) \oplus \operatorname{Sym}^2(\rho). \]
    Since $\rho$ has dimension $2$, it follows that $\wedge^2(\rho) \simeq \det(\rho) = \delta$. Hence, $\operatorname{Sym}^2(\rho)$ has dimension $3$, and since any irreducible representation of $G$ has at most dimension $2$, it follows that $\operatorname{Sym}^2(\rho) \simeq \lambda' \oplus \varphi'$ for some $1$-dimensional representation $\lambda'$.  
    Next, we show that all the summands are different. First, we note that $\rho \not \simeq \varphi'$, since $\rho$ has non-trivial central character $\gamma$, and hence $\varphi'$ has central character $\gamma^2 \neq \gamma$. Furthermore, we have 
    \begin{align*}
        1 &= (\rho, \rho) = (\rho \cdot \overline{\rho} , \mathbf{1}) = (\rho \cdot \rho \cdot \delta^{-1} , \mathbf{1}) =  (\rho \cdot \rho \cdot \overline{\delta}, \mathbf{1}) =  (\rho \cdot \rho, \delta) \\
        &= (\wedge^2( \rho) + \operatorname{Sym}^2(\rho), \delta) = (\delta, \delta) + (\operatorname{Sym}^2(\rho), \delta) = 1 + (\operatorname{Sym}^2(\rho), \delta),
    \end{align*}
    so $\delta$ is not a summand in $\operatorname{Sym}^2(\rho)$. We will later see that $(\operatorname{Sym}^2(\rho), \lambda') = 1$. 
    To summarise, we now know that $\rho \otimes \rho$ decomposes into at least three distinct representations of $G$, none of which is isomorphic to $\rho$. The representation $\varphi'$ can be reducible, so this is a case distinction we will make later on.
\end{Rem}

So far, we have found the following small part of the McKay quiver. 
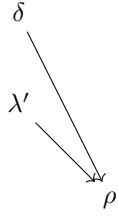
\[ 
\begin{tikzcd}[column sep = {between origins, 8ex}, row sep = {between origins, 8ex}]
\delta \arrow[rdd]  &      \\
\lambda'  \arrow[rd] &      \\
                    & \rho
\end{tikzcd}
\]
Next, note that this part can be translated by $\delta^{-1}$, which acts on $\Irr(G)$. For symmetry reasons, we now write $\lambda = \delta^{-1} \cdot \lambda'$ and $\varphi = \delta^{-1} \cdot \varphi'$. The same computations then produce the following part of the McKay quiver. 

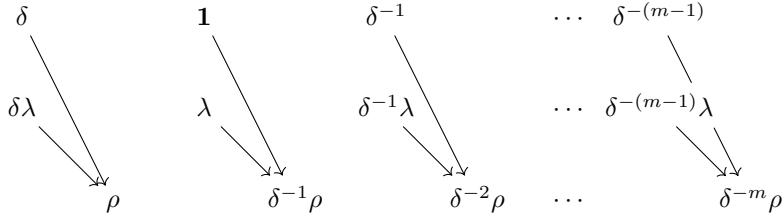
\[ 
\begin{tikzcd}[column sep = {between origins, 8ex}, row sep = {between origins, 8ex}]
\delta \arrow[rdd]  &      & \mathbf{1} \arrow[rdd]         &                  & \delta^{-1} \arrow[rdd]        &                  &     \cdots    &  \delta^{-(m-1)} \arrow[rdd,  dash pattern=on 20pt off 19pt ]        &                  \\
\delta \lambda  \arrow[rd] &      &  \lambda \arrow[rd] &                  & \delta^{-1} \lambda \arrow[rd] &                  & \cdots  &  \delta^{-(m-1)} \lambda \arrow[rd] &                  \\
                    & \rho &                                & \delta^{-1} \rho &                                & \delta^{-2} \rho &    \cdots     &                                & \delta^{-m} \rho
\end{tikzcd}
\]
Next, we connect these pieces. 

\begin{Rem}
    Clearly, there is an arrow $\rho \to \mathbf{1}$, since $\rho \simeq \rho \otimes \mathbf{1}$. Furthermore, we have that 
    \begin{align*}
        (\rho, \delta^{-1} \cdot \lambda' \cdot \rho) &= (\rho, \overline{\rho} \cdot \lambda') = (\rho \cdot \rho, \lambda') = (\delta, \lambda') + (\operatorname{Sym}^2(\rho), \lambda') \\
        &= 0 + (\operatorname{Sym}^2(\rho), \lambda') \geq 1.
    \end{align*}
    Together with the fact that $\rho$ is irreducible and $\dim(\rho) = \dim(\delta^{-1} \otimes \lambda' \otimes \rho )$, we also have 
    \[ 1 \geq (\rho, \delta^{-1} \cdot \lambda' \cdot \rho) = (\operatorname{Sym}^2(\rho), \lambda') \geq 1, \]
    so equality holds.  
\end{Rem}

Thus, we can now fill in the part of the quiver that we know as follows. 

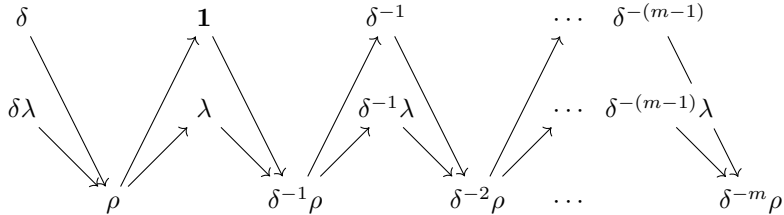
\[
\begin{tikzcd}[column sep = {between origins, 8ex}, row sep = {between origins, 8ex}]
\delta \arrow[rdd]  &                             & \mathbf{1} \arrow[rdd]         &                                         & \delta^{-1} \arrow[rdd]        &                  &   \cdots      &  \delta^{-(m-1)} \arrow[rdd,  dash pattern=on 20pt off 19pt]        &                  \\
\delta \lambda  \arrow[rd] &                             &  \lambda \arrow[rd] &                                         & \delta^{-1} \lambda \arrow[rd] &                  & \cdots  & \delta^{-(m-1)} \lambda \arrow[rd] &                  \\
                    & \rho \arrow[ruu] \arrow[ru] &                                & \delta^{-1} \rho \arrow[ruu] \arrow[ru] &                                & \delta^{-2} \rho \arrow[ru] \arrow[ruu] &    \cdots     &                                & \delta^{-m} \rho 
\end{tikzcd}
\]

Now, we need to make a case distinction. If the remaining summand $\varphi'$ of $\operatorname{Sym}^2(\rho)$ is decomposable, we obtain two further non-isomorphic representations $\lambda'_2$ and $\lambda'_3$. We write as before $\lambda_i = \delta^{-1} \cdot \lambda'_i$ and $\varphi = \delta^{-1} \cdot \varphi'$. The same argument as above shows that $(\lambda_i', \operatorname{Sym}^2(\rho)) = 1$, so we have that all summands of $\operatorname{Sym}^2(\rho)$ are different. Thus, adding in the vertices for $\lambda_2$ and $\lambda_3$, as well as the products with $\delta^{-i}$, produces the following quiver. Note that $\lambda_2 \otimes \rho$ has dimension $2$, so there are no other arrows pointing to a vertex $\lambda_2$ or $\lambda_3$. 

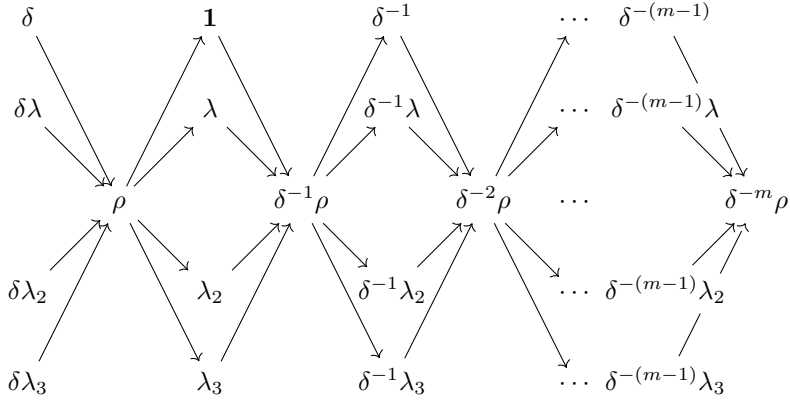
\[
\begin{tikzcd}[column sep = {between origins, 8ex}, row sep = {between origins, 8ex}]
\delta \arrow[rdd]           &                                                    & \mathbf{1} \arrow[rdd] &                                                                & \delta^{-1} \arrow[rdd]           &                  &    \cdots     &  \delta^{-(m-1)} \arrow[rdd, dash pattern=on 20pt off 19pt]           &                  \\
\delta \lambda  \arrow[rd]   &                                                    & \lambda \arrow[rd]     &                                                                & \delta^{-1} \lambda \arrow[rd]    &                  & \cdots  &  \delta^{-(m-1)} \lambda \arrow[rd]    &                  \\
                             & \rho \arrow[ruu] \arrow[ru] \arrow[rd] \arrow[rdd] &                        & \delta^{-1} \rho \arrow[ruu] \arrow[ru] \arrow[rd] \arrow[rdd] &                                   & \delta^{-2} \rho \arrow[ru] \arrow[ruu] \arrow[rd] \arrow[rdd]&  \cdots       &                                   & \delta^{-m} \rho \\
\delta \lambda_2 \arrow[ru]  &                                                    &  \lambda_2 \arrow[ru]  &                                                                & \delta^{-1} \lambda_2 \arrow[ru]  &                  &     \cdots    &  \delta^{-(m-1)} \lambda_2 \arrow[ru]  &                  \\
\delta \lambda_3 \arrow[ruu] &                                                    &  \lambda_3 \arrow[ruu] &                                                                & \delta^{-1} \lambda_3 \arrow[ruu] &                  &      \cdots   &  \delta^{-(m-1)} \lambda_3 \arrow[ruu, dash pattern=on 18pt off 17pt] &                 
\end{tikzcd}
\]

Note that $\delta$ has finite order, so we choose $m \geq 1$ to be minimal so that the vertex $\delta^{-m} \rho$ is equal to the vertex $\rho$. Indeed it is not difficult to see that $\delta$ has order $m$ or $2m$. 

\begin{Rem}
    By assumption, we have that $\delta^{-m} \rho \simeq \rho$. Taking determinants, we find that 
    \[  \delta = \det(\rho) = \det(\delta^{-m} \rho) = \delta^{-2m} \delta. \]
    This yields that $\delta^{-2m} = (\delta^{-m})^2 = 1$, showing that $\delta$ has order $m$ or $2m$. 
    Next, we compute that 
    \begin{align*}
        1 = (\rho, \delta^{-m} \rho) = (\rho \rho, \delta^{-m} \delta ) = (\delta + \delta \lambda + \delta \lambda_2 + \delta \lambda_3, \delta^{-m} \delta).
    \end{align*}
    From this it follows that $\delta^{-m} \delta$ is isomorphic to $\delta$ or to some $\delta \lambda_i$. In the former case, we have that $\delta$ has order $m$. In the latter case, we relabel so that $\delta \lambda \simeq \delta^{-m} \delta$.    
\end{Rem}

We have found the whole McKay quiver, since by \Cref{Pro: Quiver is connected} the quiver is strongly connected and in every step we added all incoming arrows to all new vertices. 

Now, we embed $G$ into $\SL_3(\Bbbk)$, meaning we add additional arrows which are opposite to the action of $\delta^{-1}$ on the quiver, that is according to the action of $\delta$. We note that $\delta$ can act on $\Irr(G)$ with orbits of different size, so it is not obvious how to connect the vertices corresponding to the $1$-dimensional representations. Depending on whether $\delta$ has order $m$ or $2m$, we have $\delta^{-(m-1)} \simeq \delta$ or $\delta^{-(m-1)} \simeq \delta \lambda$, so two quivers can appear. Thus, in the drawing below one needs to choose either the blue or the green arrows. 

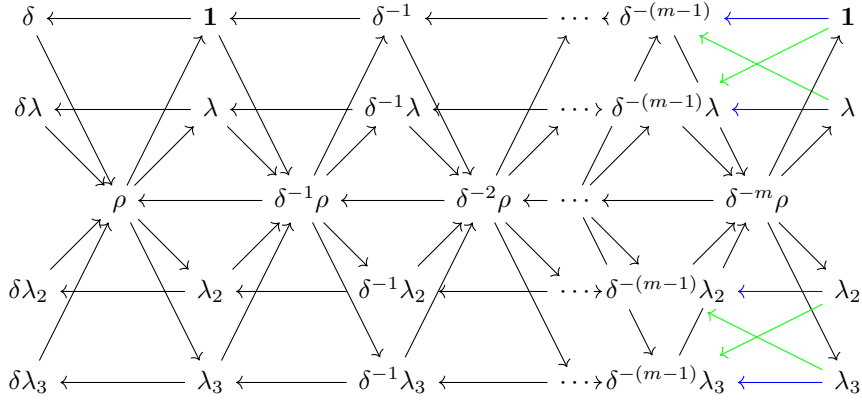
\[
\begin{tikzcd}[column sep = {between origins, 8ex}, row sep = {between origins, 8ex}]
\delta \arrow[rdd]           &                                                    & \mathbf{1} \arrow[rdd] \arrow[ll] &                                                                           & \delta^{-1} \arrow[rdd] \arrow[ll]           &                                                                           & \cdots  \arrow[ll]                                             &  \delta^{-(m-1)} \arrow[rdd, dash pattern=on 20pt off 19pt] \arrow[l]           &                                                                           & \mathbf{1}  \arrow[ll, color=blue]  \arrow[lld, color=green]         &                 \\
\delta \lambda  \arrow[rd]   &                                                    & \lambda \arrow[rd] \arrow[ll]     &                                                                           & \delta^{-1} \lambda \arrow[rd] \arrow[ll]    &                                                                           & \cdots  \arrow[ll]                                             &  \delta^{-(m-1)} \lambda \arrow[rd] \arrow[l]    &                                                                           &  \lambda  \arrow[ll, color=blue] \arrow[llu, color=green]   &                 \\
                             & \rho \arrow[ruu] \arrow[ru] \arrow[rd] \arrow[rdd] &                                   & \delta^{-1} \rho \arrow[ruu] \arrow[ru] \arrow[rd] \arrow[rdd] \arrow[ll] &                                              & \delta^{-2} \rho \arrow[ll] \arrow[ru] \arrow[ruu] \arrow[rd] \arrow[rdd] & \cdots \arrow[l] \arrow[ru] \arrow[ruu, dash pattern=on 25pt off 19pt] \arrow[rd] \arrow[rdd, dash pattern=on 20pt off 19pt] &                                             & \delta^{-m} \rho \arrow[ll] \arrow[rd] \arrow[rdd] \arrow[ru] \arrow[ruu] &                                         & \\
\delta \lambda_2 \arrow[ru]  &                                                    &  \lambda_2 \arrow[ru] \arrow[ll]  &                                                                           & \delta^{-1} \lambda_2 \arrow[ru] \arrow[ll]  &                                                                           & \cdots  \arrow[ll]                                             &  \delta^{-(m-1)} \lambda_2 \arrow[ru] \arrow[l]  &                                                                           &  \lambda_2  \arrow[ll, color=blue] \arrow[lld, color=green] &                 \\
\delta \lambda_3 \arrow[ruu] &                                                    &  \lambda_3 \arrow[ruu] \arrow[ll] &                                                                           & \delta^{-1} \lambda_3 \arrow[ruu] \arrow[ll] &                                                                           & \cdots \arrow[ll]                                              &  \delta^{-(m-1)} \lambda_3 \arrow[ruu,dash pattern=on 20pt off 17pt] \arrow[l] &                                                                           &  \lambda_3  \arrow[ll, color=blue] \arrow[llu, color=green] &                
\end{tikzcd}
\]

We return to the other case, namely when the summand $\varphi'$ of $\operatorname{Sym}^2(\rho)$ is irreducible. Recall that we write $\varphi = \delta^{-1} \varphi'$. We immediately add in all the irreducible representations $\delta^{-i} \otimes \varphi$ to obtain the following partial McKay quiver.

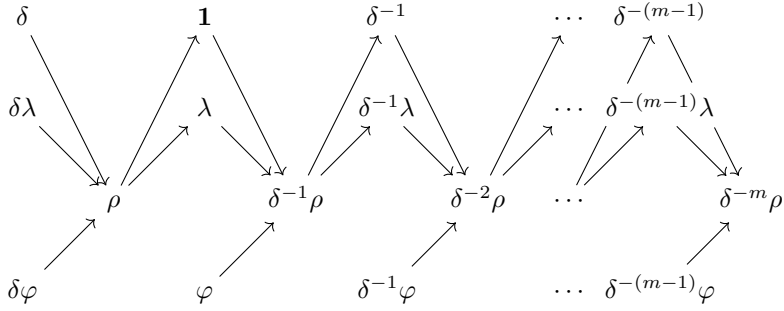
\[
\begin{tikzcd}[column sep = {between origins, 8ex}, row sep = {between origins, 8ex}]
\delta \arrow[rdd]        &                             & \mathbf{1} \arrow[rdd] &                                         & \delta^{-1} \arrow[rdd]        &                                         & \cdots                        &  \delta^{-(m-1)} \arrow[rdd, dash pattern=on 20pt off 19pt]        &                  \\
\delta \lambda \arrow[rd] &                             & \lambda \arrow[rd]     &                                         & \delta^{-1} \lambda \arrow[rd] &                                         & \cdots                        &  \delta^{-(m-1)} \lambda \arrow[rd] &                  \\
                          & \rho \arrow[ruu] \arrow[ru] &                        & \delta^{-1} \rho \arrow[ruu] \arrow[ru] &                                & \delta^{-2} \rho \arrow[ru] \arrow[ruu] & \cdots \arrow[ru] \arrow[ruu, dash pattern=on 25pt off 18pt] &                                & \delta^{-m} \rho \\
\delta \varphi \arrow[ru] &                             & \varphi \arrow[ru]     &                                         & \delta^{-1} \varphi \arrow[ru] &                                         & \cdots                        &  \delta^{-(m-1)} \varphi \arrow[ru] &                 
\end{tikzcd}
\]

We note that the newly added vertices are all new. 

\begin{Rem}
    We already noted from the central characters that $\rho \not \simeq \varphi$. Recall that since $- I_2 \in G$, the central character $\gamma$ of $\rho$ has even order. Furthermore, since $\delta \leq \rho \otimes \rho$, it follows that $\delta$ has central character $\gamma^2$. Therefore, $\delta^{i} \rho$ has central character $\gamma^{2i+1}$ for all $i$. Similarly, since $\varphi \leq \delta^{-1} \rho \otimes \rho$, we see that $\varphi$ has central character $\gamma^0$, and hence $\delta^i \varphi$ has central character $\gamma^{2i}$ for all $i$. Therefore, $\delta^i \rho \not \simeq \delta^j \varphi$ for all $i$ and $j$. 
\end{Rem}

Next, we argue that there are arrows $\delta^{-i} \rho \to \delta^{-i} \varphi $. 

\begin{Rem}
With the same notation as before, we compute 
\[(\rho, \varphi \rho) = (\rho \overline{\rho},  \varphi) = (\rho^2, \delta \varphi ) = (\delta, \delta \varphi) + (\delta \lambda, \delta \varphi) + (\delta \varphi, \delta \varphi) = 1. \]
Scaling this by $\delta^{-i}$ shows that we have arrows $\delta^{-i} \rho \to \delta^{-i} \varphi $ in the quiver. 
\end{Rem}

We have obtained the following quiver. 

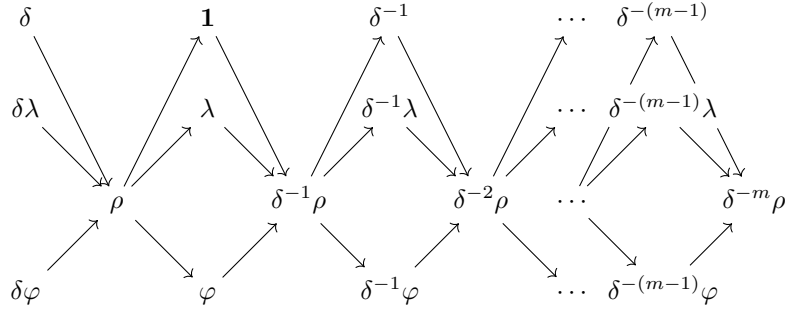
\[
\begin{tikzcd}[column sep = {between origins, 8ex}, row sep = {between origins, 8ex}]
\delta \arrow[rdd]        &                                        & \mathbf{1} \arrow[rdd] &                                                    & \delta^{-1} \arrow[rdd]        &                                                    & \cdots                                   & \delta^{-(m-1)} \arrow[rdd, dash pattern=on 20pt off 19pt]        &                  \\
\delta \lambda \arrow[rd] &                                        & \lambda \arrow[rd]     &                                                    & \delta^{-1} \lambda \arrow[rd] &                                                    & \cdots                                   & \delta^{-(m-1)} \lambda \arrow[rd] &                  \\
                          & \rho \arrow[ruu] \arrow[ru] \arrow[rd] &                        & \delta^{-1} \rho \arrow[ruu] \arrow[ru] \arrow[rd] &                                & \delta^{-2} \rho \arrow[ru] \arrow[ruu] \arrow[rd] & \cdots \arrow[ru] \arrow[ruu, dash pattern=on 25pt off 18pt] \arrow[rd] &                                & \delta^{-m} \rho \\
\delta \varphi \arrow[ru] &                                        & \varphi \arrow[ru]     &                                                    & \delta^{-1} \varphi \arrow[ru] &                                                    & \cdots                                   & \delta^{-(m-1)} \varphi \arrow[ru] &                 
\end{tikzcd}
\]

From here on, the computation is inductive in nature. We have an irreducible $2$-dimensional representation $\varphi$. We know that $\varphi \otimes \rho$ has dimension $4$, and we already know that $\rho$ is a summand in $\varphi \otimes \rho$. Hence, there is another summand $\varphi_2$ of dimension $2$ left. If it is reducible, we obtain two $1$-dimensional characters $\lambda_2$ and $\lambda_3$, and we complete the picture to the following quiver. Again, the quiver is complete since the newly added vertices have no other incoming arrows for dimension reasons, and by \Cref{Pro: Quiver is connected} we know that we are done. 

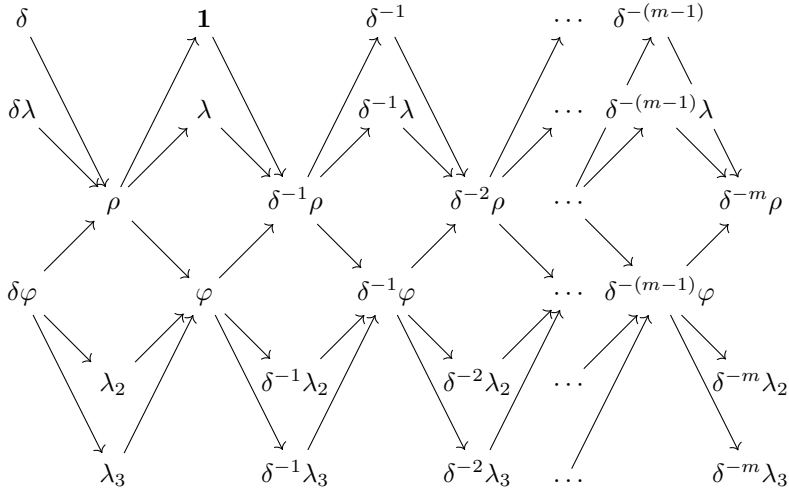
\[
\begin{tikzcd}[column sep = {between origins, 8ex}, row sep = {between origins, 8ex}]
\delta \arrow[rdd]                               &                                        & \mathbf{1} \arrow[rdd]                    &                                                    & \delta^{-1} \arrow[rdd]                               &                                                    & \cdots                                   & \delta^{-(m-1)} \arrow[rdd, dash pattern=on 20pt off 19pt]                               &                       \\
\delta \lambda \arrow[rd]                        &                                        & \lambda \arrow[rd]                        &                                                    & \delta^{-1} \lambda \arrow[rd]                        &                                                    & \cdots                                   & \delta^{-(m-1)} \lambda \arrow[rd]                        &                       \\
                                                 & \rho \arrow[ruu] \arrow[ru] \arrow[rd] &                                           & \delta^{-1} \rho \arrow[ruu] \arrow[ru] \arrow[rd] &                                                       & \delta^{-2} \rho \arrow[ru] \arrow[ruu] \arrow[rd] & \cdots \arrow[ru] \arrow[ruu, dash pattern=on 25pt off 18pt] \arrow[rd] &                                                       & \delta^{-m} \rho      \\
\delta \varphi \arrow[ru] \arrow[rd] \arrow[rdd] &                                        & \varphi \arrow[ru] \arrow[rd] \arrow[rdd] &                                                    & \delta^{-1} \varphi \arrow[ru] \arrow[rd] \arrow[rdd] &                                                    & \cdots                                   & \delta^{-(m-1)} \varphi \arrow[ru] \arrow[rd] \arrow[rdd] &                       \\
                                                 & \lambda_2 \arrow[ru]                   &                                           & \delta^{-1} \lambda_2 \arrow[ru]                   &                                                       & \delta^{-2} \lambda_2 \arrow[ru]                   & \cdots  \arrow[ru]                       &                                                       & \delta^{-m} \lambda_2 \\
                                                 & \lambda_3 \arrow[ruu]                  &                                           & \delta^{-1} \lambda_3 \arrow[ruu]                  &                                                       & \delta^{-2} \lambda_3 \arrow[ruu]                  & \dots  \arrow[ruu]                       &                                                       & \delta^{-m} \lambda_3
\end{tikzcd}
\]

If instead $\varphi_2$ is irreducible, we obtain another row in a square shaped quiver. 

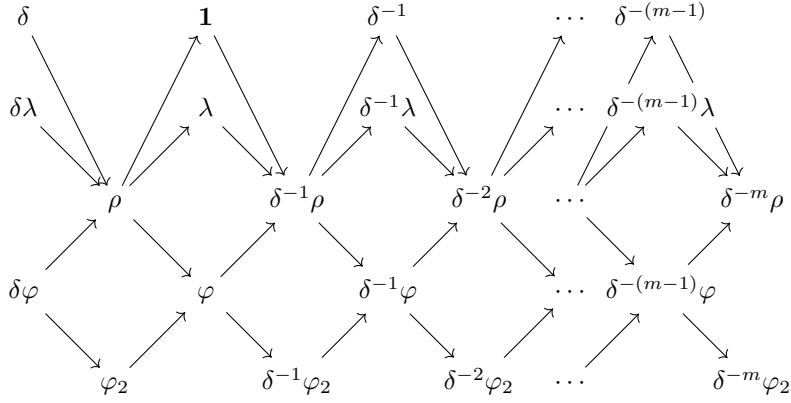
\[ 
\begin{tikzcd}[column sep = {between origins, 8ex}, row sep = {between origins, 8ex}]
\delta \arrow[rdd]                   &                                        & \mathbf{1} \arrow[rdd]        &                                                    & \delta^{-1} \arrow[rdd]                   &                                                    & \cdots                                   & \delta^{-(m-1)} \arrow[rdd, dash pattern=on 20pt off 19pt]                   &                       \\
\delta \lambda \arrow[rd]            &                                        & \lambda \arrow[rd]            &                                                    & \delta^{-1} \lambda \arrow[rd]            &                                                    & \cdots                                   & \delta^{-(m-1)} \lambda \arrow[rd]            &                       \\
                                     & \rho \arrow[ruu] \arrow[ru] \arrow[rd] &                               & \delta^{-1} \rho \arrow[ruu] \arrow[ru] \arrow[rd] &                                           & \delta^{-2} \rho \arrow[ru] \arrow[ruu] \arrow[rd] & \cdots \arrow[ru] \arrow[ruu, dash pattern=on 25pt off 18pt] \arrow[rd] &                                           & \delta^{-m} \rho      \\
\delta \varphi \arrow[ru] \arrow[rd] &                                        & \varphi \arrow[ru] \arrow[rd] &                                                    & \delta^{-1} \varphi \arrow[ru] \arrow[rd] &                                                    & \cdots                                   & \delta^{-(m-1)} \varphi \arrow[ru] \arrow[rd] &                       \\
                                     & \varphi_2 \arrow[ru]                   &                               & \delta^{-1} \varphi_2 \arrow[ru]                   &                                           & \delta^{-2} \varphi_2 \arrow[ru]                   & \cdots \arrow[ru]                        &                                           & \delta^{-m} \varphi_2
\end{tikzcd}
\]

It is important here to point out that $\varphi_2$ can not be isomorphic to any $\delta^i \rho$, so that we indeed obtain a \emph{new} row. This is the base case of an induction showing that each row is new. We first show that $\varphi_2 \not \simeq \rho$. 

\begin{Rem}
    With the above setup, decompose $\varphi \otimes  \rho \simeq \rho \oplus \varphi_2$, and assume that $\varphi_2$ is irreducible. Then we compute that 
    \begin{align*}
        ( \rho, \varphi_2) + ( \rho, \rho) &= ( \rho, \varphi_2 + \rho) = ( \rho, \varphi \cdot \rho) =  (\rho \cdot \overline{\rho}, \varphi) \\
        &= ( \rho \cdot \rho , \delta \varphi) =(\delta + \delta \lambda + \delta \varphi, \delta \varphi) \\
        &= (\delta, \delta \varphi) + ( \delta \lambda , \delta \varphi) + ( \delta \varphi, \delta \varphi) = 0 + 0 + ( \delta \varphi, \delta \varphi) = 1.
    \end{align*}
    Since $(\rho, \rho) = 1$, it follows that $( \rho, \varphi_2) = 0$. 
\end{Rem}

Next, to complete the base case, we show that $\varphi_2$ can not be isomorphic to any $\delta^{-i} \rho$. This is more involved, and we make use \Cref{Theo: Minimal dihedral groups}. 

\begin{Rem}
    With the above setup, decompose $\varphi \otimes  \rho \simeq \rho \oplus \varphi_2$, and assume that $\varphi_2$ is irreducible. Furthermore, assume that $n \geq 3$. Then, we perform a direct computation with the minimal groups from \Cref{Theo: Minimal dihedral groups}. First, let $n$ be even, and let $j$ be $0$ or $1$, so that we can cover both kinds of minimal groups, and suppose that $G$ is a minimal group. Let $g \in G$ be the diagonal generator. The character $\chi_\rho$ of $\rho$ takes value $(\epsilon_{2^{k+1}})^{j} (\epsilon_{2n} + \epsilon_{2n}^{-1})$ on $g$. Taking the square gives us the character of $\rho \otimes \rho$, and with the knowledge of the character of $\delta$ and of $\delta \lambda$, it is easy to see that 
    \begin{align*}
        \chi_{ \rho \otimes \rho}(g) &= (\epsilon_{2^{k+1}}^2)^{j} (\epsilon_{2n} + \epsilon_{2n}^{-1})^2 \\
        &= (\epsilon_{2^{k+1}}^2)^{j} + (\epsilon_{2^{k+1}}^2)^{j} + (\epsilon_{2^{k+1}}^2)^{j} (\epsilon_{2n}^2 + \epsilon_{2n}^{-2})\\
        &= \chi_{\delta}(g) + \chi_{\delta \lambda}(g) + \chi_{\delta \varphi}
    \end{align*} 
    Tensoring $\varphi$ again with $\rho$ gives us the character value 
    \begin{align*}
        \chi_{ \varphi \otimes \rho}(g) &=  (\epsilon_{2^{k+1}}^2)^{j} (\epsilon_{2n}^2 + \epsilon_{2n}^{-2}) (\epsilon_{2n} + \epsilon_{2n}^{-1}) \\
        &=  (\epsilon_{2^{k+1}}^2)^{j} (\epsilon_{2n}^3 + \epsilon_{2n}^{-3}) + (\epsilon_{2^{k+1}}^2)^{j} (\epsilon_{2n} + \epsilon_{2n}^{-1})
    \end{align*}
    and hence $\chi_{\varphi_2}(g) = (\epsilon_{2^{k+1}}^2)^{j} (\epsilon_{2n}^3 + \epsilon_{2n}^{-3})$. It now is easy to check that this value is different from the values of all characters of all $ \delta^{-i} \rho $. 
    
    If $n$ is odd, the same computation works by replacing $\epsilon_{2n}$ with $\epsilon_n$. Finally, if $G$ is not minimal, it contains a minimal group, so the fact that the characters for $\delta^{ -i}\rho$ and $\varphi_2$ are different on the subgroup shows our claim. 

    Finally, let $n=2$. Then one can show using \Cref{Theo: Minimal dihedral groups} that $\operatorname{Sym}^2(\rho)$ already decomposes into $3$ irreducible summands, so this case does not actually occur here. 
    
    We also point out here that with more computations, one can reduce $n$ to be $4$ under the assumption $\delta^{ -i}\rho \simeq \varphi_2$, so even fewer cases need to be checked by hand.  
\end{Rem}

We continue the inductive computation of the quiver, now decomposing $\rho \otimes \delta^{-1} \varphi_2 \simeq \varphi \oplus \varphi_3$ and adding another row to the quiver.

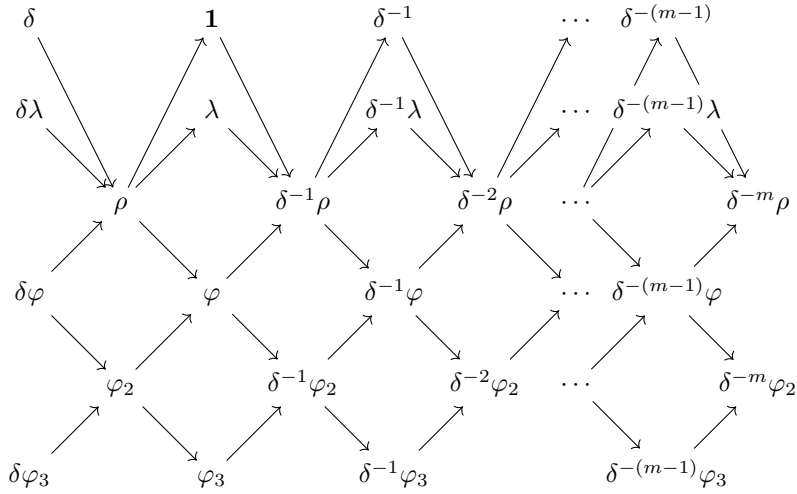
\[ 
\begin{tikzcd}[column sep = {between origins, 8ex}, row sep = {between origins, 8ex}]
\delta \arrow[rdd]                   &                                        & \mathbf{1} \arrow[rdd]        &                                                    & \delta^{-1} \arrow[rdd]                   &                                                    & \cdots                                   & \delta^{-(m-1)} \arrow[rdd, dash pattern=on 20pt off 19pt]                   &                       \\
\delta \lambda \arrow[rd]            &                                        & \lambda \arrow[rd]            &                                                    & \delta^{-1} \lambda \arrow[rd]            &                                                    & \cdots                                   & \delta^{-(m-1)} \lambda \arrow[rd]            &                       \\
                                     & \rho \arrow[ruu] \arrow[ru] \arrow[rd] &                               & \delta^{-1} \rho \arrow[ruu] \arrow[ru] \arrow[rd] &                                           & \delta^{-2} \rho \arrow[ru] \arrow[ruu] \arrow[rd] & \cdots \arrow[ru] \arrow[ruu, dash pattern=on 25pt off 18pt] \arrow[rd] &                                           & \delta^{-m} \rho      \\
\delta \varphi \arrow[ru] \arrow[rd] &                                        & \varphi \arrow[ru] \arrow[rd] &                                                    & \delta^{-1} \varphi \arrow[ru] \arrow[rd] &                                                    & \cdots                                   & \delta^{-(m-1)} \varphi \arrow[ru] \arrow[rd] &                       \\
                                     & \varphi_2 \arrow[ru] \arrow[rd]        &                               & \delta^{-1} \varphi_2 \arrow[ru] \arrow[rd]        &                                           & \delta^{-2} \varphi_2 \arrow[ru]                   & \cdots \arrow[ru] \arrow[rd]             &                                           & \delta^{-m} \varphi_2 \\
\delta \varphi_3 \arrow[ru]          &                                        & \varphi_3 \arrow[ru]          &                                                    & \delta^{-1} \varphi_3 \arrow[ru]          &                                                    &                                          & \delta^{-(m-1)} \varphi_3 \arrow[ru]          &                      
\end{tikzcd}
\]

\begin{Rem}
    To see that in each step a new row is added, we make two observations. Suppose that $\varphi_i$ has been newly added, and write $\varphi_0 = \rho$ and $\varphi_1 = \varphi$. First of all, note that each representation in the quiver has a central character which is a power of $\gamma$, the central character of $\rho$. Furthermore, representations ``in the same column'' have the same central characters. More precisely, the column containing $\rho$ has central character $\gamma$, and $\delta$ has central character $\gamma^{2}$. One checks easily that $\varphi_2$, if it exists, has the same central character as $\rho$, and so on. The column containing $\mathbf{1}$ has trivial central character, the column containing $\delta^{-1} \rho$ has central character $\gamma^{-1}$ and so on. This means that all representations $\delta^{j} \varphi_i$ for even $i$ have central character an odd power of $\gamma$, while the representations $\delta^{j} \varphi_i$ for odd $i$ have central characters which are even powers of $\gamma$. In particular, this means that a newly added $\varphi_i$ can not be isomorphic to any $ \delta^{-k} \varphi_j$ if $j$ has different parity than $i$. 
    The second observation proceeds by induction. The base case was already dealt with when we showed that $\varphi_0 = \rho$ is different from $\delta^{-k} \varphi_2$ for all $k$. We spell out the case when the newly added row containing $\varphi_i$ has even index $i$. In the odd case, all computations need to be performed with an extra factor $\delta^{-1}$. We consider whether the newly added $\varphi_i$ can be isomorphic to some $\delta^{-k} \varphi_j$ for some $k$ and $0 < j < i-1$ even. We will later see that it suffices to check this for $k=0$, but for now we proceed more generally. We compute 
    \begin{align*}
        (\delta^{-k} \varphi_j, \varphi_i)  + (\delta^{-k} \varphi_j, \varphi_{i-2}) &= (\delta^{-k} \varphi_j, \varphi_i + \varphi_{i-2}) \\
        &= (\delta^{-k} \varphi_j, \varphi_{i-1} \rho) = (\delta^{-k} \varphi_j \rho,  \delta \varphi_{i-1}) \\
        &= (\delta^{-k} \delta( \varphi_{j-1} + \varphi_{j+1}), \delta \varphi_{i-1}) \\
        &= (\delta^{-k} \delta \varphi_{j-1}, \delta \varphi_{i-1})  + (\delta^{-k} \delta \varphi_{j+1}, \delta \varphi_{i-1}) \\
        &= (\delta^{-k}  \varphi_{j-1},  \varphi_{i-1})  + (\delta^{-k}  \varphi_{j+1},  \varphi_{i-1}) .
    \end{align*}
    By induction, we have that $(\delta^{-k} \varphi_{j-1},  \varphi_{i-1}) = 0$. If also $(\delta^{-k}  \varphi_{j+1},  \varphi_{i-1}) = 0$, we are done. Otherwise, we have $(\delta^{-k}  \varphi_{j+1},  \varphi_{i-1}) = 1$, which by induction can only happen if $j+1 = i-1$. But then, mimicking the computation where we showed that $\delta^{-l} \rho \not \simeq \varphi_2$ for any $l$, we find that 
    \[ (\delta^{-k} \varphi_{j+1},  \varphi_{i-1}) =  (\delta^{-k} \varphi_j, \varphi_{i-2}),  \]
    and our claim follows from 
    \[ (\delta^{-k} \varphi_j, \varphi_i) = 0. \]
\end{Rem} 

Eventually, the newly obtained $\varphi_l$ is reducible, decomposing into $\lambda_2$ and $\lambda_3$. The same computation as before shows that $\lambda_2$ and $\lambda_3$ are two new $1$-dimensional representations, leading to the final row of the quiver. Again by \Cref{Pro: Quiver is connected}, using that we have added all incoming arrows in each step, we know that we have found the complete quiver. In the picture below, $l$ is odd. 

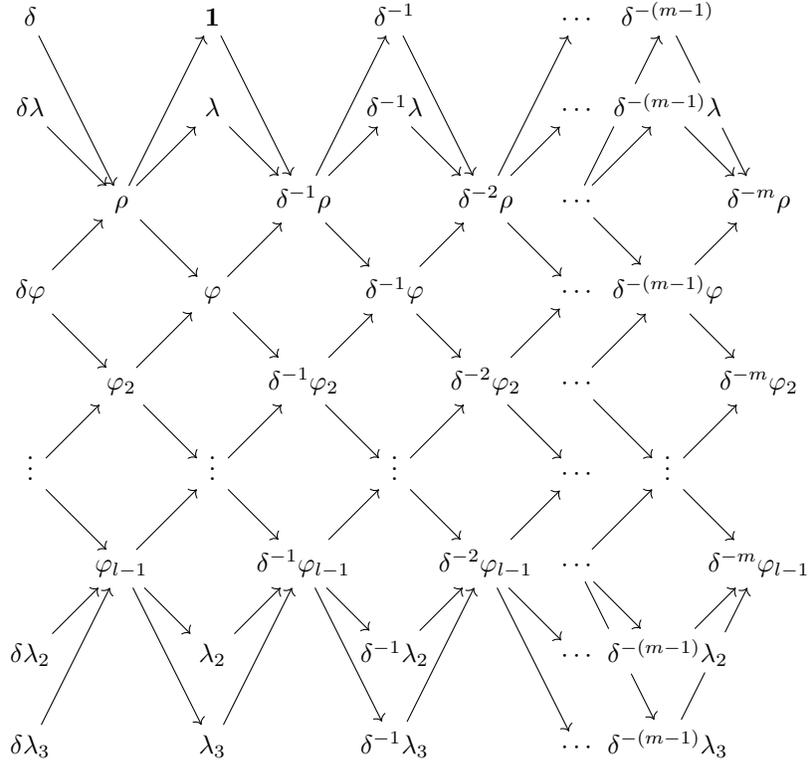
\[
\begin{tikzcd}[column sep = {between origins, 8ex}, row sep = {between origins, 8ex}]
\delta \arrow[rdd]                   &                                                 & \mathbf{1} \arrow[rdd]        &                                                              & \delta^{-1} \arrow[rdd]                   &                                                              & \cdots                                   & \delta^{-(m-1)} \arrow[rdd, dash pattern=on 20pt off 19pt]                 &                            \\
\delta \lambda \arrow[rd]            &                                                 & \lambda \arrow[rd]            &                                                              & \delta^{-1} \lambda \arrow[rd]            &                                                              & \cdots                                   & \delta^{-(m-1)} \lambda \arrow[rd]            &                            \\
                                     & \rho \arrow[ruu] \arrow[ru] \arrow[rd]          &                               & \delta^{-1} \rho \arrow[ruu] \arrow[ru] \arrow[rd]           &                                           & \delta^{-2} \rho \arrow[ru] \arrow[ruu] \arrow[rd]           & \cdots \arrow[ru] \arrow[ruu, dash pattern=on 25pt off 18pt] \arrow[rd] &                                           & \delta^{-m} \rho           \\
\delta \varphi \arrow[ru] \arrow[rd] &                                                 & \varphi \arrow[ru] \arrow[rd] &                                                              & \delta^{-1} \varphi \arrow[ru] \arrow[rd] &                                                              & \cdots                                   & \delta^{-(m-1)} \varphi \arrow[ru] \arrow[rd] &                            \\
                                     & \varphi_2 \arrow[ru] \arrow[rd]                 &                               & \delta^{-1} \varphi_2 \arrow[ru] \arrow[rd]                  &                                           & \delta^{-2} \varphi_2 \arrow[ru] \arrow[rd]                  & \cdots \arrow[ru] \arrow[rd]             &                                           & \delta^{-m} \varphi_2      \\
\vdots \arrow[ru] \arrow[rd]         &                                                 & \vdots \arrow[ru] \arrow[rd]  &                                                              & \vdots \arrow[ru] \arrow[rd]              &                                                              & \cdots                                   & \vdots \arrow[ru] \arrow[rd]              &                            \\
                                     & \varphi_{l-1} \arrow[rd] \arrow[rdd] \arrow[ru] &                               & \delta^{-1}  \varphi_{l-1} \arrow[rd] \arrow[rdd] \arrow[ru] &                                           & \delta^{-2}  \varphi_{l-1} \arrow[ru] \arrow[rd] \arrow[rdd] & \cdots \arrow[ru] \arrow[rd] \arrow[rdd, dash pattern=on 20pt off 19pt]&                                           & \delta^{-m} \varphi_{l-1}  \\
\delta \lambda_2 \arrow[ru]          &                                                 & \lambda_2 \arrow[ru]          &                                                              & \delta^{-1} \lambda_2 \arrow[ru]          &                                                              & \cdots                                   & \delta^{-(m-1)} \lambda_2 \arrow[ru]          &                            \\
\delta \lambda_3 \arrow[ruu]         &                                                 & \lambda_3 \arrow[ruu]         &                                                              & \delta^{-1} \lambda_3 \arrow[ruu]         &                                                              & \cdots                                   & \delta^{-(m-1)} \lambda_3 \arrow[ruu,dash pattern=on 20pt off 17pt]         &                           
\end{tikzcd}
\]

It remains to show that all rows have ``the same size''. However, this is somewhat ill defined since one can read the rows of $1$-dimensional representations either as $\delta$-orbits, in which case there might be only $2$, or as the rows as we drew them, in which case there are $4$. This is governed by the order of $\delta$. For our purposes of finding cuts, it suffices that $\delta$ has the same size of orbits on all the $2$-dimensional irreducible representations. 

\begin{Rem}
    We show that all $\delta$-orbits have the same size on the $2$-dimensional irreducible representations. Again, we set $\varphi_0 = \rho$, and for convenience we choose the index $i$ to be even.
    Denote by $m \geq 1$ the smallest value such that $ \delta^m \varphi_i = \varphi_i$, and by $k \geq 1$ the smallest value such that $ \delta^k \varphi_{i+1} = \varphi_{i+1}$. We compute that
    \begin{align*}
        1 &= (\varphi_i , \varphi_{i+1} \cdot \rho) = (\varphi_i , \delta^k \varphi_{i+1} \cdot \rho) \\
        &= (\varphi_i , \delta^k (\varphi_{i} + \varphi_{i+2})) = (\varphi_i , \delta^k \varphi_{i}) + (\varphi_i , \delta^k  \varphi_{i+2}) \\
        &= (\varphi_i , \delta^k \varphi_{i}) + 0
    \end{align*}
    and hence $m \mid k$. Similarly, computing 
    \begin{align*}
        1 &= (\varphi_{i+1} , \delta^{-1} \varphi_{i} \cdot \rho) = (\varphi_{i+1} , \delta^{-1} \delta^m \varphi_{i} \cdot \rho) \\
        &= (\varphi_{i+1} ,  \delta^m (\varphi_{i+1} + \varphi_{i-1}) ) = (\varphi_{i+1} ,  \delta^m \varphi_{i+1}) + (\varphi_{i+1} + \delta^m \varphi_{i-1} ) \\
        &= (\varphi_{i+1} ,  \delta^m \varphi_{i+1}) + 0,
    \end{align*}
    shows that $k \mid m$, so we find $m=k$. Note that in the above computation, the possible $\varphi_{-1}$ needs to be interpreted as the sum of two $1$-dimensional representations. Finally, note that the same computation for $\delta$ instead of $\varphi$ also shows that $\delta^m \simeq \mathbf{1}$ or $\delta^m \simeq \lambda$, so there are $2$ or $4$ $\delta$-orbits on the set of $1$-dimensional representations. 
\end{Rem}

Thus, adding in the arrows from embedding $G $ into $\SL_3(\Bbbk)$, we obtain two possible quivers, depending on the order of $\delta$ as before. We draw both of them in the picture below, where one needs to choose either the blue or the green arrows. 

We now prove that a cut exists on this quiver as long as $\delta$ is not trivial. To see this, we note that the quiver looks ``almost'' like a triangular grid on a cylinder, except for the forking at the top and bottom rows. We therefore can quotient this quiver onto another quiver that looks like a triangular grid on the torus, which is precisely what we treated in type (A).

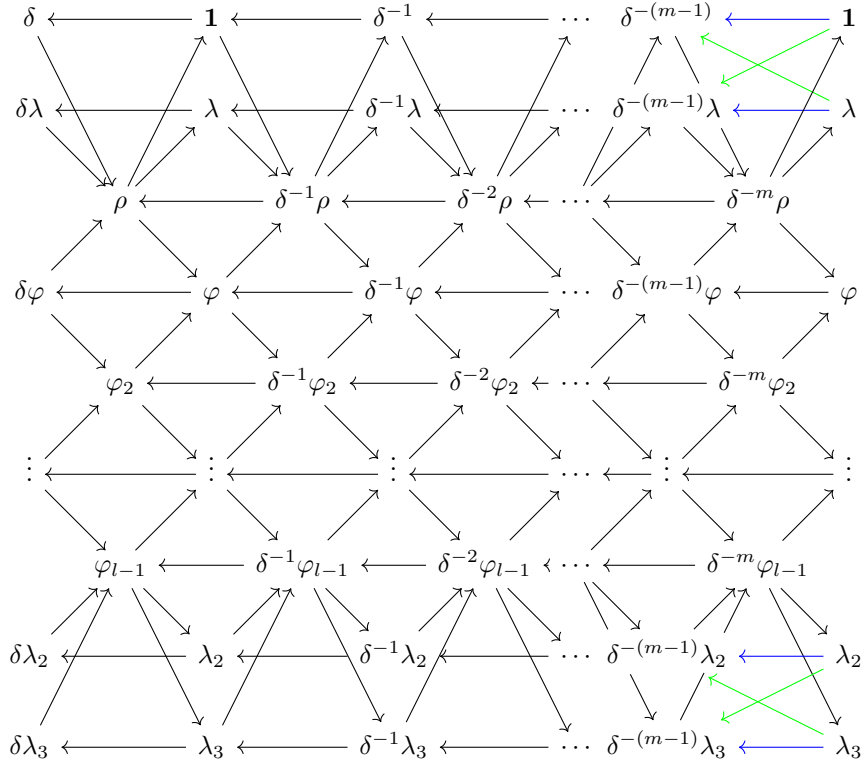
\[
\begin{tikzcd}[column sep = {between origins, 8ex}, row sep = {between origins, 8ex}]
\delta \arrow[rdd]                   &                                                 & \mathbf{1} \arrow[rdd] \arrow[ll]        &                                                                         & \delta^{-1} \arrow[rdd] \arrow[ll]                   &                                                                         & \cdots \arrow[ll]                                  & \delta^{-(m-1)} \arrow[rdd, dash pattern=on 20pt off 19pt]                  &                                                                         & \mathbf{1} \arrow[ll, color=blue] \arrow[lld, color=green] \\
\delta \lambda \arrow[rd]            &                                                 & \lambda \arrow[rd] \arrow[ll]            &                                                                         & \delta^{-1} \lambda \arrow[rd] \arrow[ll]            &                                                                         & \cdots \arrow[ll]                                  & \delta^{-(m-1)} \lambda \arrow[rd]         &                                                                         & \lambda \arrow[ll, color=blue] \arrow[llu, color=green]   \\
                                     & \rho \arrow[ruu] \arrow[ru] \arrow[rd]          &                                          & \delta^{-1} \rho \arrow[ruu] \arrow[ru] \arrow[rd] \arrow[ll]           &                                                      & \delta^{-2} \rho \arrow[ru] \arrow[ruu] \arrow[rd] \arrow[ll]           & \cdots \arrow[ru] \arrow[ruu, dash pattern=on 25pt off 18pt] \arrow[rd] \arrow[l] &                                                         & \delta^{-m} \rho \arrow[ll] \arrow[ruu] \arrow[ru] \arrow[rd]           &                       \\
\delta \varphi \arrow[ru] \arrow[rd] &                                                 & \varphi \arrow[ru] \arrow[rd] \arrow[ll] &                                                                         & \delta^{-1} \varphi \arrow[ru] \arrow[rd] \arrow[ll] &                                                                         & \cdots \arrow[ll]                                  & \delta^{-(m-1)} \varphi \arrow[ru] \arrow[rd] &                                                                         & \varphi \arrow[ll]    \\
                                     & \varphi_2 \arrow[ru] \arrow[rd]                 &                                          & \delta^{-1} \varphi_2 \arrow[ru] \arrow[rd] \arrow[ll]                  &                                                      & \delta^{-2} \varphi_2 \arrow[ru] \arrow[rd] \arrow[ll]                  & \cdots \arrow[ru] \arrow[rd] \arrow[l]             &                                                         & \delta^{-m} \varphi_2 \arrow[ll] \arrow[ru] \arrow[rd]                  &                       \\
\vdots \arrow[ru] \arrow[rd]         &                                                 & \vdots \arrow[ru] \arrow[rd] \arrow[ll]  &                                                                         & \vdots \arrow[ru] \arrow[rd] \arrow[ll]              &                                                                         & \cdots \arrow[ll]                                  & \vdots \arrow[ru] \arrow[rd] \arrow[l]                  &                                                                         & \vdots  \arrow[ll]    \\
                                     & \varphi_{l-1} \arrow[rd] \arrow[rdd] \arrow[ru] &                                          & \delta^{-1}  \varphi_{l-1} \arrow[rd] \arrow[rdd] \arrow[ru] \arrow[ll] &                                                      & \delta^{-2}  \varphi_{l-1} \arrow[ru] \arrow[rd] \arrow[rdd] \arrow[ll] & \cdots \arrow[ru] \arrow[rd] \arrow[rdd, dash pattern=on 20pt off 19pt] \arrow[l] &                                                         & \delta^{-m} \varphi_{l-1}  \arrow[ll] \arrow[rd] \arrow[rdd] \arrow[ru] &                       \\
\delta \lambda_2 \arrow[ru]          &                                                 & \lambda_2 \arrow[ru] \arrow[ll]          &                                                                         & \delta^{-1} \lambda_2 \arrow[ru] \arrow[ll]          &                                                                         & \cdots \arrow[ll]                                  & \delta^{-(m-1)} \lambda_2 \arrow[ru]           &                                                                         & \lambda_2 \arrow[ll, color=blue] \arrow[lld, color=green]  \\
\delta \lambda_3 \arrow[ruu]         &                                                 & \lambda_3 \arrow[ruu] \arrow[ll]         &                                                                         & \delta^{-1} \lambda_3 \arrow[ruu] \arrow[ll]         &                                                                         & \cdots \arrow[ll]                                  & \delta^{-(m-1)} \lambda_3 \arrow[ruu,dash pattern=on 20pt off 17pt]         &                                                                         & \lambda_3 \arrow[ll, color=blue] \arrow[llu, color=green]
\end{tikzcd}
\]

\begin{Con}
    Let $Q_G$ be the McKay quiver of a central extension $G \leq \GL_2(\Bbbk)$ of a dihedral group $D_{2n}$. Then we denote by $T(Q_G)$ the quiver obtained by identifying the $1$-dimensional representations in equivalence classes of $4$ representations each. More precisely, we identify $\delta^{-i}$ with $\delta^{-i} \lambda$ and with $\delta^{-i} \lambda_2$ and with $\delta^{-i} \lambda_3$ for $0 \leq i < m$. The arrows between pairs of identified vertices are identified as well. We denote by $q \colon Q_G \to T(Q_G)$ the quotient map. The quiver $T(Q_G)$ naturally embeds on the torus. 
\end{Con}

Next, we note that this quiver arises as a quotient of the infinite triangular quiver $\hat{Q}$ from \Cref{Sec: type (A)}. 

\begin{Con}
    Let $Q_G$ be as above. Then $T(Q_G)$ is isomorphic to a quotient of $\hat{Q}$ by the action of a cofinite sublattice $L \leq \mathbb{Z}^2$ as in the type (A) case in \Cref{Sec: type (A)}. More precisely, we fix $l' = l+1$ and $m$ as before, and we consider the lattice $L = \langle  m e_1, \lceil \frac{l'}{2} \rceil e_1 + l' e_2 \rangle \leq \mathbb{Z}^2$. We claim that $T(Q_G)$ is isomorphic to $\hat{Q}/L$. For ease of notation we write $\varphi_0 = \rho$, $\varphi_1 = \varphi$ and $\varphi_{-1} = \mathbf{1}$, so that every vertex in $T(Q_G)$ is uniquely identified as $q(\delta^i \varphi_j) $ for the pair $(i,j) \in \{ 0 , \ldots, m-1\} \times \{ -1, \ldots l-1  \}$. In particular, we note that in order to move from $\varphi_{-1}$ to $\varphi_j$, we can follow a ``zig-zag'' path by taking the reverse arrow $\varphi_{-1} \leftarrow \varphi_0$, and the normal arrow $\varphi_0 \to \varphi_1$, followed by a reverse arrow $\varphi_1 \leftarrow \varphi_2$, etc. This means we take $\lceil \frac{j+1}{2} \rceil$ many reverse arrows and $\lfloor \frac{j+1}{2} \rfloor$ many ordinary arrows. This gives rise to the following map
    \begin{align*}
      \iota \colon  T(Q_G)_0 &\to (\hat{Q}/L)_0, \\
        q(\delta^i \varphi_{j}) &\mapsto i \cdot e_1 + \lceil \frac{j+1}{2} \rceil \cdot (-e_3) + \lfloor \frac{j+1}{2} \rfloor \cdot e_2 +  L, 
    \end{align*}
    where we chose coordinates so that the action of $e_1$ on $\hat{Q}/L$ corresponds to the action of $\delta$ on $T(Q_G)$, and the ``zig-zag'' is recovered by the reverse arrows corresponding to $-e_3$ and the ordinary arrows corresponding to $e_2$. Since $-e_3 = e_1 + e_2$, the map simplifies to 
    \[ q(\delta^i \varphi_{j}) \mapsto (i + \lceil \frac{j+1}{2} \rceil) \cdot e_1 + (j+1) \cdot e_2 +  L. \]
    To check that the map is a well-defined isomorphism, we consider when the difference 
    \begin{align*}
        d &=  \iota(q(\delta^i \varphi_j)) -  \iota(q(\delta^{i'} \varphi_{j'})) \\
        &= (i + \lceil \frac{j+1}{2} \rceil) \cdot e_1 + (j+1) \cdot e_2 - (i' + \lceil \frac{j'+1}{2} \rceil) \cdot e_1 - (j'+1) \cdot e_2 +  L \\
        &= (i - i' +  \lceil \frac{j+1}{2} \rceil - \lceil \frac{j'+1}{2} \rceil) \cdot e_1 + ((j+1) - (j'+1)) \cdot e_2 + L \\
        &= (i - i' +  \lceil \frac{j+1}{2} \rceil - \lceil \frac{j'+1}{2} \rceil) \cdot e_1 + (j -j') \cdot e_2 + L
    \end{align*}
    is $0 + L$. Since $j, j' \in \{ -1, \ldots, l-1\}$, it follows that $|j-j'| < l+1 = l' $. Hence $d = 0+ L$ only if $j = j'$. Assume that $j=j'$, then 
    \[ i - i' +  \lceil \frac{j+1}{2} \rceil - \lceil \frac{j'+1}{2} \rceil = i-i'. \]
    Since $i , i' \in \{0, \ldots, m-1\} $, we have $|i-i'| < m$, and therefore $m \mid i - i'$ only if $i-i' = 0$. Thus, we have $d= 0$ if and only if $i=i'$ and $j=j'$. 
\end{Con}

Along this identification, we can transport cuts. 

\begin{Pro}\label{Pro: (B) dihedral has cut if abelian has}
    Let $G \leq \SL_3(\Bbbk)$ be a group of type (B) that is a central extension of a dihedral group. Then $Q_G$ admits a cut if $T(Q_G) \simeq \hat{Q}/L$ admits a cut. 
\end{Pro}

\begin{proof}
    Let $C' \subseteq (\hat{Q}/L)_1$ be a cut. We claim that $C = \{ a \in (Q_G)_1 \mid \iota(q(a)) \in C'  \}$ is a cut for $Q_G$. By \Cref{Cor: deter cycles cut and acyclic suffices} it is enough to show that each $3$-cycle in $Q_G$ containing exactly one arrow coming from the action of $\delta$ is cut exactly once, and that $Q_G - C'$ is acylic. First, consider which $3$-cycles in $Q_G$ are identified in $\hat{Q}/L$. The only identifications involve vertices corresponding to $1$-dimensional representations, so we immediately see that every $3$-cycle containing exactly one $\delta$-arrow and not involving any of those vertices is cut exactly once. The remaining $3$-cycles arise from the following part of the McKay quiver, as well as the mirrored version containing $\lambda_2$ and $\lambda_3$. 
    
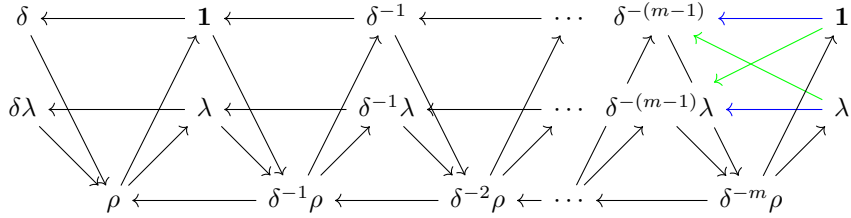
\[
\begin{tikzcd}[column sep = {between origins, 8ex}, row sep = {between origins, 8ex}]
\delta \arrow[rdd]                   &                                                 & \mathbf{1} \arrow[rdd] \arrow[ll]        &                                                                         & \delta^{-1} \arrow[rdd] \arrow[ll]                   &                                                                         & \cdots \arrow[ll]                                  & \delta^{-(m-1)} \arrow[rdd, dash pattern=on 20pt off 19pt]                  &                                                                         & \mathbf{1} \arrow[ll, color=blue] \arrow[lld, color=green] \\
\delta \lambda \arrow[rd]            &                                                 & \lambda \arrow[rd] \arrow[ll]            &                                                                         & \delta^{-1} \lambda \arrow[rd] \arrow[ll]            &                                                                         & \cdots \arrow[ll]                                  & \delta^{-(m-1)} \lambda \arrow[rd]         &                                                                         & \lambda \arrow[ll, color=blue] \arrow[llu, color=green]   \\
                                     & \rho \arrow[ruu] \arrow[ru]          &                                          & \delta^{-1} \rho \arrow[ruu] \arrow[ru] \arrow[ll]           &                                                      & \delta^{-2} \rho \arrow[ru] \arrow[ruu] \arrow[ll]           & \cdots \arrow[ru] \arrow[ruu, dash pattern=on 25pt off 18pt]  \arrow[l] &                                                         & \delta^{-m} \rho \arrow[ll] \arrow[ruu] \arrow[ru]         &                      
\end{tikzcd}
\]
    We see that if the image of a $3$-cycle has been cut exactly once, then so have the $4$ cycles that got mapped to it. Finally, we need to show that $Q_G - C$ is acyclic. But this is immediate, since if it wasn't, then the image of a cycle in $\hat{Q}/L - C'$ would be a cycle.  
\end{proof}

We arrive at the following theorem, where the case when $m=1$ is precisely when loops occur in the quiver.   

\begin{Theo}\label{Theo: Central Dihedral class}
    Let $G \leq \SL_3(\Bbbk)$ be a group of type (B) that is a central extension of a dihedral group, and denote by $\rho \colon G \to \GL_2(\Bbbk)$ the defining representation. Then $R \ast G \simeq \Bbbk Q_G/I$ admits a $3$-preprojective cut if and only if $\rho \not \simeq \det(\rho) \otimes \rho$. 
\end{Theo}

\begin{proof}
    If $\rho \simeq \det(\rho) \otimes  \rho$, we have that $m=1$. This means that $Q_G$ has loop given by determinant arrows, so $\Bbbk Q_G/I$ does not admit a cut by \Cref{Pro: Determinant loops prevent cuts}. Now suppose that $\rho \not \simeq \det(\rho) \otimes  \rho$, meaning that $m> 1$. The matrix $B$ defining the lattice $L$ with $T(Q_G) \simeq \hat{Q}/L$ is given by  $B=\begin{pmatrix}
        m&\lceil\frac{l'}{2}\rceil\\
        0&l'
    \end{pmatrix}$. Using \Cref{Theo: SL3 type (A) classification}, one can easily verify that $(\gamma_1,\gamma_2,\gamma_3)=(l',\lceil\frac{l'}{2}\rceil(m-1),\lfloor\frac{l'}{2}\rfloor(m-1))$ is a type of a cut. Hence $\hat{Q}/L$ admits a cut and by \Cref{Pro: (B) dihedral has cut if abelian has} so does $Q_G$.  
\end{proof}

Combining this with the non-central case, we arrive at the following theorem.

\begin{Theo}\label{Theo: Type (B) dihedral classification}
    Let $G \leq \SL_3(\Bbbk)$ be a group of type (B) that is a central extension of a dihedral group. Then $R \ast G \simeq \Bbbk Q_G/I$ admits a $3$-preprojective cut if and only if the defining representation $\rho \colon G \hookrightarrow \GL_2(\Bbbk)$ is not self-dual.
\end{Theo}

To conclude the dihedral case, we discuss what the self-dual case $m=1$ means for the involved groups. While we recover the dihedral and binary dihedral groups, we are not aware of a conceptual reason for this that does not involve case analysis.  

\begin{Pro}
    Let $G \leq \SL_3(\Bbbk)$ be a group of type (B) that is a central extension of a dihedral group $D_{2n}$. Denote by $\rho \colon G \to \GL_2(\Bbbk)$ the defining representation. If $\delta \otimes  \rho \simeq \rho$, then $G$ is a dihedral group, or a binary dihedral group. 
\end{Pro}

\begin{proof}
    Recall that $\delta \otimes \rho \simeq \rho$ is, by \Cref{Cor: self-duality}, equivalent to $\rho \simeq \rho^\ast$ being self-dual. Furthermore, we have that $\delta$ has order $1$ or $2$. If $\delta$ has order $1$, we know that the embedding $G \hookrightarrow \SL_3(\Bbbk)$ factors through $\SL_2(\Bbbk)$, and hence $G$ is a binary dihedral group. Assume now that $\delta$ has order $2$. Let $\epsilon_l I_2 \in G$ be a central element. Then we have that the determinant of this element satisfies $\epsilon_l^4 = 1$, and hence $l \in \{1, 2, 4\} $. However, by self-duality we immediately obtain that $l \in \{1,2\} $, and hence $\delta(\epsilon_l I_2) = 1$. In particular, writing $G = \mu_l \cdot H_{\min}$ shows that $l \in \{1,2\}$, so it almost suffices to check the minimal groups.  
    
    We now go through the minimal groups from \Cref{Theo: Minimal dihedral groups}. 
    If the degree $n$ is odd, then $G$ contains a minimal group $H_{2n, k}$. The determinant needs to be non-trivial on some element of $G$, and by inspection it is clear that this needs to happen for the second generator of $H_{2n, k}$. The determinant of the second generator of $H_{2n,k}$ is $-(\epsilon_{2^k}^2)$, which must be $1$ or $-1$. Hence we have that $\epsilon_{2^k}^2$ is $-1$ or $1$, meaning $k = 2$ or $k=1$. However, if $k = 2$, we have that $\delta = \mathbf{1}$, contrary to our assumption. Thus, we have $k=1$ and $l \in \{1,2\}. $ If $l=1$ and $k=1$, then $G \simeq D_{2n}$. If $l=2$ and $k = 1$, then $ G \simeq BD_{2n}$. If $l = 2$ and $k=1$, then $G \simeq D_{2(2n)}$. 
    
    If the degree $n$ is even, we have two options from \Cref{Theo: Minimal dihedral groups} to consider. Again, we know that the determinant has to be non-trivial on some generator of the minimal group, and precisely one of the generators has determinant $1$ by inspection. If the minimal group in $G$ is $H_{2n, k,1}$, consider the determinant of the first generator. This is $\epsilon_{2^{k+1}}^2$, and it needs to be $-1$, giving us that $k = 1$. If the minimal group in $G$ is $H_{2n, k,2}$, the determinant of the second generator is $\epsilon_{2^{k+1}}^2$, so we are in the same situation that $k=1$. No additional central element can have determinant $-1$. If the minimal group is $H_{2n, k,2}$, we therefore get $G \simeq H_{2n, k,2} \simeq D_{2n}$. If the minimal group is $H_{2n, k,1}$, we consider the first generator. The character of $\rho$ takes the value $\epsilon_{4} \cdot (\epsilon_{2n} + \epsilon_{2n}^{-1})$. Taking the complex conjugate for the dual representation, we see $\epsilon_{4} \cdot (\epsilon_{2n} + \epsilon_{2n}^{-1}) = -\epsilon_{4} \cdot ( \epsilon_{2n}^{-1} + \epsilon_{2n}) $, which only holds if $(\epsilon_{2n} + \epsilon_{2n}^{-1}) = 0$. This can only happen if $n=2$, which leaves us with the group $D_8$. 
\end{proof}

\subsubsection{The tetrahedral case}
We now go through the tetrahedral case in the same fashion as the dihedral case. Before we do this, we caution the reader familiar with \cite{NdCS} that the tetrahedral, octahedral and icosahedral cases considered in \cite{NdCS} fall into type (C) according to the classification of finite subgroups of $\SL_3(\Bbbk)$, contrary to the dihedral case that indeed was type (B). We will see that the quiver consists of many copies of the one for the binary tetrahedral group in $\SL_2(\Bbbk)$, arranged on a square grid, with another grid branching off at the central vertex of the quiver of the binary version. We remind the reader of the fact that the McKay quiver for the binary tetrahedral group looks as follows, where we labeled the vertices with the dimensions of the irreducible representations.

\[ 
\begin{tikzcd}
                         &                                                 & 1 \arrow[d, shift right]                                               &                                                 &                          \\
                         &                                                 & 2 \arrow[d, shift right] \arrow[u, shift right]                        &                                                 &                          \\
1 \arrow[r, shift right] & 2 \arrow[l, shift right] \arrow[r, shift right] & 3 \arrow[l, shift right] \arrow[r, shift right] \arrow[u, shift right] & 2 \arrow[r, shift right] \arrow[l, shift right] & 1 \arrow[l, shift right]
\end{tikzcd}
\]

As before, we denote by $G \leq \GL_2(\Bbbk)$ a finite subgroup with image $A_4 \leq \PGL_2(\Bbbk)$. We denote by $\mathbf{1}$ the trivial representation, by $\rho \colon G \to \GL_2(\Bbbk)$ the defining representation, by $\delta = \det(\rho)$ the determinant representation, and by $\gamma$ the central character of $\rho$. 

\begin{Rem}
    Let $\rho$ be the defining representation of $G$. Then $\rho \otimes \rho$ has dimension $4$, and decomposes as 
    \[ \rho \otimes \rho \simeq \wedge^2 (\rho)  \oplus \Sym^2(\rho) = \delta \oplus \Sym^2(\rho).  \]
    By \Cref{Lem: Tensors of isoclinics decompose in the same dimensions}, for any irreducible $2$-dimensional representation $\psi$, the product $\rho \otimes \psi$ decomposes into a $3$-dimensional and a $1$-dimensional irreducible summand. In particular, the summand $\Sym^2(\rho)$ of $\rho \otimes \rho$ is irreducible and will be denoted by $\delta \varphi$, where $\varphi = \delta^{-1} \Sym^2(\rho)$. As in the dihedral case, computing $(\rho, \mathbf{1} \cdot \rho) = 1$ and 
    \begin{align*}
        (\rho, \rho \varphi) = (\rho \rho, \delta \varphi) = (\delta, \delta \varphi) + (\delta \varphi, \delta \varphi) = 1
    \end{align*}
    produces the connecting arrows. 
\end{Rem}

Thus, we have found the following beginning of the McKay quiver. 
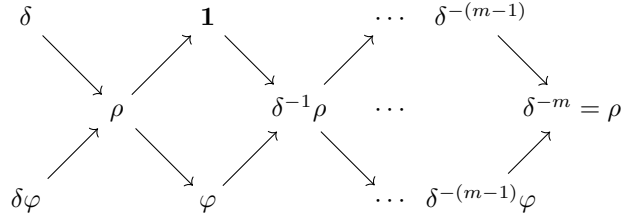
\[
\begin{tikzcd}[column sep = {between origins, 8ex}, row sep = {between origins, 8ex}]
\delta \arrow[rd]         &                            & \mathbf{1} \arrow[rd] &                                        & \cdots  & \delta^{-(m-1)} \arrow[rd]         &                    \\
                          & \rho \arrow[ru] \arrow[rd] &                       & \delta^{-1} \rho \arrow[ru] \arrow[rd] & \cdots  &                                    & \delta^{-m} = \rho \\
\delta \varphi \arrow[ru] &                            & \varphi \arrow[ru]    &                                        & \cdots  & \delta^{-(m-1)} \varphi \arrow[ru] &                   
\end{tikzcd}
\]

In the next step, we immediately add all remaining vertices of the quiver. 

\begin{Rem}
    We decompose $\varphi \otimes \rho$. This representation has dimension $6$, and we already know a $2$-dimensional summand. By \Cref{Lem: Tensors of isoclinics decompose in the same dimensions}, we can see that $\varphi \otimes \rho$ decomposes into $3$ irreducible summands, each of dimension $2$. We write $\varphi \otimes \rho = \rho \oplus \psi_2 \oplus \psi_3$. Next, decomposing $ \psi_i \otimes \rho $ via \Cref{Lem: Tensors of isoclinics decompose in the same dimensions} then shows that both $\psi_i \otimes \rho$ decompose into a sum of a $3$-dimensional irreducible representation and a $1$-dimensional representation, which we denote by $\delta \lambda_i$. Knowing this, we show that $\rho$, $\psi_2$ and $\psi_3$ are different. To see this, check that 
    \[ 1 \leq (\psi_i, \varphi \rho) = (\psi_i \rho, \delta \varphi) \leq 1,  \]
    where the last inequality follows from the fact that $\delta \varphi$ has dimension $3$ and $\psi_i \rho$ has only one $3$-dimensional summand. This also provides the arrow $\delta \varphi \to \psi_i$. Finally, we show that there exists an arrow $\psi_i \to \lambda_i$. This follows immediately from the construction since 
    \[ 1 = (\psi_i \rho, \delta \lambda_i) = (\psi_i, \lambda_i \rho). \]
\end{Rem}

By \Cref{Pro: Quiver is connected}, we have found the complete McKay quiver since in each step, we added all incoming arrows to all vertices. Before we draw the quiver, we argue that all the vertices we found are actually different. 

\begin{Rem}
    Recall that $G$ is isoclinic to the binary tetrahedral group $BT$. The latter has irreducible representations of dimensions $1, 1, 1, 2, 2, 2, 3$. Let $m$ be minimal such that $\delta^{-m} \varphi = \varphi$. Then the number $d$ of irreducible representations of $G$ of dimension $3$ satisfies $m = \frac{d}{1} = \frac{|G|}{|BT|} $. Hence, it follows that $G$ has $3m$ irreducible representations of dimension $1$, and $3m$ irreducible representations of dimension $2$. That means that in our computation above, the representations $\lambda_2$, $\lambda_3$ and $\mathbf{1}$ are all different.  
\end{Rem}

Next, we embed $G$ into $\SL_3(\Bbbk)$, meaning we add the arrows for the action of $\delta$. For this, we fix the smallest $m \geq 1$ such that $\delta^{-m} \varphi = \varphi$. For convenience, we write $\psi_1 = \rho$ and $\lambda_1 = \mathbf{1}$. 

\begin{Rem}
    Let $m \geq 1$ be minimal with $\delta^{-m} \varphi = \varphi$, and let $k_i \geq 1$ be minimal such that $\delta^{-k_i} \psi_i = \psi_i$ for $1 \leq i \leq 3$. We compute that
    \begin{align*}
        1 &= (  \varphi , \delta^{-1}\psi_i \rho) = ( \varphi, \delta^{k_i - 1} \psi \rho)  \\
        &= (\delta \varphi, \delta^{k_i -1 }(\varphi + \lambda_i)) \\
        &= (\varphi, \delta^{k_i} \varphi) + (\varphi, \delta^{k_i} \lambda_i) = (\varphi, \delta^{k_i} \varphi).
    \end{align*}
    Hence, we have that $m \mid k_i$. As in the dihedral case, it can happen that $k_i > m$, and in this case we get ``crossing arrows'' as in the dihedral case.
\end{Rem}

We now draw the McKay quiver. Since it becomes complicated, we draw it with two parts, glued along a subquiver. We invite the reader to think of the quiver as branching off at $\varphi$ into two sheets, each consisting of a triangular grid. The branching point is precisely the same as in the usual $\Tilde{E}_6$ diagram. Alternatively, and more in line with coming arguments, one may think of the quiver as consisting of three pages of a book, one for $\rho$, one for $\psi_2$, and one for $\psi_3$, all connected at a central spine containing $\varphi$. We also note that the identification of the left and right side is non-trivial. 

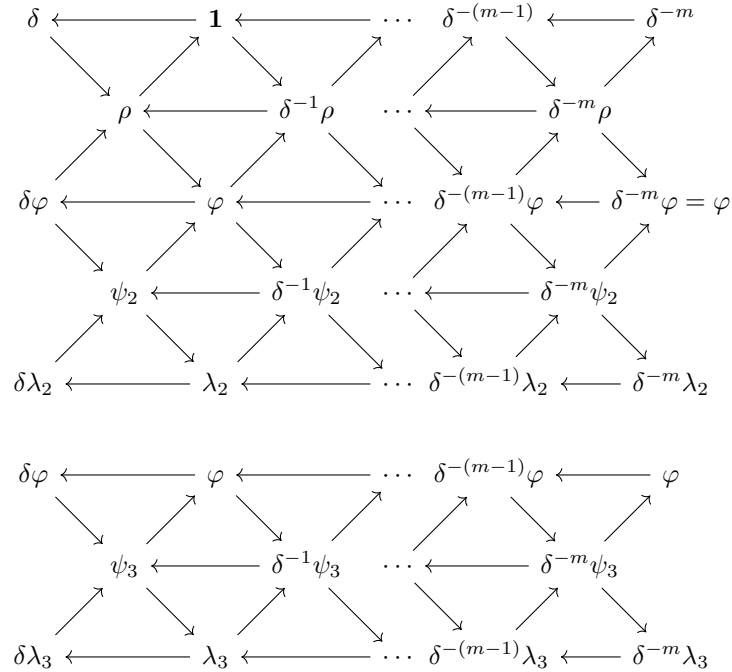
\[ 
\begin{tikzcd}[column sep = {between origins, 8ex}, row sep = {between origins, 8ex}]
\delta \arrow[rd]                    &                              & \mathbf{1} \arrow[rd] \arrow[ll]         &                                                     & \cdots \arrow[ll]             & \delta^{-(m-1)} \arrow[rd]                    &                                                      & \delta^{-m} \arrow[ll]           \\
                                     & \rho \arrow[ru] \arrow[rd]   &                                          & \delta^{-1} \rho \arrow[ru] \arrow[rd] \arrow[ll]   & \cdots \arrow[ru] \arrow[rd]  &                                               & \delta^{-m} \rho \arrow[ll] \arrow[ru] \arrow[rd]    &                                  \\
\delta \varphi \arrow[ru] \arrow[rd] &                              & \varphi \arrow[ru] \arrow[rd] \arrow[ll] &                                                     & \cdots \arrow[ll]             & \delta^{-(m-1)} \varphi \arrow[ru] \arrow[rd] &                                                      & \delta^{-m} \varphi = \varphi  \arrow[ll]  \\
                                     & \psi_2 \arrow[ru] \arrow[rd] &                                          & \delta^{-1} \psi_2 \arrow[ru] \arrow[rd] \arrow[ll] & \cdots  \arrow[ru] \arrow[rd] &                                               & \delta^{-m} \psi_2  \arrow[ll] \arrow[ru] \arrow[rd] &                                  \\
\delta \lambda_2  \arrow[ru]         &                              & \lambda_2 \arrow[ru] \arrow[ll]          &                                                     & \cdots  \arrow[ll]            & \delta^{-(m-1)} \lambda_2 \arrow[ru]          &                                                      & \delta^{-m} \lambda_2 \arrow[ll] \\
\delta \varphi \arrow[rd]            &                              & \varphi \arrow[rd] \arrow[ll]            &                                                     & \cdots \arrow[ll]             & \delta^{-(m-1)} \varphi \arrow[rd]            &                                                      & \varphi \arrow[ll]               \\
                                     & \psi_3 \arrow[ru] \arrow[rd] &                                          & \delta^{-1} \psi_3 \arrow[ru] \arrow[rd] \arrow[ll] & \cdots  \arrow[ru] \arrow[rd] &                                               & \delta^{-m} \psi_3 \arrow[ll] \arrow[ru] \arrow[rd]  &                                  \\
\delta \lambda_3 \arrow[ru]          &                              & \lambda_3 \arrow[ru] \arrow[ll]          &                                                     & \cdots \arrow[ll]             & \delta^{-(m-1)} \lambda_3 \arrow[ru]          &                                                      & \delta^{-m} \lambda_3 \arrow[ll]
\end{tikzcd}
\]

\begin{Rem}
    As seen above, the $\delta$-orbits may have different sizes. To see that this is indeed the case, consider the case of the binary tetrahedral group in one of its faithful $2$-dimensional representations not contained in $\SL_2(\Bbbk)$. This was the minimal group $BH \simeq H_0$ from \Cref{SSec: (B) structure}. Because the groups whose quivers we consider are isoclinic, it follows that the size $m$ of the orbit $\delta^i \varphi$ is at least $2$ for any group different from $BH$, since then there are at least $2$ irreducible $3$-dimensional representations.
    When the $\delta$-orbits have different size, we have $\delta^{-m} \simeq \lambda_i $ for some $i \in \{ 2,3\}$, and similarly $\delta^{-m} \rho \simeq \psi_i$. 
\end{Rem}

As seen, we can have ``crossing arrows'' arising when $k_i > m$. However, this is not an obstacle to finding cuts. We use the same construction as in the dihedral case, where we now first ``fold'' the quiver along the spine consisting of $\delta^i \varphi$ to obtain a triangular grind on the cylinder, and then mapping this to a torus by identifying $1$-dimensional representations. To visualise the folding, one may first think of the quiver for the binary tetrahedral group in $\SL_2(\Bbbk)$, and fold it according to its order $3$ automorphism. The resulting quiver is a double type $A$ diagram, which is mapped to a type $\Tilde{A}$ diagram by identifying the two vertices at the end. 

\begin{Con}
    Let $Q_G$ be the quiver computed above. Then we define the quiver $T'(Q_G)$ by identifying the $1$- and the $2$-dimensional representations in equivalence classes of size $3$ each. More precisely, we identify $\delta^i$ with $\delta^i \lambda_2$ and with $\delta^i \lambda_3$ for $0 \leq i < m$. We also identify $\delta^i \rho$ with $\delta^i \psi_2$ and with $\delta^i \psi_3$ for $0 \leq i < m$. Arrows between pairs of identified vertices are identified as well. We denote by $q' \colon Q_G \to T'(Q_G)$ the quotient map.
    Next, we define $T(Q_G)$ as the quotient of $T'(Q_G)$ that arises from identifying $q'(\delta^i )$ with $q'(\delta^i \varphi)$ for $0 \leq i < m$. As before, we identify arrows between identified vertices. We denote by $q \colon Q_G \to T(Q_G)$ the quotient map. The quiver $T(Q_G)$ embeds on the torus.
\end{Con}

Next, we show that this quiver arises as a quotient $\hat{Q}/L$ as in \Cref{Sec: type (A)}. 

\begin{Con}
    Let $Q_G$ be as above. Then $T(Q_G)$ is isomorphic to a quotient of $\hat{Q}$ by the action of a cofinite sublattice $L \leq \mathbb{Z}^2$ as in the type (A) case in \Cref{Sec: type (A)}. More precisely, let $m$ be as before, and we consider the lattice $L = \langle  m e_1, e_1 + 2 e_2 \rangle \leq \mathbb{Z}^2$. We claim that $T(Q_G)$ is isomorphic to $\hat{Q}/L$. Recall that we write $\psi_1 = \rho$ and $\lambda_1 =\mathbf{1}$. It is easy to see that there is a cycle in $T(Q_G)$ from $\mathbf{1}$ to itself, arising from the path $\mathbf{1} \to \delta \to \rho \to \varphi $ in $Q_G$. We construct an isomorphism $\iota \colon T(Q_G) \to \hat{Q}/L$ so that the action of $e_1$ on $\hat{Q}/L$ corresponds to the action of $\delta$, and so that the identified $3$-cycle gives rise to the second generator $e_1 + 2e_2$ of $L$. Recall that every vertex in $T(Q_G)$ is uniquely determined as $q(\delta^i)$ or $q(\delta^i \rho)$ for a unique $0 \leq i < m$, and we consequently write $q(\delta^i \rho^j)$ for $0 \leq i < m$ and $0 \leq j \leq 1$. This gives rise to the map
    \begin{align*}
        \iota \colon T(Q_G) &\to \hat{Q}/L \\
            q(\delta^i \rho^j) &\mapsto i e_1 + j e_2 + L.
    \end{align*}
    To check that the map is a well-defined isomorphism, we consider when the difference 
    \begin{align*}
        d &=  \iota(q(\delta^i) \rho^{j} ) - \iota(q(\delta^{i'} \rho^{j'})) \\
        &= i e_1 + j e_2 - i' e_1 - j' e_2 + L
    \end{align*}
    is $0 + L$. Since $j, j' \in \{ 0, 1\}$, it follows that $|j-j'| < 2 $. Hence $d = 0+ L$ only if $j = j'$. Assume that $j=j'$, then 
    \[ i e_1 + j e_2 - i' e_1 - j' e_2 + L = i e_1 - i' e_1 + L. \]
    Since $i , i' \in \{0, \cdots, m-1\} $, we have $|i-i'| < m$, and therefore $m \mid i - i'$ only if $i-i' = 0$. Thus, we have $d= 0$ if and only if $i=i'$ and $j=j'$. 
\end{Con}

As before, we show that we can transfer higher preprojective cuts along the identification $Q_G \to T(Q_G)$. 

\begin{Pro}\label{Pro: (B) tetrahedral has cut if abelian has}
    Let $G \leq \SL_3(\Bbbk)$ be a group of type (B) that is a central extension of the tetrahedral group. Then $Q_G$ admits a cut if $T(Q_G) \simeq \hat{Q}/L$ admits a cut. 
\end{Pro}

\begin{proof}
    Let $C' \subseteq (\hat{Q}/L)_1$ be a cut. We claim that $C = \{ a \in (Q_G)_1 \mid \iota(q(a)) \in C'  \}$ is a cut for $Q_G$. By \Cref{Cor: deter cycles cut and acyclic suffices} it is enough to show that each $3$-cycle in $Q_G$ containing exactly one arrow coming from the action of $\delta$ is cut exactly once, and that $Q_G - C'$ is acylic. To see this, note that every such cycle gets mapped to $3$ such cycles in $T(Q_G)$. Finally, to see that $Q_G - C'$ is acylic, note that if it was not, then neither would be $\hat{Q}/L - C$.  
\end{proof}

We arrive at the following theorem, where the case when $m=1$ is precisely when $G \simeq H_0$. 

\begin{Theo}\label{Theo: Type (B) tetrahedral classification}
    Let $G \leq \SL_3(\Bbbk)$ be a group of type (B) that is a central extension of the tetrahedral group. Then $R \ast G \simeq \Bbbk Q_G/I$ admits a $3$-preprojective cut if and only if $G \not \simeq BH$ where $BH$ is the binary tetrahedral group. 
\end{Theo}

\begin{proof}
    If $G \simeq BH$ is the binary tetrahedral group, then $G$ has a unique irreducible $3$-dimensional representation and hence $m=1$. This means that $Q_G$ has a loop given by a determinant arrow, so $\Bbbk Q_G/I$ does not admit a cut by \Cref{Pro: Determinant loops prevent cuts}. Now suppose that $G$ is a different group. This means we have $m \geq 2$. The matrix $B$ defining the lattice $L$ with $T(Q_G) \simeq \hat{Q}/L$ is given by  $B=\begin{pmatrix}
        m &1\\
        0&2
    \end{pmatrix}$. 
    Using \Cref{Theo: SL3 type (A) classification}, one can easily verify that $(\gamma_1,\gamma_2,\gamma_3)=(2,m-1, m-1)$ is a type of a cut. Hence $\hat{Q}/L$ admits a cut and by \Cref{Pro: (B) tetrahedral has cut if abelian has} so does $Q_G$.
\end{proof}

\subsubsection{The octahedral case}
We now go through the octahedral case in the same fashion as the dihedral and tetrahedral case. By now the strategy is quite clear: We know the McKay quiver for the binary octahedral group $BH \leq \SL_2(\Bbbk)$, and we can use this together with \Cref{Lem: Dimension 2 adjoint trick} to deduce the shape of the the McKay quiver for every group $G \leq \GL_2(\Bbbk)$ with image $H\simeq S_4$ in $\PGL_2(\Bbbk)$. We then infer the existence of a cut using methods for type (A). 

We begin by reminding the reader of the McKay quiver for the binary octahedral group $BH$. We label the vertices by the dimensions of the irreducible representations of $BH$. 
\[
\begin{tikzcd}
                         &                                                 &                                                 & 2 \arrow[d, shift right]                                               &                                                 &                                                 &                          \\
1 \arrow[r, shift right] & 2 \arrow[l, shift right] \arrow[r, shift right] & 3 \arrow[r, shift right] \arrow[l, shift right] & 4 \arrow[r, shift right] \arrow[l, shift right] \arrow[u, shift right] & 3 \arrow[r, shift right] \arrow[l, shift right] & 2 \arrow[r, shift right] \arrow[l, shift right] & 1 \arrow[l, shift right]
\end{tikzcd}
\]

As before, we denote by $G \leq \GL_2(\Bbbk)$ a finite subgroup with image $S_4 \leq \PGL_2(\Bbbk)$. We denote by $\mathbf{1}$ the trivial representation, by $\rho \colon G \to \GL_2(\Bbbk)$ the defining representation, by $\delta = \det(\rho)$ the determinant representation, and by $\gamma$ the central character of $\rho$. 

\begin{Rem}
    The tensor square $\rho \otimes \rho = \delta \oplus \varphi_1'$ decomposes into a $1$ and a $3$-dimensional representation. We write as before $\psi_1 = \delta^{-1} \psi_1'$. To fill in the arrow $\rho \to \psi_1$, note that 
    \[ (\rho, \psi_1 \rho) = (\rho \rho, \delta \psi_1) = (\delta, \delta \psi_1) + (\delta \psi_1 , \delta \psi_1).  \]
    Next, we consider $\psi_1 \otimes \rho$, which has dimension $6$ and decomposes as $\rho \oplus \varphi$ for a representation $\varphi$ of dimension $4$. The dimensions can be found from \Cref{Lem: Tensors of isoclinics decompose in the same dimensions}. Using 
    \[ ( \delta \psi_1, \varphi \rho) = (\psi_1 \rho, \varphi) = 1 \]
    shows that there exists an arrow $\delta \psi_1 \to \varphi$.
\end{Rem}

Thus, we have reached a ``central'' vertex at which we expect the quiver to split into several sheets. 

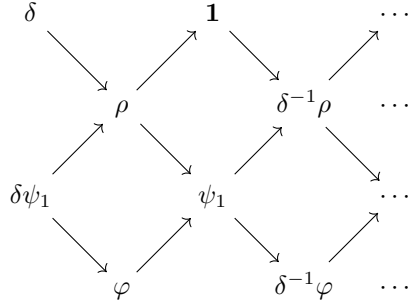
\[\begin{tikzcd}[column sep = {between origins, 8ex}, row sep = {between origins, 8ex}]
\delta \arrow[rd]                   &                            & \mathbf{1} \arrow[rd]        &                                        & \cdots  \\
                                    & \rho \arrow[rd] \arrow[ru] &                              & \delta^{-1} \rho \arrow[ru] \arrow[rd] & \cdots  \\
\delta \psi_1 \arrow[ru] \arrow[rd] &                            & \psi_1 \arrow[ru] \arrow[rd] &                                        & \cdots  \\
                                    & \varphi \arrow[ru]         &                              & \delta^{-1} \varphi \arrow[ru]         & \cdots 
\end{tikzcd}
\]

\begin{Rem}
    Now we consider $\varphi \otimes \rho$. This has dimension $8$, and by \Cref{Lem: Tensors of isoclinics decompose in the same dimensions} it decomposes into $3$ irreducible summands, two of which have dimension $3$ and one of which has dimension $2$. We denote the $2$-dimensional summand by $\rho_3$. We saw that one of the $3$-dimensional summands is $\delta \psi_1$, so we denote the other by $\delta \psi_2$. We furthermore note that $\psi_2 \otimes \rho$ decomposes into a $4$ and a $2$-dimensional representation by \Cref{Lem: Tensors of isoclinics decompose in the same dimensions}, and we denote the $2$-dimensional summand by $\rho_2$. This allows us to show that the two $3$-dimensional summands $\delta \psi_1, \delta \psi_2$ of $\varphi \otimes \rho$ are different by computing $ (\delta \psi_i, \varphi \rho) = (\psi_i \rho,  \varphi) = (\varphi + \rho_i, \varphi) = 1, $
    where we used that $(\rho_i, \varphi) = 0$ for dimension reasons. 
    To complete the computation of the new rows, note that $\rho_2 \otimes \rho$ has dimension $4$ and decomposes into $\delta \psi_2$ of dimension $3$, as well as a $1$-dimensional representation, which we denote by $\delta \lambda_2$. Finally, it is easy to check that $(\psi_2, \lambda_2 \rho) = 1$.
\end{Rem}

We have thus found the complete McKay quiver by adding all the $\delta$-translates of the newly computed vertices. Indeed, for dimension reasons the $\delta^{-i} \lambda_2$ do not have any other incoming arrows, so by \Cref{Pro: Quiver is connected} we are done. Before we draw the quiver, we show that all vertices are different, and we add the determinant arrows. 

\begin{Rem}
    Recall that the dimensions of irreducible representations of the binary octahedral group are $1,1,2,2,2,3,3,4$. Let $m \geq 1$ be minimal with $\delta^m \varphi = \varphi$, then we know that $G$ has $m$ irreducible representations of dimension $4$, and $2m$ irreducible representations of dimensions $3$ and $1$ each, and $3m$ irreducible representations of dimension $2$. 
\end{Rem}

We add in the determinant arrows. We fix $m\geq 1$ minimal with $\delta^{m} \varphi = \varphi$. As in the tetrahedral case, the $\delta$-orbits can have different sizes. However, the only information that we need for the coming construction of cuts is that the $\delta$-orbit of $\rho_3$ is of size $m$. 

\begin{Rem}
    Consider the representations $\delta^{i} \rho_3$. We have $(\delta^{-1} \rho_3 , \varphi \rho) = 1$ by assumption, and hence $1 = ( \varphi, \rho_3 \rho) $. For dimension reasons, we therefore have $\varphi \simeq \rho_3 \otimes \rho$. Let $k\geq 1$ be minimal with $\delta^k \rho_3 = \rho_3$. Then we have 
    \begin{align*}
        1 &= (\varphi, \rho_3 \rho) = (\varphi , \delta^m \rho_3 \rho) \\
        &= (\varphi \rho \delta^{-1}, \delta^m \rho_3) \\
        &= (\psi_1 + \psi_2 + \rho_3, \delta^m \rho_3) = (\rho_3, \delta^m, \rho_3),
    \end{align*}
    so $k \mid m$. Similarly, we find that 
    \begin{align*}
        1 &= (\varphi, \rho_3 \rho) = (\varphi , (\delta^k \rho_3) \rho) = (\varphi, \delta^k (\rho_3 \rho)) = (\varphi, \delta^k \varphi),
    \end{align*}
    so we have $m \mid k$, and hence $k = m$.
\end{Rem}

We now draw the quiver, and again we draw it with two pieces, glued at a subquiver containing $\varphi$.

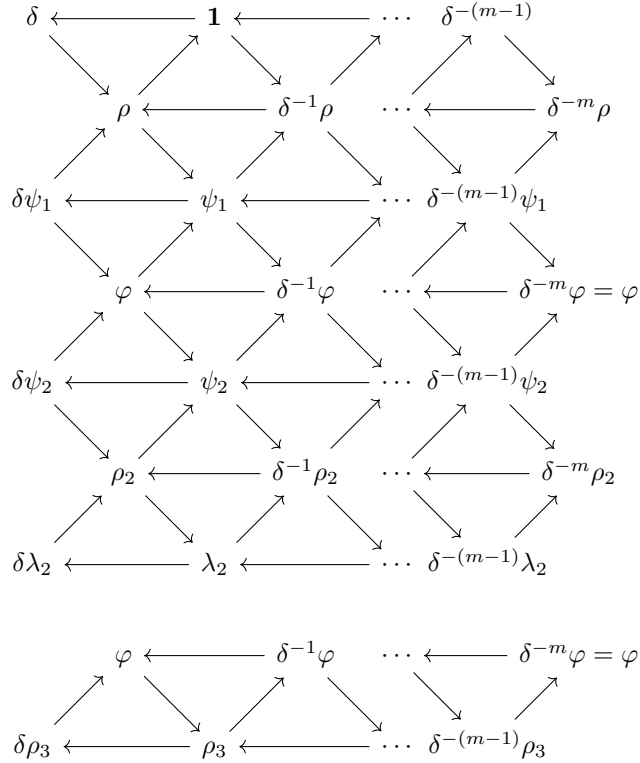
\[
\begin{tikzcd}[column sep = {between origins, 8ex}, row sep = {between origins, 8ex}]
\delta \arrow[rd]                   &                               & \mathbf{1} \arrow[rd] \arrow[ll]        &                                                      & \cdots  \arrow[ll]            & \delta^{-(m-1)} \arrow[rd]                   &                                          \\
                                    & \rho \arrow[rd] \arrow[ru]    &                                         & \delta^{-1} \rho \arrow[ru] \arrow[rd] \arrow[ll]    & \cdots  \arrow[ru] \arrow[rd] &                                              & \delta^{-m} \rho \arrow[ll]              \\
\delta \psi_1 \arrow[ru] \arrow[rd] &                               & \psi_1 \arrow[ru] \arrow[rd] \arrow[ll] &                                                      & \cdots \arrow[ll]             & \delta^{-(m-1)} \psi_1 \arrow[ru] \arrow[rd] &                                          \\
                                    & \varphi \arrow[ru] \arrow[rd] &                                         & \delta^{-1} \varphi \arrow[ru] \arrow[rd] \arrow[ll] & \cdots  \arrow[ru] \arrow[rd] &                                              & \delta^{-m} \varphi = \varphi \arrow[ll] \\
\delta \psi_2 \arrow[ru] \arrow[rd] &                               & \psi_2 \arrow[rd] \arrow[ru] \arrow[ll] &                                                      & \cdots \arrow[ll]             & \delta^{-(m-1)} \psi_2 \arrow[ru] \arrow[rd] &                                          \\
                                    & \rho_2 \arrow[ru] \arrow[rd]  &                                         & \delta^{-1} \rho_2 \arrow[ru] \arrow[ll] \arrow[rd]  & \cdots \arrow[ru] \arrow[rd]  &                                              & \delta^{-m} \rho_2 \arrow[ll]            \\
\delta \lambda_2 \arrow[ru]         &                               & \lambda_2 \arrow[ru] \arrow[ll]         &                                                      & \cdots \arrow[ll]             & \delta^{-(m-1)} \lambda_2 \arrow[ru]         &                                          \\
                                    & \varphi \arrow[rd]            &                                         & \delta^{-1} \varphi \arrow[rd] \arrow[ll]            & \cdots  \arrow[rd]            &                                              & \delta^{-m} \varphi = \varphi \arrow[ll] \\
\delta \rho_3 \arrow[ru]            &                               & \rho_3 \arrow[ru] \arrow[ll]            &                                                      & \cdots \arrow[ll]             & \delta^{-(m-1)} \rho_3 \arrow[ru]            &                                         
\end{tikzcd}
\]

Now we fold the quiver onto a torus as before. As noted, the $\delta$-orbit of $\rho$ can contain $\rho_2$, just as the $\delta$-orbit of $\mathbf{1}$ can contain $\lambda_2$. However, we first identify these parts of the quiver, so this will not be an obstruction to finding cuts. To visualise the folding, one can first think of the quiver of the binary ocatehdral group, and fold it according to its automorphism of order $2$. The resulting quiver is a double type $A$ diagram, which is mapped to a type $\Tilde{A}$ diagram by identifying the end vertices. 

\begin{Con}
    Let $Q_G$ be the quiver computed above. Then we define the quiver $T'(Q_G)$ by identifying the $1$- and $3$-dimensional representations in equivalence classes of size $2$ each, and identifying certain $2$-dimensional representations in classes of size $2$. More precisely, we identify $\delta^i$ with $\delta^i \lambda_2$ for $0 \leq i < m$. We also identify $\delta^i \rho$ with $\delta^i \rho_2$ for $0 \leq i < m$. We also identify $\delta^i \psi_1 $ with $\delta^i \psi_2 $ for $0 \leq i < m$. Arrows between pairs of identified vertices are identified as well. We denote by $q' \colon Q_G \to T'(Q_G)$ the quotient map. Note that we have not made any identification of the vertices $\delta^i \rho_2$. 
    Next, we define $T(Q_G)$ as the quotient of $T'(Q_G)$ that arises from identifying $q'(\delta^i )$ with $q'(\delta^i \rho_2)$ for $0 \leq i < m$. As before, we identify arrows between identified vertices. We denote by $q \colon Q_G \to T(Q_G)$ the quotient map. The quiver $T(Q_G)$ embeds on the torus.
\end{Con}

Next, we show that this quiver arises as a quotient $\hat{Q}/L$ as in \Cref{Sec: type (A)}. 

\begin{Con}
    Let $Q_G$ be as above. Then $T(Q_G)$ is isomorphic to a quotient of $\hat{Q}$ by the action of a cofinite sublattice $L \leq \mathbb{Z}^2$ as in the type (A) case in \Cref{Sec: type (A)}. More precisely, let $m$ be as before, and we consider the lattice $L = \langle  m e_1, 2e_1 + 4 e_2 \rangle \leq \mathbb{Z}^2$. We claim that $T(Q_G)$ is isomorphic to $\hat{Q}/L$. We construct an isomorphism $\iota \colon T(Q_G) \to \hat{Q}/L$ so that the action of $e_1$ on $\hat{Q}/L$ corresponds to the action of $\delta$. Note that there are cycles in $T(Q_G)$ arising from paths in $Q_G$ from $\mathbf{1}$ to $\rho_2$. We therefore construct $\iota$ so that the $6$-cycle gives rise to the second generator $2e_1 + 4e_2$ of $L$. Recall that every vertex in $T(Q_G)$ is uniquely determined as $q(\delta^i)$ or $q(\delta^i \rho)$ or $q(\delta^i \psi_1)$ or $q(\delta^i \varphi)$ for a unique $0 \leq i < m$. This gives rise to the map
    \begin{align*}
        \iota \colon T(Q_G) &\to \hat{Q}/L \\
            q(\delta^i ) &\mapsto i e_1 + L \\
            q(\delta^i \rho) &\mapsto (i+1) e_1 + e_2 + L \\
            q(\delta^i \psi_1) &\mapsto (i+1) e_1 + 2e_2 + L \\
            q(\delta^i \varphi) &\mapsto (i+2) e_1 + 3e_2 + L
    \end{align*}
    As before, one checks easily that this is a well-defined isomorphism of quivers.
\end{Con}

\begin{Pro}\label{Pro: (B) octahedral has cut if abelian has}
    Let $G \leq \SL_3(\Bbbk)$ be a group of type (B) that is a central extension of the octahedral group. Then $Q_G$ admits a cut if $T(Q_G) \simeq \hat{Q}/L$ admits a cut. 
\end{Pro}

\begin{proof}
    Let $C' \subseteq (\hat{Q}/L)_1$ be a cut. As in \Cref{Pro: (B) tetrahedral has cut if abelian has}, one checks that $C = \{ a \in (Q_G)_1 \mid \iota(q(a)) \in C'  \}$ is a cut for $Q_G$.   
\end{proof}

We arrive at the following theorem, where the case when $m=1$ is precisely when $G \leq \SL_2(\Bbbk)$. 

\begin{Theo}\label{Theo: Type (B) octahedral classification}
    Let $G \leq \SL_3(\Bbbk)$ be a group of type (B) that is a central extension of the octahedral group. Then $R \ast G \simeq \Bbbk Q_G/I$ admits a $3$-preprojective cut if and only if $\rho \colon G \hookrightarrow \SL_3(\Bbbk)$ does not factor through an embedding $\SL_2(\Bbbk) \hookrightarrow \SL_3(\Bbbk)$.
\end{Theo}

\begin{proof}
    If $G \leq \SL_2(\Bbbk)$, then $\delta = \mathbf{1}$ and hence $Q_G$ has loops given by determinant arrows, so $\Bbbk Q_G/I$ does not admit a cut by \Cref{Pro: Determinant loops prevent cuts}. Now suppose that $G$ does not embed in $\SL_2(\Bbbk)$. One can check that $G$ is then not isomorphic to the binary octahedral group. This means we have $m \geq 2$. The matrix $B$ defining the lattice $L$ with $T(Q_G) \simeq \hat{Q}/L$ is given by  $B=\begin{pmatrix}
        m &2\\
        0&4
    \end{pmatrix}$. 
    Using \Cref{Theo: SL3 type (A) classification}, one can easily verify that $(\gamma_1,\gamma_2,\gamma_3)=(4,2m-2, 2m-2)$ is a type of a cut for $m \geq 2$. Hence $\hat{Q}/L$ admits a cut and by \Cref{Pro: (B) octahedral has cut if abelian has} so does $Q_G$.
\end{proof}

\subsubsection{The icosahedral case}
We now sketch the remaining case, where $A_5 \simeq H \leq \PGL_2(\Bbbk) $ is the group of icosahedral symmetries. We consider a finite group $G \leq \GL_2(\Bbbk)$ with image $H$ in $\PGL_2(\Bbbk)$. We use the same strategy and notation as before. 

We begin by reminding the reader of the McKay quiver for the binary icosahedral group $BH$. We label the vertices by the dimensions of the irreducible representations of $BH$. 

\[
\begin{tikzcd}
                         &                                                 &                                                 &                                                 &                                                 & 3 \arrow[d, shift right]                                               &                                                 &                          \\
1 \arrow[r, shift right] & 2 \arrow[r, shift right] \arrow[l, shift right] & 3 \arrow[r, shift right] \arrow[l, shift right] & 4 \arrow[r, shift right] \arrow[l, shift right] & 5 \arrow[r, shift right] \arrow[l, shift right] & 6 \arrow[r, shift right] \arrow[l, shift right] \arrow[u, shift right] & 4 \arrow[r, shift right] \arrow[l, shift right] & 2 \arrow[l, shift right]
\end{tikzcd}
\]

Note that this quiver has no symmetries, so we can already except that the $\delta$-orbits will all be of the same size. However, we begin as always by decomposing $\rho \otimes \rho$. We obtain a summand of dimension $3$, then consider its tensor product with $\rho$ that has a summand of dimension $4$, and this procedure continues until we reach an irreducible representations of dimension $6$.

\begin{Rem}
    We repeatedly apply \Cref{Lem: Tensors of isoclinics decompose in the same dimensions}. Decompose $\rho \otimes \rho \simeq \delta \oplus \delta \psi_{3,1} $, where $\psi_{3,1}$ is irreducible of dimension $3$. Next, decompose $\psi_{3,1} \otimes \rho \simeq \rho \oplus \psi_{4,1}$ where $\psi_{4,1}$ is irreducible of dimension $4$. Then, decompose $\delta^{-1} \psi_{4,1} \otimes \rho \simeq \psi_{3,1} \oplus \psi_5$, where $\psi_5$ is irreducible of dimension $5$. Finally, decompose $\psi_5 \otimes \rho \simeq \psi_{4,1} \oplus \psi_6$ for $\psi_6$ irreducible of dimension $6$.    
\end{Rem}

Using the same action of $\delta$ and the same dimension arguments as before, we obtain the following quiver. We write $\psi_{2,1} = \rho$, so that the first index of $\psi_{i,j}$ indicates its dimension $i$. 

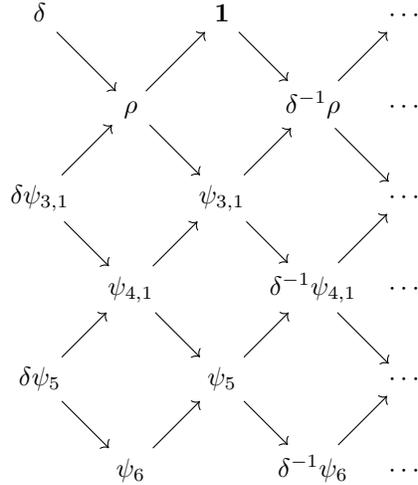
\[
\begin{tikzcd}[column sep = {between origins, 8ex}, row sep = {between origins, 8ex}]
\delta \arrow[rd]                         &                                    & \mathbf{1} \arrow[rd]              &                                                & \cdots \\
                                          & \rho \arrow[ru] \arrow[rd]         &                                    & \delta^{-1} \rho \arrow[ru] \arrow[rd]         & \cdots \\
{\delta \psi_{3,1}} \arrow[ru] \arrow[rd] &                                    & {\psi_{3,1}} \arrow[rd] \arrow[ru] &                                                & \cdots \\
                                          & {\psi_{4,1}} \arrow[ru] \arrow[rd] &                                    & {\delta^{-1} \psi_{4,1}} \arrow[ru] \arrow[rd] & \cdots \\
\delta \psi_5 \arrow[ru] \arrow[rd]       &                                    & \psi_5 \arrow[ru] \arrow[rd]       &                                                & \cdots \\
                                          & \psi_6 \arrow[ru]                  &                                    & \delta^{-1} \psi_6 \arrow[ru]                  & \cdots
\end{tikzcd}
\]

As before, we have reached the branching point. 

\begin{Rem}
    Consider $\delta^{-1} \psi_6 \otimes \rho$. Since the binary icosahedral group has a unique irreducible representation of dimension $6$, we can infer from \Cref{Lem: Tensors of isoclinics decompose in the same dimensions} that $\delta^{-1} \psi_6 \otimes \rho$ decomposes into irreducible summands of dimension $3$, $4$ and $5$. We denote the $4$-dimensional summand by $\psi_{4,2}$ and the $3$-dimensional summand by $\psi_{3,2}$. Again using \Cref{Lem: Tensors of isoclinics decompose in the same dimensions}, we see for $i \in \{3,4\}$ that $\psi_{i,2}$ is different from $\delta^{-j} \psi_{i,1}$ for all $j$ since their tensor products with $\rho$ decompose into summands of different dimensions. 
    We decompose $\psi_{4,2} \otimes \rho$, again using \Cref{Lem: Tensors of isoclinics decompose in the same dimensions}, so find one last irreducible representation $\psi_{2,2}$ of dimension $2$. We have that $\psi_{2,2} \otimes \rho$ has dimension $4$, and it is easy to see that $(\delta \psi_{4,2}, \psi_{2,2} \rho) = 1$, so we have found the whole McKay quiver. 
\end{Rem}

We let $m \geq 1$ be minimal with $\delta^m \psi_6 = \psi_6$. 

\begin{Rem}
    As in the octahedral case, one can compute that $m$ is also minimal with $\delta^m \psi_{3,2} = \psi_{3,2}$. Thus, using the isoclinism to the binary icosahedral group, it follows that the $\delta$-orbit of $\psi_{3,1}$ also has size $m$. Similarly, using that $\psi_{4,1} \otimes \rho $ and $\psi_{4,2} \otimes \rho$ decompose into irreducible summands of different dimensions, both of their $\delta$-orbits are of size $m$. The same argument shows that $\psi_{2,1} = \rho$ and $\psi_{2,2}$ each have a $\delta$-orbit of size $m$, and the fact that $\psi_5$ and $\mathbf{1}$ have $\delta$-orbits of size $m$ follows again from the isoclinism to the binary icosahedral group, since it has unique irreducible representations of dimensions $1$ and $5$. 
\end{Rem}

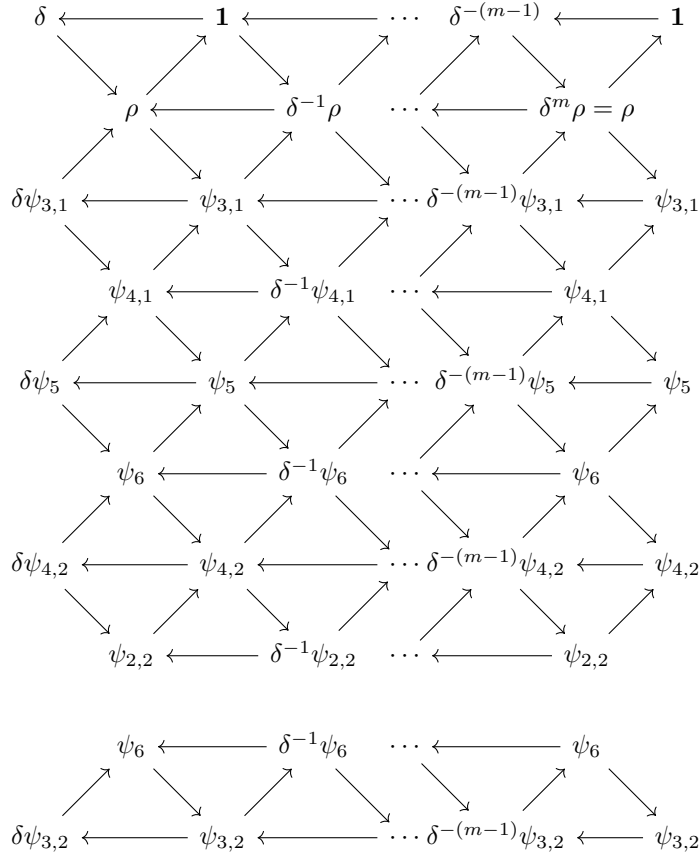
\[ 
\begin{tikzcd}[column sep = {between origins, 8ex}, row sep = {between origins, 8ex}]
\delta \arrow[rd]                         &                                    & \mathbf{1} \arrow[rd] \arrow[ll]              &                                                           & \cdots \arrow[ll]            & \delta^{-(m-1)} \arrow[rd]                         &                                                       & \mathbf{1} \arrow[ll]   \\
                                          & \rho \arrow[ru] \arrow[rd]         &                                               & \delta^{-1} \rho \arrow[ru] \arrow[rd] \arrow[ll]         & \cdots \arrow[ru] \arrow[rd] &                                                    & \delta^m \rho = \rho \arrow[ru] \arrow[rd] \arrow[ll] &                         \\
{\delta \psi_{3,1}} \arrow[ru] \arrow[rd] &                                    & {\psi_{3,1}} \arrow[rd] \arrow[ru] \arrow[ll] &                                                           & \cdots \arrow[ll]            & {\delta^{-(m-1)} \psi_{3,1}} \arrow[ru] \arrow[rd] &                                                       & {\psi_{3,1}} \arrow[ll] \\
                                          & {\psi_{4,1}} \arrow[ru] \arrow[rd] &                                               & {\delta^{-1} \psi_{4,1}} \arrow[ru] \arrow[rd] \arrow[ll] & \cdots \arrow[ru] \arrow[rd] &                                                    & {\psi_{4,1}} \arrow[ru] \arrow[ll] \arrow[rd]         &                         \\
\delta \psi_5 \arrow[ru] \arrow[rd]       &                                    & \psi_5 \arrow[ru] \arrow[rd] \arrow[ll]       &                                                           & \cdots \arrow[ll]            & \delta^{-(m-1)} \psi_5 \arrow[ru] \arrow[rd]       &                                                       & \psi_5 \arrow[ll]       \\
                                          & \psi_6 \arrow[ru] \arrow[rd]       &                                               & \delta^{-1} \psi_6 \arrow[ru] \arrow[rd] \arrow[ll]       & \cdots \arrow[ru] \arrow[rd] &                                                    & \psi_6 \arrow[rd] \arrow[ru] \arrow[ll]               &                         \\
{\delta \psi_{4,2}} \arrow[rd] \arrow[ru] &                                    & {\psi_{4,2}} \arrow[ru] \arrow[rd] \arrow[ll] &                                                           & \cdots \arrow[ll]            & {\delta^{-(m-1)} \psi_{4,2}} \arrow[ru] \arrow[rd] &                                                       & {\psi_{4,2}} \arrow[ll] \\
                                          & {\psi_{2,2}} \arrow[ru]            &                                               & {\delta^{-1} \psi_{2,2}} \arrow[ru] \arrow[ll]            & \cdots \arrow[ru]            &                                                    & {\psi_{2,2}} \arrow[ru] \arrow[ll]                    &                         \\
                                          & \psi_6 \arrow[rd]                  &                                               & \delta^{-1} \psi_6 \arrow[rd] \arrow[ll]                  & \cdots  \arrow[rd]           &                                                    & \psi_6 \arrow[rd] \arrow[ll]                          &                         \\
{\delta \psi_{3,2}} \arrow[ru]            &                                    & {\psi_{3, 2}} \arrow[ru] \arrow[ll]           &                                                           & \cdots \arrow[ll]            & {\delta^{-(m-1)} \psi_{3,2}} \arrow[ru]            &                                                       & {\psi_{3,2}} \arrow[ll]
\end{tikzcd}
\]

We have drawn the full quiver above, as before with two parts. Vertices of the same labels are identified, as well as the obvious arrows. 

Since all rows have the same size, we can fold this quiver easily. To visualise the folding we now apply, one can think of the quiver for the binary icosahedral group. It is folded to a doubled type $A$ diagram by choosing the $3$-dimensional representation which is adjacent to the $6$-dimensional representation, and folding it onto the $5$-dimensional representation. Then the type $A$ diagram is mapped to a type $\Tilde{A}$ diagram by identifying the end points.  

\begin{Con}
    Let $Q_G$ be the McKay quiver as computed above. Then we define the quiver $T'(Q_G)$ by identifying the vertices $\delta^i \psi_{3,2}$ with the vertices $\delta^i \psi_5$ for $0 \leq i < m$. Arrows between pairs of identified vertices are identified as well. We denote by $q' \colon Q_G \to T'(Q_G)$ the quotient map. The resulting quiver can be seen as a triangular grid on a cylinder. 
    Then we define $T(Q_G)$ by identifying $q'(\delta^i)$ with $q'(\delta^i \psi_{2,2})$ for $0 \leq i < m$, and again we identify arrows between pairs of identified vertices. The resulting quotient map is denoted by $q \colon Q_G \to T(Q_G)$. The quiver $T(Q_G)$ embeds on the torus
\end{Con}

As before, we now show that $T(Q_G)$ can be obtained as a quotient $\hat{Q}/L$ as in \Cref{Sec: type (A)}. 

\begin{Con}
    Let $Q_G$ be as above. Then $T(Q_G)$ is isomorphic to a quotient of $\hat{Q}$ by the action of a cofinite sublattice $L \leq \mathbb{Z}^2$ as in the type (A) case in \Cref{Sec: type (A)}. More precisely, let $m$ be as before, and consider the lattice $L = \langle  m e_1, 4e_1 + 7 e_2 \rangle \leq \mathbb{Z}^2$. We claim that $T(Q_G)$ is isomorphic to $\hat{Q}/L$. We construct an isomorphism $\iota \colon T(Q_G) \to \hat{Q}/L$ so that the action of $e_1$ on $\hat{Q}/L$ corresponds to the action of $\delta$. Note that there are cycles in $T(Q_G)$ arising from paths in $Q_G$ from $\mathbf{1}$ to $\psi_{2,2}$. We therefore construct $\iota$ so that the $11$-cycle gives rise to the second generator $4e_1 + 3e_2$ of $L$. We write $\psi_{1,1} = \mathbf{1}$ and $\psi_{5,1} = \psi_5$. Then every vertex in $T(Q_G)$ is uniquely determined as $q(\delta^i \psi_{j,1})$ for a unique $0 \leq i < m$ and $j \in \{ 1,\ldots ,6 \}$ or as $q(\delta^i \psi_{4, 2}) $. This gives rise to the map
    \begin{align*}
        \iota \colon T(Q_G) &\to \hat{Q}/L \\
            q(\delta^i \psi_{j, 1} ) &\mapsto (i + \lceil \frac{j-1}{2} \rceil )e_1 + (j-1)e_2 + L \\
            q(\delta^i \psi_{4,2}) &\mapsto (i+3) e_1 + 6e_2 + L 
    \end{align*}
    As before, one checks easily that this is a well-defined isomorphism of quivers.
\end{Con}

\begin{Pro}\label{Pro: (B) icosahedral has cut if abelian has}
    Let $G \leq \SL_3(\Bbbk)$ be a group of type (B) that is a central extension of the icosahedral group. Then $Q_G$ admits a cut if $T(Q_G) \simeq \hat{Q}/L$ admits a cut. 
\end{Pro}

\begin{proof}
    Let $C' \subseteq (\hat{Q}/L)_1$ be a cut. As in \Cref{Pro: (B) tetrahedral has cut if abelian has}, one checks that $C = \{ a \in (Q_G)_1 \mid \iota(q(a)) \in C'  \}$ is a cut for $Q_G$.   
\end{proof}

We arrive at the following theorem, where the case when $m=1$ is precisely when $G \leq \SL_2(\Bbbk)$. 

\begin{Theo}\label{Theo: Type (B) icosahedral classification}
    Let $G \leq \SL_3(\Bbbk)$ be a group of type (B) that is a central extension of the icosahedral group. Then $R \ast G \simeq \Bbbk Q_G/I$ admits a $3$-preprojective cut if and only if $\rho \colon G \hookrightarrow \SL_3(\Bbbk)$ does not factor through an embedding $\SL_2(\Bbbk) \hookrightarrow \SL_3(\Bbbk)$.
\end{Theo}

\begin{proof}
    If $G \leq \SL_2(\Bbbk)$, then $\delta = \mathbf{1}$, so $m=1$, which means that $Q_G$ has loops given by determinant arrows, so $\Bbbk Q_G/I$ does not admit a cut by \Cref{Pro: Determinant loops prevent cuts}. Now suppose that $G$ does not embed in $\SL_2(\Bbbk)$. One can check that $G$ is then not isomorphic to the binary icosahedral groupp. This means we have $m \geq 2$. The matrix $B$ defining the lattice $L$ with $T(Q_G) \simeq \hat{Q}/L$ is given by  $B=\begin{pmatrix}
        m &4\\
        0&7
    \end{pmatrix}$. 
    Using \Cref{Theo: SL3 type (A) classification}, one can easily verify that $(\gamma_1,\gamma_2,\gamma_3)=(7,4m-4, 3m-3)$ is a type of a cut for $m \geq 2$. Hence $\hat{Q}/L$ admits a cut and by \Cref{Pro: (B) icosahedral has cut if abelian has} so does $Q_G$.
\end{proof}

Combining the results from this section, we have the following result for groups of type (B). One could rephrase this theorem in terms of isomorphism types of groups, but we prefer noting that it is essentially self-duality of the defining representation and some small cases that prevent the existence of a cut.

\begin{Theo}\label{Theo: Type (B) classification}
    Let $G \leq \GL_2(\Bbbk) \leq \SL_3(\Bbbk)$ be a non-abelian group of type (B). Then $R \ast G \simeq \Bbbk Q_G/I$ admits a $3$-preprojective cut if and only if the defining representation $\rho \colon G \to \GL_2(\Bbbk)$ is not self-dual and $G$ is not isomorphic to a subgroup of $\SL_2(\Bbbk)$. 
\end{Theo}

\section{The exceptional cases}\label{Sec: Except}
In the classification of finite subgroups of $\SL_3(\Bbbk)$, there remain finitely many groups which were not covered in the cases (A) - (D). For each, we now present its structure and its McKay quiver, as well as a $3$-preprojective cut if it exists. If no cut exists, this can be verified by applying \Cref{Pro: loop in potential}. The results in this section can be computed and verified by hand, but we made extensive use of GAP \cite{GAP4}, the package OSCAR \cite{OSCAR} and the Graphs library \cite{Graphs2021} in Julia \cite{Julia-2017}, so we end by giving a brief description of the involved computations. 

In this section, we use the following convention when drawing McKay quivers. 

\begin{Conv}\label{Conv: Quiver drawings}
    We draw many McKay quivers below. The vertices are labeled according to the numbering of the irreducible characters as computed by GAP. If a vertex with the same label appears twice, the vertices are identified. We draw the usual arrows as $i \to j $, and two parallel arrows as $i \Rightarrow j$. 
    If an arrow $i \to j$ appears twice in the drawing, be it together with $i \to j'$ or together with $i' \to j'$ or together $i' \to j$, for $i=i'$ and $j=j'$, the arrows are identified. Similarly, pairs of parallel arrows are identified to yield one copy of the two parallel arrows. 
\end{Conv}

\subsection{Types (E) - (G)}
There are three exceptional subgroups of $\SL_3(\Bbbk)$ which arise as iterated extensions of a group of type (C), and which were found as the primitive groups with imprimitive normal subgroups. 
We fix notation, following Yau and Yu \cite{YauYu}. We need the following matrices. 
\begin{alignat*}{3}
    s &= \left( \begin{smallmatrix}
        1 & 0 & 0 \\
        0 & \epsilon_3 &0 \\
        0 & 0 & \epsilon_3^2 
    \end{smallmatrix} \right)  \quad 
    &&t =  \left( \begin{smallmatrix}
        0 & 1 &0 \\
        0 &0 &1 \\
        1 & 0 &0 
    \end{smallmatrix} \right) \quad 
    &&v = \frac{1}{i \sqrt{3}} \cdot  \left( \begin{smallmatrix}
        1 & 1 & 1 \\
        1 & \epsilon_3 & \epsilon_3^2 \\
        1 & \epsilon_3^2 & \epsilon_3
    \end{smallmatrix} \right) \\
    u &=  \left( \begin{smallmatrix}
        \epsilon_9^2 & 0 & 0 \\
        0 & \epsilon_9^2 & 0 \\
        0 & 0 & \epsilon_9^2 \epsilon_3
    \end{smallmatrix} \right) \quad 
    &&p  = \frac{1}{i \sqrt{3}} \cdot \left( \begin{smallmatrix}
        1 & 1 & \epsilon_3^2 \\
        1 & \epsilon_3 & \epsilon_3 \\
        \epsilon_3 & 1 & \epsilon_3
    \end{smallmatrix} \right)
    &&
\end{alignat*}
Note that $|v| = 4$, and that $v^2 = \left( \begin{smallmatrix} -1 & 0 & 0 \\ 0 & 0 & -1 \\ 0 & -1 & 0  \end{smallmatrix}  \right) $, and that $p = u v u^{-1}$.

\subsubsection{Type (E)}
Let $G = \langle s,t,v\rangle$. It is easy to see that 
\[ G \simeq \langle s, t \rangle \rtimes \langle v \rangle \simeq   ((C_3 \times C_3) \rtimes C_3 )\rtimes C_4. \]
Hence $|G| = 108$. Furthermore, this can be seen as an extension of a group of type (C), or of a group of type (D). Indeed, $H_1 = \langle s,t \rangle \simeq (C_3 \times C_3) \rtimes C_3$ is a group of type (C), and $H_2 =\langle s,t,v^2\rangle \simeq H_1 \rtimes C_2 $ is of type (D). 

One can compute the McKay quiver directly, and we draw it below together with the adjacency matrix, according to \Cref{Conv: Quiver drawings}. It is easy to see that all shortest cycles in $Q_G$ have length $3$, and we display in red a cut such that every $3$-cycle is homogeneous of degree $1$. 

\[  \left(
\begin{smallmatrix}
  0& 0& 0& 0& 0& 0& 0& 1& 0& 0& 0& 0& 0& 0 \\
  0& 0& 0& 0& 1& 0& 0& 0& 0& 0& 0& 0& 0& 0 \\ 
  0& 0& 0& 0& 0& 0& 0& 0& 1& 0& 0& 0& 0& 0 \\
  0& 0& 0& 0& 0& 0& 0& 0& 0& 0& 0& 1& 0& 0 \\ 
  0& 0& 0& 0& 0& 1& 0& 0& 0& 1& 1& 0& 0& 0 \\
  0& 1& 0& 0& 0& 0& 0& 0& 0& 0& 0& 0& 1& 1 \\ 
  1& 0& 0& 0& 0& 0& 0& 0& 0& 0& 0& 0& 1& 1 \\
  0& 0& 0& 0& 0& 0& 1& 0& 0& 1& 1& 0& 0& 0 \\ 
  0& 0& 0& 0& 0& 1& 1& 0& 0& 1& 0& 0& 0& 0 \\
  0& 0& 1& 0& 0& 0& 0& 0& 0& 0& 0& 0& 1& 1 \\ 
  0& 0& 0& 1& 0& 0& 0& 0& 0& 0& 0& 0& 1& 1 \\
  0& 0& 0& 0& 0& 1& 1& 0& 0& 0& 1& 0& 0& 0 \\ 
  0& 0& 0& 0& 1& 0& 0& 1& 1& 0& 0& 1& 0& 0 \\
  0& 0& 0& 0& 1& 0& 0& 1& 1& 0& 0& 1& 0& 0 
\end{smallmatrix} \right)
\]

\begin{tikzcd}[column sep={2em,between origins},row sep={2em,between origins}]
9 \arrow[dddd, red] \arrow[rrrr, red] &               &                              &  & 7 \arrow[rdd] \arrow[lldd] \arrow[rrdd]  &               &                              &  & 12 \arrow[llll, red] \arrow[rrrr, red] \arrow[dddd, red] &               &                              &  & 6 \arrow[rdd] \arrow[lldd] \arrow[rrdd]  &               &                              &  & 9 \arrow[llll, red]  &               \\
                            &               &                              &  &                                          &               &                              &  &                                           &               &                              &  &                                          &               &                              &  &                             &               \\
                            & 3 \arrow[luu] & 13 \arrow[lluu] \arrow[rrdd] &  &                                          & 1 \arrow[ldd] & 14 \arrow[lldd] \arrow[rruu] &  &                                           & 4 \arrow[luu] & 13 \arrow[lluu] \arrow[rrdd] &  &                                          & 2 \arrow[ldd] & 14 \arrow[lldd] \arrow[rruu] &  &                             &  \\
                            &               &                              &  &                                          &               &                              &  &                                           &               &                              &  &                                          &               &                              &  &                             &               \\
10 \arrow[ruu] \arrow[rruu] &               &                              &  & 8 \arrow[llll, red] \arrow[rrrr, red] \arrow[uuuu, red] &               &                              &  & 11 \arrow[ruu] \arrow[lluu] \arrow[rruu]  &               &                              &  & 5 \arrow[llll, red] \arrow[rrrr, red] \arrow[uuuu, red] &               &                              &  & 10  \arrow[lluu] &              
\end{tikzcd}

\begin{Pro}
    Let $G \leq \SL_3(\Bbbk)$ be the group of type (E). Then $R \ast G$ has a $3$-preprojective cut.  
\end{Pro}

\begin{proof}
    The above is indeed a cut. First, note that the cut quiver is acyclic: The vertices $5$, $8$, $9$ and $12$ are sinks in the cut quiver, while $6$, $7$, $10$ and $11$ are sources, and it is easy to check that there are no cycles involving the remaining vertices. Next, note that every $3$-cycle is homogeneous of degree $1$: The obvious $3$-cycles in the quiver are visible as triangles. The remaining $3$-cycles all contain a vertex $13$ or $14$ and pass through the top or bottom row, so it is again easy to see also these cycles are of degree $1$.  
\end{proof}

\subsubsection{Type (F)}
We continue using the same notation as before. Let $G = \langle s,v,t,p \rangle$. Again, it is easy to compute that $H_1 = \langle s,t \rangle \simeq (C_3 \times C_3) \rtimes C_3$ is of type (C), and that 
\[ G = \langle s,t\rangle \rtimes \langle v,p \rangle \simeq H_1 \rtimes Q_8.  \]
Hence, $|G| = 216$. Also note that the group of type (E) is normal of index $2$ in the group of type (F). 

We compute the McKay quiver directly, and draw it below with its adjacency matrix, according to \Cref{Conv: Quiver drawings}. We see that the quiver can be obtained from four isomorphic copies, identified along the full subquiver with vertices $5$, $14$, $15$, $16$. As for type (E), all shortest cycles have length $3$, and we display a cut such that each $3$-cycle is homogeneous of degree $3$. 

\[ \left(
\begin{smallmatrix}
 0& 0& 0& 0& 0& 0& 0& 0& 0& 0& 0& 1& 0& 0& 0& 0 \\ 0& 0& 0& 0& 0& 1& 0& 0& 0& 0& 0& 0& 0& 0& 0& 0 \\ 
   0& 0& 0& 0& 0& 0& 0& 1& 0& 0& 0& 0& 0& 0& 0& 0 \\ 0& 0& 0& 0& 0& 0& 0& 0& 0& 1& 0& 0& 0& 0& 0& 0 \\ 
   0& 0& 0& 0& 0& 0& 0& 0& 0& 0& 0& 0& 0& 1& 0& 0 \\ 0& 0& 0& 0& 0& 0& 1& 0& 0& 0& 0& 0& 0& 0& 1& 0 \\ 
   0& 1& 0& 0& 0& 0& 0& 0& 0& 0& 0& 0& 0& 0& 0& 1 \\ 0& 0& 0& 0& 0& 0& 0& 0& 1& 0& 0& 0& 0& 0& 1& 0 \\ 
   0& 0& 1& 0& 0& 0& 0& 0& 0& 0& 0& 0& 0& 0& 0& 1 \\ 0& 0& 0& 0& 0& 0& 0& 0& 0& 0& 1& 0& 0& 0& 1& 0 \\ 
   0& 0& 0& 1& 0& 0& 0& 0& 0& 0& 0& 0& 0& 0& 0& 1 \\ 0& 0& 0& 0& 0& 0& 0& 0& 0& 0& 0& 0& 1& 0& 1& 0 \\ 
   1& 0& 0& 0& 0& 0& 0& 0& 0& 0& 0& 0& 0& 0& 0& 1 \\ 0& 0& 0& 0& 0& 0& 1& 0& 1& 0& 1& 0& 1& 0& 1& 0 \\ 
   0& 0& 0& 0& 1& 0& 0& 0& 0& 0& 0& 0& 0& 0& 0& 2 \\ 0& 0& 0& 0& 0& 1& 0& 1& 0& 1& 0& 1& 0& 2& 0& 0 
\end{smallmatrix} \right)
\]

\[
\begin{tikzcd}[sep=2ex]
                                     & 1 \arrow[ld]                         &                          &  &                                      & 2 \arrow[ld]                         &                          &  &                                      & 3 \arrow[ld]                         &                          &  &                                      & 4 \arrow[ld]                         &                          \\
12 \arrow[rr] \arrow[dd, red]             &                                      & 13 \arrow[lu, red] \arrow[ld, red] &  & 6 \arrow[rr] \arrow[dd, red]              &                                      & 7 \arrow[lu, red] \arrow[ld, red]  &  & 8 \arrow[rr] \arrow[dd, red]              &                                      & 9 \arrow[lu, red] \arrow[ld, red]  &  & 10 \arrow[rr] \arrow[dd, red]             &                                      & 11 \arrow[lu, red] \arrow[ld, red] \\
                                     & 16 \arrow[lu] \arrow[rd, Rightarrow] &                          &  &                                      & 16 \arrow[lu] \arrow[rd, Rightarrow] &                          &  &                                      & 16 \arrow[lu] \arrow[rd, Rightarrow] &                          &  &                                      & 16 \arrow[lu] \arrow[rd, Rightarrow] &                          \\
15 \arrow[ru, Rightarrow] \arrow[rd] &                                      & 14 \arrow[uu] \arrow[ll, red] &  & 15 \arrow[rd] \arrow[ru, Rightarrow] &                                      & 14 \arrow[ll, red] \arrow[uu] &  & 15 \arrow[rd] \arrow[ru, Rightarrow] &                                      & 14 \arrow[ll, , red] \arrow[uu] &  & 15 \arrow[rd] \arrow[ru, Rightarrow] &                                      & 14 \arrow[uu] \arrow[ll, red] \\
                                     & 5 \arrow[ru]                         &                          &  &                                      & 5 \arrow[ru]                         &                          &  &                                      & 5 \arrow[ru]                         &                          &  &                                      & 5 \arrow[ru]                        &                         
\end{tikzcd}
\]

\begin{Pro}
    Let $G \leq \SL_3(\Bbbk)$ be the group of type (F). Then $R \ast G$ has a $3$-preprojective cut.  
\end{Pro}

\begin{proof}
    The above is indeed a cut. First, note that the cut quiver is acyclic: All of the four cut subquivers in the picture above are acyclic. In order to pass from one of the subquivers to another, one needs to pass through the smaller subquiver shared by all of them. However, note that all the arrows going into the shared subquiver are cut, hence the whole cut quiver is indeed acyclic. Next, note that every $3$-cycle is homogeneous of degree $1$: The obvious $3$-cycles in the quiver are visible as triangles, and no extra $3$-cycles arise from the identification. 
\end{proof}

\subsubsection{Type (G)}
Let $G = \langle s,v,t,u \rangle$, and note that $p = uvu^{-1}$. From this, we immediately see that $G$ of type (G) has the group of type (F) as a normal subgroup of index $3$, and we can read off the structure as $ G =  (H_1 \rtimes Q_8) \rtimes C_3$. Hence $|G| = 648$. 

We compute the McKay quiver directly, and draw it below with its adjacency matrix, according to \Cref{Conv: Quiver drawings}. As for types (E) and (F), all shortest cycles have length $3$, and we display a cut such that each $3$-cycle is homogeneous of degree $1$. 

\[ \left(
\begin{smallmatrix}
     0& 0& 0& 0& 0& 0& 0& 0& 0& 0& 1& 0& 0& 0& 0& 0& 0& 0& 0& 0& 0& 0& 0& 0 \\ 
   0& 0& 0& 0& 0& 0& 0& 0& 0& 0& 0& 1& 0& 0& 0& 0& 0& 0& 0& 0& 0& 0& 0& 0 \\
   0& 0& 0& 0& 0& 0& 0& 0& 0& 0& 0& 0& 1& 0& 0& 0& 0& 0& 0& 0& 0& 0& 0& 0 \\ 
   0& 0& 0& 0& 0& 0& 0& 0& 0& 0& 0& 0& 0& 0& 0& 0& 0& 1& 0& 0& 0& 0& 0& 0 \\ 
   0& 0& 0& 0& 0& 0& 0& 0& 0& 0& 0& 0& 0& 0& 0& 0& 1& 0& 0& 0& 0& 0& 0& 0 \\ 
   0& 0& 0& 0& 0& 0& 0& 0& 0& 0& 0& 0& 0& 0& 0& 0& 0& 0& 1& 0& 0& 0& 0& 0 \\ 
   0& 0& 0& 0& 0& 0& 0& 0& 0& 0& 0& 0& 0& 0& 0& 0& 0& 0& 0& 0& 0& 0& 0& 1 \\ 
   0& 0& 1& 0& 0& 0& 0& 0& 0& 0& 0& 0& 0& 0& 0& 0& 0& 0& 0& 0& 0& 1& 0& 0 \\ 
   1& 0& 0& 0& 0& 0& 0& 0& 0& 0& 0& 0& 0& 0& 0& 0& 0& 0& 0& 1& 0& 0& 0& 0 \\ 
   0& 1& 0& 0& 0& 0& 0& 0& 0& 0& 0& 0& 0& 0& 0& 0& 0& 0& 0& 0& 1& 0& 0& 0 \\ 
   0& 0& 0& 0& 0& 0& 0& 0& 1& 0& 0& 0& 0& 1& 0& 0& 0& 0& 0& 0& 0& 0& 0& 0 \\ 
   0& 0& 0& 0& 0& 0& 0& 0& 0& 1& 0& 0& 0& 0& 1& 0& 0& 0& 0& 0& 0& 0& 0& 0 \\ 
   0& 0& 0& 0& 0& 0& 0& 1& 0& 0& 0& 0& 0& 0& 0& 1& 0& 0& 0& 0& 0& 0& 0& 0 \\ 
   0& 0& 0& 0& 0& 1& 0& 0& 0& 0& 0& 0& 0& 0& 0& 0& 0& 0& 0& 1& 0& 1& 0& 0 \\ 
   0& 0& 0& 0& 1& 0& 0& 0& 0& 0& 0& 0& 0& 0& 0& 0& 0& 0& 0& 1& 1& 0& 0& 0 \\ 
   0& 0& 0& 1& 0& 0& 0& 0& 0& 0& 0& 0& 0& 0& 0& 0& 0& 0& 0& 0& 1& 1& 0& 0 \\ 
   0& 0& 0& 0& 0& 0& 0& 0& 1& 0& 0& 0& 0& 0& 1& 0& 0& 0& 0& 0& 0& 0& 1& 0 \\ 
   0& 0& 0& 0& 0& 0& 0& 0& 0& 1& 0& 0& 0& 0& 0& 1& 0& 0& 0& 0& 0& 0& 1& 0 \\ 
   0& 0& 0& 0& 0& 0& 0& 1& 0& 0& 0& 0& 0& 1& 0& 0& 0& 0& 0& 0& 0& 0& 1& 0 \\ 
   0& 0& 0& 0& 0& 0& 0& 0& 0& 0& 1& 0& 0& 0& 0& 0& 1& 0& 1& 0& 0& 0& 0& 1 \\ 
   0& 0& 0& 0& 0& 0& 0& 0& 0& 0& 0& 1& 0& 0& 0& 0& 1& 1& 0& 0& 0& 0& 0& 1 \\ 
   0& 0& 0& 0& 0& 0& 0& 0& 0& 0& 0& 0& 1& 0& 0& 0& 0& 1& 1& 0& 0& 0& 0& 1 \\ 
   0& 0& 0& 0& 0& 0& 1& 0& 0& 0& 0& 0& 0& 0& 0& 0& 0& 0& 0& 1& 1& 1& 0& 0 \\ 
   0& 0& 0& 0& 0& 0& 0& 0& 0& 0& 0& 0& 0& 1& 1& 1& 0& 0& 0& 0& 0& 0& 1& 0 
\end{smallmatrix} \right)
\]

\[
\begin{tikzcd}[column sep = 3ex, row sep = 6ex, nodes in empty cells=true]
\arrow[phantom, from=3-4, to=4-4, "23"{yshift = -4pt, name=TTL}] \arrow[phantom, from=3-6, to=4-6, "23"{yshift = -4pt, name=TTR}]    \arrow[phantom, from=4-5, to=5-5, "23"{yshift = -4pt, name=TTD}]  \arrow[phantom, from=4-5, to=3-5, "7"{yshift = 4pt, name=ST}]        &                          &                                     & 1 \arrow[rd]                                                &                                     &                                     &                                                             &                          &                         &              \\
             &                          & 9 \arrow[ru, red] \arrow[rd, red]             &                                                             & 11 \arrow[ll] \arrow[rd]            &                                     & 6 \arrow[rd]                                                &                          &                         &              \\
             & 17 \arrow[ru] \arrow[rd] &                                     & 20  \arrow[ll] \arrow[rd] \arrow[ru] \arrow[rrrr, bend left] &                                     & 14 \arrow[ll, red] \arrow[ru, red] \arrow[rd, red] &                                                             & 19 \arrow[ll] \arrow[rd] &                         &              \\
5 \arrow[ru] &                          & 15 \arrow[ll, red] \arrow[ru, red] \arrow[rd, red] &                                                             & 24 \arrow[ll] \arrow[ru] \arrow[rd] &                                     & 22 \arrow[ru] \arrow[ll] \arrow[rd] \arrow[lldd, bend left] &                          & 8 \arrow[ll, red] \arrow[rd, red] &              \\
             & 12 \arrow[ru] \arrow[rd] &                                     & 21 \arrow[ru] \arrow[ll] \arrow[rd] \arrow[lluu, bend left] &                                     & 16 \arrow[ru, red] \arrow[ll, red] \arrow[rd, red] &                                                             & 13 \arrow[ru] \arrow[ll] &                         & 3 \arrow[ll] \\
2 \arrow[ru] &                          & 10 \arrow[ll, red] \arrow[ru, red]            &                                                             & 18 \arrow[ll] \arrow[ru]            &                                     & 4 \arrow[ll]                                                &                          &                         &    \arrow[from= 3-2, to = TTL] \arrow[from= 3-8, to = TTR] \arrow[from= 3-2, to = TTL] \arrow[from= 6-5, to = TTD]  \arrow[from= 4-5, to = TTL] \arrow[from= 4-5, to = TTR] \arrow[from= 4-5, to = TTD] \arrow[from= TTL, to = 3-4, red] \arrow[from= TTR, to = 4-7, red]   \arrow[from= TTD, to = 5-4, red] \arrow[from = TTL, to = ST, red] \arrow[from=ST, to=4-5]       
\end{tikzcd}
\]

\begin{Pro}
    Let $G \leq \SL_3(\Bbbk)$ be the group of type (G). Then $R \ast G$ has a $3$-preprojective cut.  
\end{Pro}

\begin{proof}
    The above is indeed a cut. To see this, one can first restrict attention to the underlying triangular grid, for which the picture clearly is a cut. Then, adding the remaining part of the quiver, note that all new triangles are of degree $1$ and that the cut quiver is acyclic, since the vertices $20$, $21$ and $22$ are sources in the cut quiver, and $23$ is a sink.  
\end{proof}

\subsection{Types (H)-(I)}
There are two simple groups in $\SL_3(\Bbbk)$, which we list now. We need the following matrices. We warn the reader that our matrix $i_2$ does not precisely match the matrix called $R$ by Yau and Yu in \cite{YauYu}. To be precise, we mean by $\sqrt{7}$ a positive real number, and we specify the root of $-1$ explicitly. This is important because the term $\frac{1}{\sqrt{-7}}$, when used in GAP, picks up an extra sign $-1$ that leads to $i_2$ having determinant $-1$. 
\begin{align*}
    &h_1 = \left( \begin{smallmatrix}
        1 & 0 & 0 \\
        0 & \epsilon_5^4 & 0 \\
        0 & 0 & \epsilon_5
    \end{smallmatrix} \right), h_2 = \left( \begin{smallmatrix}
        -1 & 0 & 0 \\
        0 & 0 & -1 \\
        0 & -1 & 0
    \end{smallmatrix} \right), 
    h_3 = \frac{1}{\sqrt{5}} \left( \begin{smallmatrix}
        1 & 1 & 1 \\
        2 & \epsilon_5^2 + \epsilon_5^3 & \epsilon_5 + \epsilon_5^4 \\
        2 & \epsilon_5 + \epsilon_5^4 & \epsilon_5^2 + \epsilon_5^3
    \end{smallmatrix} \right) \\
    &i_1 = \left( \begin{smallmatrix} \epsilon_7 & 0 & 0 \\ 0 & \epsilon_7^2 & 0 \\ 0 & 0 & \epsilon_7^4    \end{smallmatrix}  \right), i_2 = \frac{\epsilon_4}{\sqrt{7}} \left( \begin{smallmatrix} a & b & c \\ b & c & a \\ c & a & b    \end{smallmatrix}  \right)
\end{align*}
where $a = \epsilon_7^4 - \epsilon_7^3$, $b= \epsilon_7^2 - \epsilon_7^5$, $c = \epsilon_7 - \epsilon_7^6$, and $\epsilon_4 = e^{\pi i}$.  

In both cases, we will see that the quivers contain loops. To show that no cut exists, we show that some loop to its third power is a summand in the potential and use \Cref{Pro: loop in potential}. 

\subsubsection{Type (H)}
Let $G = \langle h_1, h_2, h_3 \rangle $. One can compute that $G \simeq A_5$. We draw the McKay quiver and the adjacency matrix below. Note that the quiver has three loops.
\[ \left(
\begin{smallmatrix}
    0& 0& 1& 0& 0 \\ 0& 0& 0& 1& 1 \\ 1& 0& 1& 0& 1 \\ 0& 1& 0& 1& 1 \\ 0& 1& 1& 1& 1 
\end{smallmatrix} \right)
\]

\[ \begin{tikzcd}
    1 \arrow[r, shift left=1] & 3 \arrow[r, shift left=1] \arrow[l, shift left=1] \arrow[loop, distance=2em, in=125, out=55]& 5 \arrow[l, shift left=1] \arrow[r, shift left = 1] \arrow[dr] \arrow[loop, distance=2em, in=125, out=55] & 4 \arrow[l, shift left= 1] \arrow[d, shift left = 1] \arrow[loop, distance=2em, in=125, out=55] \\
     & & & 2 \arrow[u, shift left = 1] \arrow[ul, shift left  = 2] 
\end{tikzcd} \]

\begin{Pro}
    Let $G \leq \SL_3(\Bbbk)$ be the group of type (H). Then $R \ast G$ does not admit a $3$-preprojective cut. 
\end{Pro}

\begin{proof}
    We show that a power of the loop $l$ at vertex $3$ appears in the potential. Note that in our labeling convention, vertex $3$ corresponds to the defining representation $ \rho \colon G \hookrightarrow \SL_3(\Bbbk)$. The irreducible representation labeled $2$ is Galois-conjugate to the one labeled $3$, so these may be interchanged. 
    We follow \cite[Section 3]{BSW} to compute the coefficient of $l^3$ in the potential. We interpret $l$ as a homomorphism $\rho \to \rho \otimes \rho$, and write it as a matrix $L$. We denote by $I_3$ the identity matrix of size $3$, and by $\alpha$ the antisymmetrizer map $V^{\otimes 3} \to \bigwedge^3 V $ for $V = \Bbbk^3$. We need to compute 
    \[ ( I_3 \otimes \alpha ) \cdot (L \otimes I_3 \otimes I_3) \cdot (L \otimes I_3) \cdot L  \]
    to obtain an endomorphism of the representation $\rho$. This is then a scalar multiple of $\operatorname{id}_\rho$, and this scalar is the one in front of $l^3$ in the potential. In particular, it suffices to check that this is non-zero. Furthermore, since $\dim_\Bbbk \Hom_G(\rho, \rho \otimes \rho) = 1$, the matrix $L$ is uniquely determined up to base-change and scaling. For the basis we have chosen such that $G$ acts by the matrices $h_i$ and the corresponding basis of $\rho \otimes \rho$, a matrix $L$ is given by 
    \[ L = \left( \begin{smallmatrix}
        0 & 0 & 0 \\ 0 & 1 & 0 \\ 0 & 0 & -1 \\ 0 & -1 & 0 \\ 0 & 0 & 0 \\ \frac{1}{2} & 0 & 0 \\ 0 & 0 & 1 \\ -\frac{1}{2} & 0 & 0 \\ 0 & 0 & 0
    \end{smallmatrix} \right), \]
    and it is easy to check that the relevant scalar is $1$. Thus it follows from \Cref{Pro: loop in potential} that no cut exists. 
\end{proof}

\subsubsection{Type (I)}
Let $G = \langle i_1, i_2, t \rangle$. One can compute that $G \simeq \operatorname{PSL}_3(\mathbb{F}_2) = \SL_3(\mathbb{F}_2) = \GL_3(\mathbb{F}_2) \simeq \operatorname{PSL}_2(\mathbb{F}_7)$. We draw the McKay quiver and the adjacency matrix below. Note that the quiver has two loops. 

\[ \left(
\begin{smallmatrix}
     0& 0& 1& 0& 0& 0 \\ 1& 0& 0& 0& 0& 1 \\ 0& 1& 0& 1& 0& 0 \\ 0& 1& 0& 0& 1& 1 \\ 0& 0& 0& 1& 1& 1 \\ 0& 0& 1& 1& 1& 1
\end{smallmatrix} \right)
\]

\[ \begin{tikzcd}
1 \arrow[rd] & 2 \arrow[r] \arrow[l] & 6 \arrow[ld] \arrow[d, shift left = 1] \arrow[r, shift left = 1] \arrow[loop, distance=2em, in=125, out=55] & 5 \arrow[l, shift left = 1] \arrow[ld, shift left = 1] \arrow[loop, distance=2em, in=125, out=55] \\
             & 3 \arrow[u] \arrow[r] & 4 \arrow[lu] \arrow[u, shift left = 1] \arrow[ru, shift left = 1]                                           &                                                                  
\end{tikzcd} \]

\begin{Pro}
    Let $G \leq \SL_3(\Bbbk)$ be the group of type (I). Then $R \ast G$ does not admit a $3$-preprojective cut. 
\end{Pro}

\begin{proof}
    As in type (H), one can write a matrix for the loop $l$ at vertex $6$ and then compute the coefficient at $l^3$ in the potential to see that it is non-zero. Thus it follows from \Cref{Pro: loop in potential} that no cut exists.   
\end{proof}

\subsection{Types (J)-(L)}
There are three exceptional subgroups of $\SL_3(\Bbbk)$ which were found as the primitive groups with a normal subgroup so that the defining representation restricts to a decomposable one of the normal subgroup. The quotient by this normal subgroup is related to the groups of type (H) and (I). We need one more generator, following again Yau and Yu \cite{YauYu}, denoted by 
\begin{align*}
    w = \frac{1}{\sqrt{5}} \left( \begin{smallmatrix}
        1 & a_1 & a_1 \\ 2a_2  & \epsilon_5^2+\epsilon_5^3 & \epsilon_5+\epsilon_5^4 \\ 2a_2 & \epsilon_5+\epsilon_5^4 & \epsilon_5^2+\epsilon_5^3 
    \end{smallmatrix}  \right),
\end{align*}
where $a_1 = \frac{1}{4}(-1 + \sqrt{-15})$ and $a_2 = \overline{a_1}$ the complex conjugate. 

\subsubsection{Type (J)}
Let $G = \langle h_1, h_2, h_3, \epsilon_3 I_3 \rangle$. One can compute that $G \simeq \GL_2(\mathbb{F}_4)$, and that $\langle \epsilon_3 I_3 \rangle$ is normal in $G$ with quotient $G/\langle \epsilon_3 I_3 \rangle \simeq A_5 $, which is the group of type (H). It follows that $|G| = 180$. We draw the McKay quiver and its adjacency matrix below, according to \Cref{Conv: Quiver drawings}. 

\[ \left(
\begin{smallmatrix}
        0& 0& 0& 0& 0& 0& 0& 1& 0& 0& 0& 0& 0& 0& 0 \\ 0& 0& 0& 0& 0& 0& 0& 0& 1& 0& 0& 0& 0& 0& 0 \\ 
    0& 0& 0& 0& 1& 0& 0& 0& 0& 0& 0& 0& 0& 0& 0 \\ 0& 0& 0& 0& 0& 0& 0& 0& 0& 0& 1& 0& 0& 1& 0 \\ 
    0& 1& 0& 0& 0& 0& 0& 1& 0& 0& 0& 0& 0& 1& 0 \\ 0& 0& 0& 0& 0& 0& 0& 0& 0& 0& 0& 1& 0& 0& 1 \\ 
    0& 0& 0& 0& 0& 0& 0& 0& 0& 1& 0& 0& 1& 0& 0 \\ 0& 0& 1& 0& 0& 0& 0& 0& 1& 0& 0& 0& 0& 0& 1 \\ 
    1& 0& 0& 0& 1& 0& 0& 0& 0& 0& 0& 0& 1& 0& 0 \\ 0& 0& 0& 0& 0& 1& 0& 0& 0& 0& 1& 0& 0& 1& 0 \\ 
    0& 0& 0& 0& 0& 0& 1& 0& 0& 0& 0& 1& 0& 0& 1 \\ 0& 0& 0& 1& 0& 0& 0& 0& 0& 1& 0& 0& 1& 0& 0 \\ 
    0& 0& 0& 0& 0& 1& 0& 1& 0& 0& 1& 0& 0& 1& 0 \\ 0& 0& 0& 0& 0& 0& 1& 0& 1& 0& 0& 1& 0& 0& 1 \\ 
    0& 0& 0& 1& 1& 0& 0& 0& 0& 1& 0& 0& 1& 0& 0 
\end{smallmatrix} \right)
\]

\[ 
\begin{tikzcd}[column sep=1ex]
2 \arrow[rd, red]                         &                         & 1 \arrow[rd]                                    &                                      & 3 \arrow[rd]                                    &                                      & 2 \arrow[rd, red]                                    &                           \\
                                     & 9 \arrow[ru] \arrow[rd] &                                                 & 8 \arrow[ru, red] \arrow[rd, red] \arrow[ll, red]   &                                                 & 5 \arrow[ru] \arrow[rd] \arrow[ll]   &                                                 & 9 \arrow[ll]              \\
14 \arrow[ru, red] \arrow[rd, red] \arrow[rdd, red] &                         & 13 \arrow[ru] \arrow[ll] \arrow[rd] \arrow[rdd] &                                      & 15 \arrow[ru] \arrow[ll] \arrow[rd] \arrow[rdd] &                                      & 14 \arrow[ru, red] \arrow[ll, red] \arrow[rd, red] \arrow[rdd, red] &                           \\
                                     & 12 \arrow[ru]           &                                                 & 11 \arrow[ru, red] \arrow[ll, red] \arrow[lld, red] &                                                 & 10 \arrow[ru] \arrow[ll] \arrow[lld] &                                                 & 12 \arrow[ll] \arrow[lld] \\
                                     & 7 \arrow[ruu]           &                                                 & 6 \arrow[ruu, red] \arrow[llu, red]            &                                                 & 4 \arrow[ruu] \arrow[llu]            &                                                 & 7 \arrow[llu]            
\end{tikzcd}
\]
\begin{Pro}
    Let $G \leq \SL_3(\Bbbk)$ be the group of type (J). Then $R \ast G$ has a $3$-preprojective cut.  
\end{Pro}

\begin{proof}
    The above is indeed a cut. Note that every $3$-cycle is homogeneous of degree $1$: The obvious $3$-cycles in the quiver are visible as triangles, and the remaining $3$-cycles arise from the gluing along the vertical sides, i.e. by following the horizontal arrows, and clearly those are of degree $1$ as well. Next, note that the cut quiver is acyclic: A cycle can not involve a source or a sink, but any vertex that is not a source or sink is connected only to sources and sinks. 
\end{proof}

\subsubsection{Type (K)}
Let $G = \langle i_1, i_2, t, \epsilon_3 I_3 \rangle $. One can compute that $G \simeq \operatorname{PSL}_3(\mathbb{F}_2) \times C_3$, from which it follows that $|G| = 504$. Clearly, the group of type (I) is a subgroup. We draw the McKay quiver and its adjacency matrix below, according to \Cref{Conv: Quiver drawings}. 

\[ 
\left(
\begin{smallmatrix}
        0&0&0&0&0&0&1&0&0&0&0&0&0&0&0&0&0&0 \\  0&0&0&0&0&0&0&0&1&0&0&0&0&0&0&0&0&0 \\
    0&0&0&0&1&0&0&0&0&0&0&0&0&0&0&0&0&0 \\  0&1&0&0&0&0&0&0&0&0&0&0&0&0&0&0&1&0 \\
    0&0&0&0&0&1&0&0&0&0&1&0&0&0&0&0&0&0 \\  0&0&1&0&0&0&0&0&0&0&0&0&0&0&0&0&0&1 \\
    0&0&0&0&0&0&0&1&0&0&0&1&0&0&0&0&0&0 \\  1&0&0&0&0&0&0&0&0&0&0&0&0&0&0&1&0&0 \\
    0&0&0&1&0&0&0&0&0&1&0&0&0&0&0&0&0&0 \\  0&0&0&0&0&1&0&0&0&0&0&0&0&1&0&0&1&0 \\
    0&0&0&0&0&0&0&1&0&0&0&0&0&0&1&0&0&1 \\  0&0&0&1&0&0&0&0&0&0&0&0&1&0&0&1&0&0 \\
    0&0&0&0&0&0&0&0&0&0&1&0&0&1&0&0&1&0 \\  0&0&0&0&0&0&0&0&0&0&0&1&0&0&1&0&0&1 \\
    0&0&0&0&0&0&0&0&0&1&0&0&1&0&0&1&0&0 \\  0&0&0&0&0&0&1&0&0&0&1&0&0&1&0&0&1&0 \\
    0&0&0&0&0&0&0&0&1&0&0&1&0&0&1&0&0&1 \\  0&0&0&0&1&0&0&0&0&1&0&0&1&0&0&1&0&0
\end{smallmatrix}\right)
\]

\[
\begin{tikzcd}[row sep = 8ex]
2 \arrow[rr,red] &                         & 9 \arrow[rr] \arrow[ld]  &                                     & |[alias = Ten]|10 \arrow[dl, bend right] \arrow[rr] \arrow[ld, phantom, "15"{name = A}] \arrow[<-, to path={-- ([xshift=2ex, yshift=2ex]A)}]           &                                     & 6 \arrow[rr,red] \arrow[ld,red]  &                         & 3 \arrow[ld] \\
             & 4 \arrow[rr] \arrow[lu] &                          & 17 \arrow[to path={-- ([xshift=-2ex, yshift=-2ex]A)},red] \arrow[rr, bend right,red]  \arrow[lu,red] \arrow[ld, bend right,red]  & |[alias = T]| 13 \arrow[l] \arrow[ to path={(A) -- (\tikztotarget)}]                                       &|[alias = E]| 18 \arrow[l] \arrow[lu, phantom, "14"{name = B}] \arrow[to path={ (B) -- (E) },red] \arrow[to path={(Ten) -- (B)}] \arrow[rr] \arrow[lu, bend right] \arrow[ld, bend right, crossing over, crossing over clearance = 0.5ex] \arrow[->, to path={ ([xshift=-2ex]B) -- ([xshift = 2ex] A) },red] \arrow[<-,  to path={ (B) -- (T) } ] &                          & 5 \arrow[lu] \arrow[ld] &              \\
             &                         & 12 \arrow[rr, bend right] \arrow[lu] \arrow[ur, phantom, "13"{name = ThirteenL}] \arrow[ur, to path={ (ThirteenL) -- (\tikztotarget)}]  \arrow[to path={ -- (ThirteenL)}] &   |[alias = FourteenL]| 14 \arrow[l,red] \arrow[<-, r]               & 16 \arrow[ul, phantom, "15"{name = FifteenL}] \arrow[<-,  to path={ -- (FifteenL)}] \arrow[ul, <-, to path={(FifteenL) -- (\tikztotarget)},red] \arrow[rr, bend right] \arrow[lu, bend right, crossing over, crossing over clearance=0.5ex] \arrow[ld] \arrow[to path={(ThirteenL) -- (FourteenL)}] \arrow[to path={(FourteenL) -- (FifteenL)},red] \arrow[to path={(FifteenL) -- (ThirteenL)}]  \arrow[ur, phantom, "14"{name = FourteenR}] \arrow[to path= {-- (FourteenR) } ] \arrow[ur, to path={(FourteenR) -- (\tikztotarget)},red]    &  |[alias = FifteenR]|  15 \arrow[l] \arrow[<-,r,red]  \arrow[to path={(FourteenR) -- (FifteenR)},red]                           & 11 \arrow[ul, phantom, "13"{name=ThirteenR}] \arrow[lu, bend right,red] \arrow[ld,red] \arrow[<-, to path={ -- (ThirteenR)}]  \arrow[ul, to path={ (\tikztotarget) -- (ThirteenR)}] &  \arrow[to path={(FifteenR) -- (ThirteenR)}] \arrow[to path={(ThirteenR) -- (FourteenR)}]                      &              \\
             &                         &                          & 7 \arrow[rr,red] \arrow[lu,red]             &                                     & 8 \arrow[lu] \arrow[ld]             &                          &                         &              \\
             &                         &                          &                                     & 1 \arrow[lu]                        &                                     &                          &                         &             
\end{tikzcd}
\]
We also note that the group $C_3$ acts on this quiver by rotation, and taking the quotient quiver we obtain the quiver for the group of type (I). In fact, taking the skew-group algebra we obtain $\Bbbk Q/I \ast C_3 \simeq_M R \ast H$, where $H$ is the group of type (I), and this is the same ``unskewing'' as we used in \cite{DramburgGasanova3}, see also \cite{ReitenRiedtmann}. This also explains why there exists no cut that is invariant under this symmetry, since otherwise it would give rise to a cut for type (I).  

\begin{Pro}
    Let $G \leq \SL_3(\Bbbk)$ be the group of type (K). Then $R \ast G$ has a $3$-preprojective cut.  
\end{Pro}

\begin{proof}
    The above is indeed a cut. Note that every $3$-cycle is homogeneous of degree $1$: All $3$-cycles are clearly visible as triangles, with possibly bent arrows, and all are of degree $1$. Note that the cut quiver is acyclic: A cycle can not involve a source or a sink, but any vertex that is not a source or sink is connected only to sources and sinks. 
\end{proof}

\subsubsection{Type (L)}
Let $G = \langle h_1, h_2, h_3, \epsilon_3 I_3, w \rangle $. One can compute that $\langle \epsilon_3 I_3 \rangle$ is normal in $G$ and that $G/ \langle \epsilon_3 I_3 \rangle \simeq A_6 $. It follows that $|G| = 1080$. We draw the McKay quiver and its adjacency matrix below, according to \Cref{Conv: Quiver drawings}.

\[ 
\left(
\begin{smallmatrix}
       0& 0& 1& 0& 0& 0& 0& 0& 0& 0& 0& 0& 0& 0& 0& 0& 0 \\    0& 0& 0& 0& 0& 0& 0& 0& 0& 0& 0& 0& 0& 1& 0& 0& 0 \\  
    0& 0& 0& 0& 1& 0& 0& 0& 1& 0& 0& 0& 0& 0& 0& 0& 0 \\    0& 0& 0& 0& 0& 0& 0& 0& 0& 0& 0& 1& 0& 0& 0& 0& 0 \\  
    1& 0& 0& 0& 0& 0& 0& 0& 0& 0& 1& 0& 0& 0& 0& 0& 0 \\    0& 0& 0& 0& 0& 0& 0& 0& 0& 0& 0& 0& 0& 0& 0& 1& 0 \\  
    0& 0& 0& 0& 0& 0& 0& 0& 0& 0& 0& 0& 0& 0& 0& 1& 0 \\    0& 0& 0& 0& 1& 0& 0& 0& 0& 0& 0& 0& 0& 0& 0& 0& 1 \\  
    0& 0& 0& 0& 0& 0& 0& 0& 0& 0& 1& 0& 0& 0& 1& 0& 0 \\    0& 0& 0& 0& 0& 0& 0& 0& 0& 0& 0& 0& 1& 0& 0& 1& 0 \\  
    0& 0& 1& 0& 0& 0& 0& 1& 0& 0& 0& 0& 0& 0& 0& 1& 0 \\    0& 1& 0& 0& 0& 0& 0& 0& 0& 0& 0& 0& 1& 0& 0& 1& 0 \\  
    0& 0& 0& 1& 0& 0& 0& 0& 0& 0& 0& 0& 0& 1& 0& 0& 1 \\    0& 0& 0& 0& 0& 0& 0& 0& 0& 1& 0& 1& 0& 0& 1& 0& 0 \\  
    0& 0& 0& 0& 0& 0& 0& 1& 0& 0& 0& 0& 1& 0& 0& 1& 0 \\    0& 0& 0& 0& 0& 0& 0& 0& 1& 0& 0& 0& 0& 1& 0& 0& 2 \\  
    0& 0& 0& 0& 0& 1& 1& 0& 0& 1& 1& 1& 0& 0& 1& 0& 0 
\end{smallmatrix}\right)
\]

\[
\begin{tikzcd}[column sep=1ex, execute at end picture={ \draw[commutative diagrams/Rightarrow, red] plot [smooth, tension=1] coordinates { ([xshift=10pt] S) (NR) (T) (EL) ([xshift=-10pt] E) }; }]
                                                                                              &                         & 1 \arrow[rd] \arrow[loop, phantom, distance=2em, in=125, out=55, ""{coordinate, name=T}]                                                      &                         &                                                                  \\
                                                                                              & 5 \arrow[rd] \arrow[ru] &                                                                     & 3 \arrow[rd, red] \arrow[ll, red] &                                                                  \\
8 \arrow[dd, red] \arrow[ru, red] \arrow[loop, phantom, distance=2em, in=215, out=145, ""{coordinate, name=EL}]                                                                      &                         & 11 \arrow[rrdd] \arrow[ll] \arrow[ru]                               &                         & 9 \arrow[lldd] \arrow[ll] \arrow[loop, phantom, distance=2em, in=35, out=325, ""{coordinate, name=NR}]                                       \\
                                                                                              &                         & 6 \arrow[rrd]  \arrow[d, phantom, ""{coordinate, name=Z}]                                                     &                         &                                                                  \\
|[alias=E]| 17 \arrow[rr] \arrow[rruu] \arrow[rrdd] \arrow[rrdddd, bend right=60] \arrow[rru] \arrow[rrd] &                         & 15 \arrow[rr] \arrow[lluu] \arrow[lldd]                             &                         & |[alias=S]| 16 \arrow[uu, red] \arrow[dd, red]  \\
                                                                                              &                         & 7 \arrow[rru]                                                       &                         &                                                                  \\
13 \arrow[uu, red] \arrow[rrrr, bend right, red] \arrow[rd, red]                                             &                         & 10 \arrow[rruu] \arrow[ll]                                          &                         & 14 \arrow[lluu] \arrow[ll] \arrow[lldd, bend left]               \\
                                                                                              & 4 \arrow[rd]            &                                                                     & 2 \arrow[ru, red]            &                                                                  \\
                                                                                              &                         & 12 \arrow[ru] \arrow[lluu, bend left] \arrow[rruuuu, bend right=60] &                         &                                                                 
\end{tikzcd} \]

\begin{Pro}
    Let $G \leq \SL_3(\Bbbk)$ be the group of type (L). Then $R \ast G$ has a $3$-preprojective cut.  
\end{Pro}

\begin{proof}
    The above is indeed a cut. To see that the cut quiver is acyclic, note that the vertices $4,5,9,14,17$ are sources in the cut quiver and $2,3,8,13,16$ are sinks. The remaining vertices are connected only to sources or sinks. To see that all $3$-cycles are of degree $1$, first note that this is clear for any cycle not containing $16$ or $17$. For the remaining $3$-cycles, it suffices to note that $16$ is a sink and $17$ is a source in the cut quiver, and that any cycle involving exactly one of them passes through one of the vertices $6$, $7$, $10$, $11$ or $15$, all of which are not adjacent to a cut arrow. Finally, the cycles containing both $16$ and $17$ can be checked by hand.
\end{proof}

We summarise our computations in the following theorem. 

\begin{Theo}\label{Theo: exceptional class}
    Let $G \leq \SL_3(\Bbbk)$ be an exceptional subgroup. Then $R \ast G \simeq \Bbbk Q_G/I$ admits a $3$-preprojective cut if and only if $G$ is not of type (H) or (I).  
\end{Theo}

\subsection{Computing McKay quivers and cuts}
We give a brief account how to compute a McKay quiver and cuts by hand or with a standard computer algebra system. Note that this discussion is not restricted to dimension $3$.  

\begin{Rem}
    Recall that the McKay quiver $Q = Q_G$ is defined by having as vertices the irreducible representations $\Irr(G)$. The number of arrows from vertex $i$ to $j$ is given by 
    \[ \dim \Hom_{\Bbbk G}(S_i, S_j \otimes V), \]
    where $V$ is the defining representation of $G$, and $S_i$ and $S_j$ are the irreducible representations corresponding to $i$ and $j$. 
    Since the above does not need knowledge of the explicit homomorphisms, we can instead rephrase it in terms of characters. We label the vertices by the irreducible characters of $G$, and compute the number of arrows from the character $\chi_i$ to $\chi_j$ by the standard inner product of characters
    \[ (\chi_i, \chi_j \cdot \chi_V). \]
    In GAP, all of the necessary functions are implemented. In particular, our labeling of all the vertices in all the quivers is consistent with the labeling of the irreducible characters of $G$ as computed by GAP. 
\end{Rem}

Next, we explain how to compute a cut. It was pointed out to us by Stephan Wagner that this is an instance of an exact cover problem, which is known to be NP-complete in general. Note that when $d=3$, the restricted problem is an exact cover by $3$-sets, which is still NP-complete \cite{CompAndIntr}. In the case where the group $G$ is abelian, our results from \cite{DramburgGasanova2} can be used to compute cuts in a different way. However, we are not aware of any additional structure underlying the case of cutting McKay quivers for arbitrary $G \leq \SL_d(\Bbbk)$ that reduces the complexity of this task. 

\begin{Rem}
    Recall that a cut needs to satisfy the condition that every cycle in the support of the potential $\omega$ is homogeneous of degree $1$. Denote by $S = \supp(\omega)$ the set of cycles appearing in $\omega$, and for each arrow $a \in Q_1$ consider the set
    \[ t_a = \{ c \in S \mid a \text{ appears in } c\}. \]
    Let $ T = \{t_a \mid a \in Q_1\}$. Then $T$ is a set of subsets of $S$. Cutting an arrow $a$ means cutting all cycles in $t_a$. Now in order to obtain a cut, all cycles in $S$ need to be cut exactly once, which means our cut $C$ needs to produce a partition 
    \[ \{ t_a \mid a \in C \} \]
    of the set $S$. This is called an exact cover problem: Given a set of subsets $T$ of a set $S$, decide whether there exists a selection of elements of $T$ that form a partition of $S$. 
\end{Rem}

Solving an exact cover problem can be done by standard algorithms like Knuth's Algorithm X. After finding a solution, it remains to check that the cut quiver defines a finite dimensional algebra. Let us note two simplifications of the above approach, that essentially underlies \Cref{Pro: All cycles cut and locally finite suffices}.

\begin{Rem}
    Computing the potential $\omega$ or its support is a nontrivial task since it requires a computation of bases for $\Hom_{\Bbbk G}(S_i , S_j \otimes V)$ and compositions of the basis vectors. This can be done in GAP, but is computationally more expensive than computing inner products of characters. 
    Furthermore, computing the algebra defined by the cut quiver and computing its dimension can be done in GAP (using QPA) but again is computationally expensive. 
    However, it was asked in \cite[Question 5.9]{HIO} when the quivers of $(d-1)$-hereditary algebras are acyclic. This is not always the case, see \cite{TomonagaSilting}, but we expect it to hold in the case of $(d-1)$-hereditary algebras arising from skew-group algebras. Thus, we restrict our attention to cuts such that the resulting quiver is acyclic. This simplifies both problems above. Instead of computing $S = \supp(\omega)$, simply take $S = \{ c \mid c \text{ is a cycle of length } d \} \supseteq \supp(\omega)$, which can be done purely combinatorially by standard graph algorithms. Finding an exact cover for this set as before then assures us that $\omega$ will be homogeneous of degree $1$. If $d=3$ and $Q_G$ has no loops or $2$-cycles, then the cut quiver is acyclic and hence generates a finite dimensional algebra. We note that the existence of loops or $2$-cycles is a phenomenon of ``small'' groups and that they do not exist ``generically''. For higher values of $d$, the resulting cut quiver could contain cycles of a length strictly smaller than $d$, but again finding cycles is a standard task for existing graph algorithms.   
\end{Rem} 

The resulting workflow for a given group looks as follows. 
\begin{enumerate}
    \item Compute $G$ from explicit matrix generators. 
    \item Compute $\Irr(G)$ and the defining character $\chi_V$. 
    \item Compute $(\chi_i, \chi_j \cdot \chi_V)$ for all $i, j \in Q_0$ to obtain the quiver $Q_G$. 
    \item Compute the set $S$ of all $3$-cycles in $Q_G$ and the set $T = \{ t_a \mid a \in Q_1\}$. 
    \item Compute one (or all) solutions to the exact cover problem $(S, T)$. 
\end{enumerate}

In order to check any of the results in this section, the following steps suffice. 
\begin{enumerate}
    \item Compute $G$ from explicit matrix generators. 
    \item Compute $\Irr(G)$ and the defining character $\chi_V$. 
    \item Compute $(\chi_i, \chi_j \cdot \chi_V)$ for all $i, j \in Q_0$ to obtain the quiver $Q_G$. 
    \item Compute the set $S$ of all $3$-cycles in $Q_G$. 
    \item For the given cut, check that every element in $S$ has exactly one cut arrow. 
    \item For the given cut, check that the cut quiver is acyclic.
\end{enumerate}

\section{The classification of $3$-preprojective algebras \texorpdfstring{$R \ast G$}{R * G}} \label{Sec: Classification}
We now summarise how the results of this article, together with \cite{DramburgGasanova, DramburgGasanova3}, give rise to a classification of $3$-preprojective gradings on skew-group algebras given by cuts. We provide several different perspectives, and explain why a uniform description seems difficult. Recall that the types of finite subgroups of $\SL_3(\Bbbk)$ are labeled (A) to (L), where (E) to (L) are the exceptional subgroups. 

If $G \leq \SL_3(\Bbbk)$ is of type (A), then \Cref{Theo: SL3 type (A) classification} tells us that $R \ast G$ admits a $3$-preprojective structure if and only if the divisibility condition from the matrix expression in \Cref{Theo: SL3 type (A) classification} admits a strictly positive solution. Furthermore, we know from \cite{DramburgGasanova} that such a solution exists if and only if $G \not \simeq C_2 \times C_2$ and if $G \hookrightarrow \SL_3(\Bbbk)$ does not factor through $\SL_2(\Bbbk)$. In particular, all quivers $Q_G$ for abelian $G$ where $R \ast G$ is not higher preprojective have a loop, \emph{except for} $G = C_2 \times C_2$. 

In type (B), we saw in this article that the only obstruction to $R \ast G$ admitting a $3$-preprojective cut is the existence of determinant loops in $Q_G$. This is equivalent to the defining representation being self-dual, and $G$ not being a binary group. 

In types (C) and (D), it was proven in \cite{DramburgGasanova3} that $R \ast G$ admits a $3$-preprojective structure if and only if $9 \mid |G|$. Interestingly, in all negative cases this is again witnessed by a loop in $Q_G$. 

For the exceptional groups, we saw in this article that all types except type (H) and (I) admit a $3$-preprojective cut. These two groups are the two simple exceptional groups, and again their quivers $Q_G$ contain a loop. 

We want to briefly comment on the prevalence of loops in the negative cases we found. The nature of the loops in type (B) is quite different from those in type (C) and (D) and the exceptional cases. In type (B), it is easy to see that the determinant loops to their third powers do not appear in the potential for $R \ast G$. In contrast, for types (C), (D), (H) and (I) some loop does appear to its third power in the potential, which can also be used to show the non-existence of cuts. It would be interesting to find a general proof that shows that loops for skew-group algebras mean that no preprojective structure can exists, as this would give a better understanding of when cycles in $n$-representation infinite algebras can or can not appear and help put the example of Tomonaga \cite{TomonagaSilting} into perspective. One might for example expect that it is the Koszul property for skew-group algebras that makes loops impossible for higher preprojective structures.  

We return our classification and summarise it as follows. 

\begin{Theo}
    Let $G \leq \SL_3(\Bbbk)$ be a finite subgroup. Then $R \ast G$ admits a $3$-preprojective structure if and only if $Q_G$ has no loops, and $G \not \simeq C_2 \times C_2$. 
\end{Theo}

However, the above phrasing seems unsatisfying since the case $G \simeq C_2 \times C_2$ seems to have no clear relationship to loops. Furthermore, we note that the superpotential $\omega$ for $\Bbbk Q_G/I$ is supported on cycles of length $3$, so if the cycles factor into smaller cycles, one of them is necessarily a loop. We therefore do not expect the perspective on loops to generalise to higher dimensions in a straightforward way, since for example finite subgroups of $\SL_2(\Bbbk) \times \SL_2(\Bbbk)$ give rise to $4$-Calabi-Yau algebras which are not $4$-preprojective and still contain no loops. 

Instead, one might take a more geometric perspective. The invariant ring $R^G$ is the center of $R \ast G$, so a grading on $R \ast G$ induces one on $R^G$. Such a grading in turn defines a $\Bbbk^\ast$-action on $X = \operatorname{Spec}(R^G) = \Bbbk^3 /G$. One can therefore ask how to characterise those $\Bbbk^\ast$-actions that necessarily come from higher preprojective cuts. In type (A), we showed in \cite{DramburgGasanova2} that these $\Bbbk^\ast$-actions can be identified with certain crepant divisors for $X$. Furthermore, in types (C) and (D), we produce in \cite{DramburgGasanova3} $3$-preprojective gradings on $R \ast G$ from those on $R \ast N$, where $N \leq G$ is an abelian normal subgroup and $G \simeq N \rtimes H$. As a byproduct, we obtain that the existence of a $3$-preprojective grading on $R \ast G$ is witnessed by a crepant divisor for $\operatorname{Spec}(R^N)$, on which $H$ acts such that $X = \operatorname{Spec}(R^N)/H$. 
We therefore believe that the existence of certain divisors in (a resolution of) $X$ or a closely related variety is equivalent to the existence of a $3$-preprojective structure on $R \ast G$. However, we have no apriori argument for why this should be expected. 

Since neither perspective is fully satisfactory, we finish by embracing the diversity of the subgroups of $\SL_3(\Bbbk)$, and simply state the classification according to the different types. 

\begin{Theo}\label{Theo: All classification}
    Let $G \leq \SL_3(\Bbbk)$ be a finite group. Then $R \ast G \simeq_M \Bbbk Q_G/I$ admits a $3$-preprojective structure if and only if the group $G$ satisfies one of the following conditions:  
    \begin{itemize}
        \item $G$ is of type (A), $G \not \simeq C_2 \times C_2$, and its embedding does not factor through $\SL_2(\Bbbk)$. 
        \item $G$ is of type (B), its defining representation is not self-dual, and it is not isomorphic to a finite subgroup of $\SL_2(\Bbbk)$.
        \item $G$ is of type (C) or (D) and has order divisible by $9$.
        \item $G$ is an exceptional subgroup and not of type (H) or type (I). 
    \end{itemize}
\end{Theo}
 
\section*{Acknowledgements}
We thank Martin Herschend for many helpful discussions and comments. We thank Oleksandra Gasanova for many helpful discussions, and in particular for suggesting using type (A) to produce cuts for type (B), and for help with drawing the quivers for the exceptional subgroups. We also thank Robert Moscrop and Shani Meynet for explaining their work in \cite{McKayDecomp} to us, as well as pointing out the notion of isoclinism. We are grateful to Álvaro Nolla de Celis for answering a question regarding his PhD thesis. We thank to Stephan Wagner for pointing us to the term ``exact cover problem'' that otherwise would have been impossible to find. The work in \Cref{Sec: Except} was greatly aided by computations made in GAP \cite{GAP4}, and to a lesser extent in Julia \cite{Julia-2017}, using the package OSCAR \cite{OSCAR}. 

\printbibliography

\end{document}